\documentclass{amsart} 
\usepackage{bbm}
\usepackage{amsfonts}
\usepackage{amsmath,amsthm,amscd,mathrsfs,amssymb,graphicx,enumerate,
	dsfont}
\usepackage{hyperref}
\usepackage[usenames,dvipsnames]{xcolor}
\usepackage{xypic}
\hypersetup{colorlinks=true,citecolor=NavyBlue,linkcolor=Brown,urlcolor=Orange}

\newtheorem*{thm3}{Theorem}
\newtheorem*{conj3}{Conjecture}
\newtheorem{thm2}{Theorem}
\newtheorem{conj2}{Conjecture}
\newtheorem{cor2}{Corollary}
\newtheorem{thm}{Theorem}[section]
\newtheorem{cor}[thm]{Corollary}
\newtheorem{lem}[thm]{Lemma}
\newtheorem{prop}[thm]{Proposition}
\newtheorem{conj}[thm]{Conjecture}
\theoremstyle{definition}
\newtheorem{defn}[thm]{Definition}

\newtheorem{ex}[thm]{Example}
\newtheorem{rmk}[thm]{Remark}

\newtheorem{ques}[thm]{Question}

 \DeclareMathOperator{\Spec}{Spec}
 
\DeclareMathOperator{\id}{id}
 
\DeclareMathOperator{\End}{End} 
\DeclareMathOperator{\Hom}{Hom} 
 \DeclareMathOperator{\rk}{rk}

\DeclareMathOperator{\DHom}{DHom}

\newcommand{\Z}{\ensuremath\mathds{Z}}
\newcommand{\N}{\ensuremath\mathds{N}}
\newcommand{\Q}{\ensuremath\mathds{Q}}

\newcommand{\fS}{\ensuremath\mathfrak{S}}

\newcommand{\h}{\ensuremath\mathfrak{h}}

\newcommand{\PP}{\ensuremath\mathds{P}}

\newcommand{\calO}{\ensuremath\mathcal{O}}

\newcommand{\calN}{\ensuremath\mathscr{N}}

\newcommand{\CH}{\ensuremath\mathrm{CH}}
\newcommand{\CHM}{\ensuremath\mathrm{CHM}}
\newcommand{\NumM}{\ensuremath\mathrm{NumM}}
\newcommand{\DCH}{\ensuremath\mathrm{DCH}}

\renewcommand{\bar}{\overline}
\newcommand{\inj}{\hookrightarrow}
\newcommand{\surj}{\twoheadrightarrow}
\newcommand{\lra}{\xrightarrow}

\newcommand{\cart}{\ar@{}[dr]|\square} % cartesian diagrams, write it after the
\newcommand{\isom}{\simeq} %isomorphism
\renewcommand{\tilde}{\widetilde}
\renewcommand{\1}{\mathop{\mathds{1}}\nolimits} %trivial Chow motive
\newcommand{\M}{\ensuremath\mathscr{M}}
\newcommand{\im}{\ensuremath\mathrm{Im}}
\newcommand{\pr}{\ensuremath\mathrm{pr}}

\newcommand{\starM}{\star_{\mathrm{Mult}}}
\newcommand{\starC}{\star_{\mathrm{Chern}}}

\begin{document}

	\title[Distinguished cycles and Section Property]{Distinguished cycles on
		varieties with motive of abelian type and the Section Property}
	\author{Lie Fu}
	\author{Charles Vial}

	
	\thanks{Lie Fu is supported by the Agence Nationale de la Recherche (ANR)
		through ECOVA (ANR-15-CE40-0002), HodgeFun (ANR-16-CE40-0011), LABEX MILYON
		(ANR-10-LABX-0070) of Universit\'e de Lyon and  \emph{Projet Inter-Laboratoire}
		2017, 2018 by F\'ed\'eration de Recherche en Math\'ematiques Rh\^one-Alpes/Auvergne
		CNRS 3490.}
	
	%	\thanks{}
	%	
	
	\address{Institut Camille Jordan, Universit\'e Claude Bernard Lyon 1, France}
	\email{fu@math.univ-lyon1.fr}
	\address{
		Universit\"at Bielefeld, Germany}
	\email{vial@math.uni-bielefeld.de}
	
	\date{\today}
	
	\begin{abstract} 
		A remarkable result of Peter O'Sullivan asserts that the algebra epimorphism
		from the rational Chow ring of an abelian variety to its rational Chow ring
		modulo numerical equivalence admits a (canonical) section. Motivated by
		Beauville's splitting principle, we formulate a conjectural Section Property
		which predicts that for smooth projective holomorphic symplectic varieties there
		exists such a section of algebra whose image contains all the Chern classes of
		the variety. In this paper, we investigate this property for (not necessarily
		symplectic) varieties with Chow motive of abelian type. We introduce the notion
		of symmetrically distinguished abelian motive and use it to provide a sufficient
		condition for a smooth projective variety to admit such a section. We then give
		series of examples of varieties for which our theory works. For instance, we
		prove the existence of such a section for arbitrary products of varieties with
		Chow groups of finite rank, abelian varieties, hyperelliptic curves, Fermat
		cubic hypersurfaces, Hilbert schemes of points on an abelian surface or a Kummer
		surface or a K3 surface with Picard number at least 19, and generalized Kummer
		varieties. The latter cases provide evidence for the conjectural Section
		Property and exemplify the mantra that the motives of holomorphic symplectic
		varieties should behave as the motives of abelian varieties, as algebra objects. 
	\end{abstract}
	
	\maketitle
	\setcounter{tocdepth}{2}

	\tableofcontents
	
	%		(Beware that our notation conflicts with the notation of \cite{MR2187148}
	%		where $\DCH(X)$ stands for the subalgebra generated
	%		by divisors.)
	
	\vspace{-30pt}
	\section*{Introduction}
	
	Let $X$ be a smooth projective variety over a field $k$. We denote by $\CH(X)$ its Chow ring
	with rational coefficients, and by $\overline{\CH}(X)$ the quotient of $\CH(X)$
	by numerical equivalence of algebraic cycles. 
	The aim of this work is to build upon a recent result of O'Sullivan
	\cite{MR2795752} and give sufficient conditions on a smooth projective variety
	$X$ for the $\Q$-algebra epimorphism $\CH(X) \twoheadrightarrow
	\overline{\CH}(X)$ to admit a section that contains the Chern classes of $X$.
	This amounts to \emph{lift} numerical cycle classes to cycle classes in the Chow groups
	such that the lifted cycles form a subalgebra and the lifting of the Chern
	classes are the corresponding Chow-theoretic Chern classes.
	
	\subsection{Motivation\,: the motives of holomorphic symplectic varieties}
	It is an insight of Beauville that the motives of smooth projective holomorphic
	symplectic varieties should behave in a similar way to the motives of abelian
	varieties as algebra objects in the category of Chow motives. Following the
	seminal work \cite{BVK3}, Beauville \cite{MR2187148} (meta-)conjectured that the
	conjectural Bloch--Beilinson filtration on the Chow ring of holomorphic
	symplectic varieties should split. This will subsequently be referred to as the
	\emph{splitting principle}. That the conjectural Bloch--Beilinson filtration on
	the Chow ring of abelian varieties should split was established by Beauville
	\cite{MR826463}. 
	
	\subsubsection{The conjecture of Beauville} A first verifiable consequence of
	this splitting principle for \emph{simply connected} holomorphic symplectic
	varieties is the following concrete conjecture, called \emph{weak splitting
		property}\,; see \cite{MR2187148} for details.
	
	\begin{conj3}[Beauville \cite{MR2187148}] Let $X$ be a simply connected
		\footnote{This condition ensures that $\CH^1(X) \twoheadrightarrow
			\overline{\CH}^1(X)$ is an isomorphism.} smooth projective holomorphic
		symplectic variety, and denote by $R(X)$ the subalgebra of $\CH(X)$ generated by
		divisors. Then the composition of the following natural maps is injective
		$$R(X) \hookrightarrow \CH(X) \twoheadrightarrow \overline{\CH}(X).$$
	\end{conj3}

	This conjecture was checked for K3 surfaces in the seminal work of Beauville
	and Voisin \cite{BVK3}, and in \cite{MR2187148} Beauville checked it for Hilbert
	schemes of length-2 and length-3 subschemes on a K3 surface. The conjecture was
	later strengthened by Voisin \cite{MR2435839} who added the Chern classes of $X$
	to the set of generators of $R(X)$ (see also \cite{MR3524175}). Since then, the
	strengthened conjecture has been shown to hold in a number of cases\,; see
	\cite{MR2435839}, \cite{MR3356741}, \cite{MR3351754}, \cite{Ries2016}, \cite[\S
	10]{MHRCK3} and \cite{FranchettaHK}.
	
	\subsubsection{Multiplicative Chow--K\"unneth decompositions}
    Beauville's splitting principle was reformulated in \cite{SV}
	directly on the level of Chow motives, without pre-supposing the existence of the Bloch--Beilinson filtration. In the case of abelian varieties, Deninger and Murre \cite{DM} constructed
	a canonical Chow--K\"unneth decomposition of the motive of an abelian variety,
	lifting to the motivic level the decomposition of Beauville on the level of the
	Chow ring \cite{MR826463}. It can be checked that the decomposition of
	Deninger--Murre is compatible with the algebra structure on the Chow motives of
	abelian varieties\,; following \cite{SV}, we say that abelian varieties admit a \emph{multiplicative
		Chow--K\"unneth decomposition}. We refer to Section~\ref{S:mck} for definitions
	and properties of (multiplicative) Chow--K\"unneth decompositions. Similarly, for holomorphic symplectic varieties, the splitting principle suggests the following case-by-case verifiable
	
	\begin{conj3}[Multiplicative Chow--K\"unneth decomposition \cite{SV}]
		A holomorphic symplectic variety $X$ admits a multiplicative self-dual
		Chow--K\"unneth decomposition with the additional property that the Chern
		classes $c_i(X)$ belong to $\CH(X)_{(0)}$.\footnote{See \eqref{E:grading} for
			the definition of the grading $\CH(X)_{(*)}$.}
	\end{conj3}
	
	The decomposition of the small diagonal for K3 surfaces of Beauville--Voisin
	\cite{BVK3} in fact establishes the existence of a multiplicative
	Chow--K\"unneth decomposition for K3 surfaces\,; see \cite[Proposition
	8.14]{SV}. The existence of a multiplicative Chow--K\"unneth decomposition was
	established for the Hilbert scheme of length-$2$ subschemes on a K3 surface in
	\cite{SV}, more generally  for the Hilbert scheme of length-$n$ subschemes on a
	K3 or abelian surface in \cite{VialHilb}, and for generalized Kummer varieties in
	\cite{MHRCKummer}.

	\subsubsection{O'Sullivan's theorem}
	There is yet another verifiable consequence of Beauville's splitting principle, which
	will be our main focus here. The Bloch--Beilinson conjectures (or Murre's
	conjecture (D) \cite{Murre}) predict that for any smooth projective variety, the
	composition $\CH^i(X)_{(0)} \hookrightarrow \CH^i(X) \twoheadrightarrow
	\overline{\CH}^i(X)$ is an isomorphism of $\Q$-vector spaces for all $i$. In the
	case where the conjectural Bloch--Beilinson filtration splits, $\CH(X)_{(0)}$ is
	a $\Q$-subalgebra of $\CH(X)$ and we would therefore expect that $\CH(X)_{(0)}$
	provides a section to the $\Q$-algebra epimorphism $\CH(X) \twoheadrightarrow
	\overline{\CH}(X)$. In the case of abelian varieties, this was conjectured by
	Beauville \cite{MR826463}. A breakthrough in that direction is the following
	result due to O'Sullivan.
	%	, based on the work of Kimura \cite{MR2107443} and Andr\'e--Kahn
	%\cite{MR1956434}.
	
	\begin{thm3}[O'Sullivan \cite{MR2795752}]
		Let $A$ be an abelian variety. Then the
		$\Q$-algebra epimorphism $$\CH(A) \twoheadrightarrow \overline\CH(A)$$
		admits a  section (as $\Q$-algebras), whose image consists of
		\emph{symmetrically distinguished cycles} in the sense of Definition
		\ref{def:SD}.
	\end{thm3}
	
	See Theorems \ref{thm:section} and \ref{thm:SD} for a more precise version of
	O'Sullivan's theorem. In particular, O'Sullivan's theorem establishes the
	following version\footnote{This question was asked by Voisin as a more
		accessible consequence of Beauville's more general conjecture in
		\cite{MR826463}.} of Beauville's conjecture for abelian varieties (see
	\cite{Ancona} and \cite{MR3456759} for alternative proofs)\,: \emph{if $A$ is an
		abelian variety, then the subalgebra of $\CH(A)$ generated by symmetric divisors
		injects into cohomology \emph{via} the cycle class map.}
	In this paper, inspired by the work of O'Sullivan \cite{MR2795752} on the Chow
	rings of abelian varieties, we would like to address the following consequence
	of Beauville's splitting principle.
	
	\begin{conj2}[Section Property] \label{C:section}
		Let $X$ be a smooth projective holomorphic symplectic variety. Then the
		$\Q$-algebra epimorphism $$\CH(X) \twoheadrightarrow \overline\CH(X)$$
		admits a  section (as $\Q$-algebras) whose image contains the Chern classes of
		$X$.
	\end{conj2}
	
	Conjecture \ref{C:section} implies Beauville's \emph{weak splitting property}
	Conjecture \cite{MR2187148}, as well as its refinement due to Voisin
	\cite{MR2435839}, because $\CH^1(X) \twoheadrightarrow \overline\CH^1(X)$ is an
	isomorphism for smooth projective varieties $X$ with vanishing irregularity.
	%	A striking consequence of the result of O'Sullivan \cite{MR2795752} asserts
	%that
	%	Conjecture \ref{C:section} holds for abelian varieties. 
	We prove the following result (Propositions \ref{P:products}, \ref{cor:BirHK}, \ref{prop:kummer},
	\ref{prop:k3large}, \ref{P:hilb} and \ref{P:genK}) in support of
	Conjecture~\ref{C:section}.

	\begin{thm2}\label{T:genK}
		Let $X$ be a product of holomorphic symplectic varieties that are birational
		to either the Hilbert scheme of length-$n$ subschemes on an abelian surface or a
		Kummer surface or a K3 surface with Picard number $\geq 19$, or a generalized
		Kummer variety. Then Conjecture \ref{C:section}
		holds for $X$.
	\end{thm2}
	
Finally, we note that the notion of symmetrically distinguished cycles on an abelian variety $A$ depends on the choice of an origin for $A$, and in particular that there are at least as many sections to the algebra epimorphism $\CH(A) \to \overline{\CH}(A)$ as the number of rational equivalence classes of points on $A$. However, in the case of smooth projective irreducible holomorphic symplectic (\emph{i.e.}, hyper-K\"ahler) varieties, we expect that a section as in Conjecture \ref{C:section}, if it exists, is unique\,; and we also expect that cycles that are either classes of co-isotropic subvarieties (see~\cite{MR3524175}) or restrictions of cycles defined on the \emph{universal family} belong to the image of the section (we refer to \cite{FranchettaHK} for some evidence).

	\subsection{Distinguished cycles on varieties with motive of abelian type}
	Although our primary motivation for this work was to establish Theorem
	\ref{T:genK}, we were led to consider the following broader question (see
	Question \ref{Question})\,:
	Suppose $X$ is a smooth projective variety whose Chow motive is isomorphic to
	a direct summand of the motive of an abelian variety (such varieties are said to
	have \emph{motive of abelian type}, see Definition \ref{def:Mab}).  
	Are there sufficient conditions on $X$ that ensure that the epimorphism
	$\CH(X) \twoheadrightarrow \overline{\CH}(X)$ admits a section that is
	compatible with the intersection product\,? 
	For that purpose we introduce the notion of \emph{distinguished cycles} on varieties with motive of abelian type\,; see Definition
	\ref{def:distinguished}. 
	Precisely, distinguished cycles depend \emph{a priori} on the choice of a \emph{marking}\,:
	a  marking for a variety $X$ (see Definition~\ref{def:marking}) is an isomorphism $\phi:
	\h(X)\xrightarrow{\isom} M$ of Chow motives, where $M$ is\footnote{Strictly
		speaking, $M$ should be an object in the category $\M^{ab}_{sd}$ introduced in Definition
		\ref{def:SDMotive}.} a direct summand of a
	Chow motive of the form $\oplus_{i} \h(A_{i})(n_{i})$ cut out by an idempotent
	matrix $P$ of \emph{symmetrically distinguished} cycles, where $A_{i}$ is an
	abelian variety,  and $n_{i}\in \Z$. Given such a marking,  the group of
	\emph{distinguished cycles} $\DCH_{\phi}(X)$ consists of the image under $P_{*}$ of the
	symmetrically distinguished cycles on each $A_i$, in the sense of O'Sullivan
	(see Definition \ref{def:SD}), transported \emph{via} the induced isomorphism
	$\phi_{*}:\CH(X)\xrightarrow{\isom}\CH(M)$.  The question becomes\,: What are
	sufficient conditions on  a marking $\phi$ for $\DCH_{\phi}(X)$ to be a
	subalgebra of $\CH(X)$\,? In Proposition \ref{prop:Disting}, we show that it
	suffices that the following condition holds
	$$(\starM) \text{ The small diagonal  }  \delta_{X} \text{  belongs to }
		\DCH_{\phi^{\otimes 3}}(X^{3}). $$
	%\begin{itemize}
		%\item[$(\star1)$]  the diagonal $\Delta_{X}$ belongs to $\DCH_{\phi^{\otimes
		%2}}(X^{2})$, that is, under the induced isomorphism $\phi^{\otimes 2}_{*}:
		%\CH(X^{2})\xrightarrow{\isom} \CH(M^{\otimes 2})$, the image of $\Delta_{X}$ is
		%symmetrically distinguished\,;
		%\item[$(\starM)$] the small diagonal $\delta_{X}$ belongs to
		%$\DCH_{\phi^{\otimes 3}}(X^{3})$.
	%\end{itemize}
	%	
	%	the classes of the diagonal $\Delta_X \subset X\times X$ and of the small
	%diagonal $\delta_X \subset X\times X\times X$ be transported to symmetrically
	%distinguished cycles \emph{via} $\phi \otimes \phi$ and $\phi \otimes \phi
	%\otimes \phi$, respectively.
	Since it is natural to expect that the Chern classes are distinguished, we will
	also require that the Chern classes of $X$ are transported  to symmetrically
	distinguished cycles \emph{via} $\phi$, \emph{i.e.}, that the marking $\phi$ also
	satisfies the condition
	%\begin{itemize}
		$$(\starC) \text{ All Chern classes } c_{1}(X), c_{2}(X), \cdots \text{ belong to } \DCH_{\phi}(X). $$
	%\end{itemize}
	These two conditions are gathered to Condition $(\star)$ in Definition
	\ref{def:Star}, where we also consider the more general situation where $X$ is endowed
	with the action of a finite group $G$. The condition $(\starC)$ is not only esthetically pleasing, it is also essential to establish that the condition $(\star)$ is stable under natural constructions such as blow-ups (Proposition~\ref{prop:Blup}).
	\medskip
	
	Therefore in order to prove Theorem \ref{T:genK}, it is enough to exhibit a
	suitable marking for $X$ such that the Chern classes and the small
	diagonal are distinguished with respect to the (product) markings. If such a
	suitable marking for $X$ exists, we will say that $X$ satisfies $(\star)$\,; see
	Definition \ref{def:Star}. This condition is strictly stronger than the
	condition of having motive of abelian type\,; see Section \ref{S:counterex} for
	examples of varieties with motive of abelian type that do not satisfy $(\star)$
	and/or 
	%	in Proposition \ref{prop:gencurve}, we show that a very general curve of genus
	%$>2$ does not satisfy $(\star)$. 
	are such that the $\Q$-algebra epimorphism $\CH(X) \twoheadrightarrow
	\overline\CH(X)$ does not admit a section. Thus that smooth projective
	holomorphic symplectic varieties should satisfy the \emph{Section Property} in
	Conjecture \ref{C:section} is remarkable. We also want to stress that the
	original Section Property, \emph{i.e.}, the existence of section of the algebra
	epimorphism $\CH(X)\surj \bar{\CH}(X)$, does not behave well enough under basic
	operations, for instance, products, blow-ups, quotients \emph{etc.}\,; however,
	the closely related Condition $(\star)$ is essentially motivic and behaves much
	better, see Section \ref{sect:examples}.
	In view of Proposition \ref{prop:Disting}, one could also be optimistic and go
	as far as proposing\,:
	\begin{conj2}[Distinguished Marking]
A  smooth projective holomorphic symplectic variety admits a marking that satisfies $(\star)$.
	\end{conj2} 
\noindent 	In particular,  this conjecture stipulates that smooth projective holomorphic symplectic varieties have
	motives of abelian type. Some evidence towards the latter is provided by recent
	work of Kurnosov--Soldatenkov--Verbitsky \cite{KSV} on Kuga--Satake
	constructions.
	\medskip
	
	Although holomorphic symplectic varieties seem to play a central role, we
	provide many other examples of smooth projective varieties $X$ that satisfy
	$(\star)$ and hence are such that the
	$\Q$-algebra epimorphism $\CH(X) \twoheadrightarrow \overline\CH(X)$
	admits a section whose image contains the Chern classes of $X$. The building
	blocks (see Section \ref{sect:Examples}) are given by abelian varieties
	(O'Sullivan's theorem), varieties with Chow groups of finite rank
	(Proposition~\ref{prop:trivial}), hyperelliptic curves (Corollary
	\ref{cor:hyperelliptic}), cubic Fermat hypersurfaces (Proposition
	\ref{prop:Fermat}), K3 surfaces with Picard rank $\geq 19$
	(Proposition~\ref{prop:k3large}), and generalized Kummer varieties
	(Proposition~\ref{P:genK}). One can then construct new examples (see
	Section~\ref{sect:examples}) of varieties satisfying $(\star)$ by taking
	products (Proposition \ref{P:products}), certain projective bundles and blow-ups
	(Example \ref{ex:Schur}, Propositions \ref{prop:PB} and \ref{prop:Blup}, here that the Chern classes are
	distinguished plays a central role), certain \'etale or cyclic quotients
	(Propositions \ref{prop:cover}, \ref{prop:etale} and \ref{P:quotient}), Hilbert
	squares and the first two nested Hilbert schemes (Propositions \ref{prop:Hilb2}
	and \ref{prop:NestHilb}), Hilbert schemes and nested Hilbert schemes of curves
	or surfaces satisfying $(\star)$ (Remark \ref{rem:hyperelliptic} and Proposition
	\ref{P:hilb}), and birational transforms of irreducible symplectic varieties
	(Corollary \ref{cor:BirHK}). Combining the above-mentioned results, we obtain
	
	\begin{thm2} \label{thm2 CK} Let $E$ be the smallest collection of isomorphism classes of smooth
		projective complex varieties that contains varieties with Chow groups of
		finite rank (as $\Q$-vector spaces), abelian varieties, hyperelliptic curves,
		cubic Fermat hypersurfaces, K3 surfaces with Picard rank $\geq 19$, and
		generalized Kummer varieties, 
		and that is
		stable under the following operations\,:
		\begin{enumerate}[(i)]
			\item if $X$ and $Y$ belong to $E$, then $X\times Y$ belongs to $E$\,;
			%\item if $X\to Y$ is a finite \'etale morphism and $X$ belongs to $E$, then
			%so does $Y$\,;
			\item if $X$ belongs to $E$, then $\PP(\oplus_{i}\mathbb
			S_{\lambda_{i}}{T}_X)$ belongs to $E$, where
			${T}_X$ is the tangent bundle of $X$, the $\lambda_{i}$'s are non-increasing
			sequences of integers and $\mathbb S_{\lambda_{i}}$ is the Schur functor
			associated to $\lambda_{i}$\,;
			\item if $X$ belongs to $E$, then the Hilbert scheme of length-$2$
			subschemes $X^{[2]}$, as well as the nested Hilbert schemes $X^{[1,2]}$ and
			$X^{[2,3]}$ belong to $E$\,;
			\item if $X$ is a curve or a surface that belongs to $E$, then for any $n\in \N$, the Hilbert
			scheme of length-$n$
			subschemes $X^{[n]}$, as well as the nested Hilbert schemes $X^{[n, n+1]}$
			belong to $E$.
			\item if one of two birationally equivalent irreducible holomorphic
			symplectic varieties belongs to $E$, then so does the other.
		\end{enumerate}
		If $X$ is a smooth projective variety whose isomorphism class belongs to $E$, then $X$ admits a marking that satisfies $(\star)$, so that the
		$\Q$-algebra epimorphism $\CH(X) \twoheadrightarrow \overline\CH(X)$
		admits a  section (as $\Q$-algebras) whose image contains the Chern classes
		of $X$.
	\end{thm2}
	
	It is further shown in \cite{LatVial}  that a certain complete family of Calabi--Yau varieties and certain rigid Calabi--Yau varieties, constructed by Cynk and Hulek, as well as certain varieties constructed by Schreieder satisfy the condition $(\star)$, so that these varieties can be added to the set $E$ of Theorem \ref{thm2 CK}.\medskip
	
	An immediate consequence of Theorem \ref{thm2 CK} is the following concrete result related to
	Beauville's weak splitting property and Beauville--Voisin conjecture (but beyond
	the hyper-K\"ahler context)\,:
	\begin{cor2}
		Let $X$ be a smooth projective variety that belongs to the collection $E$ of
		Theorem \ref{thm2 CK}. Assume that $X$ is regular \footnote{A smooth projective
			variety $X$ over an algebraically closed field $k$ is called \emph{regular}, if
			its Picard variety is trivial, so that the projection morphism $\CH^1(X) \to
			\overline{\CH}^1(X)$ is an isomorphism. Note that the \emph{irregularity},
			\emph{i.e.}, the dimension of the Picard variety, is always less or equal to
			$\dim H^{1}(X, \mathcal O_{X})$, and equal to $\dim H^{1}(X, \mathcal O_{X})$
			when $char(k)=0$ by Hodge theory.} and denote $R(X)$ the $\Q$-subalgebra of
		$\CH(X)$ generated by divisors and Chern classes. Then the natural composition $$R(X)
		\hookrightarrow \CH(X) \twoheadrightarrow \overline{\CH}(X)$$ is injective.
	\end{cor2}

	Note that all smooth projective varieties which we can show satisfy $(\star)$
	were already shown to admit a self-dual multiplicative Chow--K\"unneth
	decomposition\,; see \cite[Theorem 2]{SV2}, and \cite{MHRCKummer} for the case
	of generalized Kummer varieties. In fact, condition $(\star)$ implies the
	existence of a multiplicative Chow--K\"unneth decomposition (Proposition
	\ref{prop:multmarking}).  Note also that the structure of 
	Section~\ref{sect:examples} is similar to the structure of \cite[Section
	3]{SV2}. We refer to Section \ref{S:mck} for more on multiplicative
	Chow--K\"unneth decompositions and links to this work. Finally, we note that
	while the result of Beauville--Voisin \cite{BVK3} shows that the
	$\Q$-algebra epimorphism $\CH(S) \twoheadrightarrow \overline\CH(S)$
	admits a  section whose image contains the Chern classes of $S$, for a K3
	surface $S$, and while it can be shown \cite{VialHilb} that the Hilbert scheme
	of length-$n$ subschemes on a K3 surface has a self-dual multiplicative
	Chow--K\"unneth decomposition, we do not know how to show in general that a K3
	surface satisfies the condition $(\star)$, nor do we know how to show that the
	Hilbert scheme of length-$n$ subschemes on a K3 surface satisfies the
	\emph{Section Property} (Conjecture \ref{C:section}). In fact it is even an open
	problem to show in general that K3 surfaces have motive of abelian type.

	\medskip
	\noindent \textbf{Conventions and Notations.}	We work throughout the paper over
	an arbitrary algebraically closed field $k$, except in \S\S \ref{s:birational}, \ref{s:Fermat}, \ref{s:k3large} and \ref{S:counterex} where $k$ is assumed to be the field of complex numbers . Chow groups $\CH^i$ are always
	understood to be with rational coefficients. For a smooth projective variety
	$X$, we will write $\CH(X)$ for the (graded) rational Chow ring $\bigoplus_i
	\CH^i(X)$. We will denote by $\overline\CH^i(X)$  the rational Chow group modulo
	numerical equivalence and $\overline{\CH}(X)$ the rational Chow ring modulo
	numerical equivalence.
%	\,; the image of a cycle $a$ along the projection $\CH(X) \to \bar\CH(X)$ will be denoted $\bar a$. 
	%	
	%	If a $d$-dimensional variety $X$ is endowed with the action of a finite group
	%$G$, then we denote $g_X \in \CH^d(X\times X)$ the graph of morphism induced by
	%the action of $g\in G$ on $X$.
	%	
	An \emph{abelian variety} is always assumed to be connected and with a fixed
	origin.

	\medskip
	
	\noindent \textbf{Acknowledgments.} We thank Bruno Kahn for useful discussions.
	We are especially grateful to Peter O'Sullivan for his comments and for
	explaining his work to us with patience and clarity, and to the excellent referee for his or her pertinent remarks.

	\section{Symmetrically distinguished cycles}\label{sect:SD}
	In this section, we review the theory of symmetrically distinguished cycles
	developed by O'Sullivan in \cite{MR2795752} and, with a view towards
	applications, extend it slightly following the authors' previous work
	\cite{MHRCKummer} joint with Zhiyu Tian.

	\subsection{Motives of abelian type}
	Let $\CHM:=\CHM(k)_{\Q}$ and $\NumM:=\NumM(k)_{\Q}$ be respectively the
	category of rational Chow motives and that of rational numerical motives over
	the base field $k$. By definition, there is a natural (full) projection
	functor\,: $$\CHM\to \NumM,$$ which sends a Chow motive to the corresponding
	numerical motive and sends any cycle/correspondence modulo rational equivalence
	to its class modulo numerical equivalence. A typical object in these two
	categories is a triple $(X, p, n)$ with $X$ a smooth projective variety
	over $k$, $p\in \CH^{\dim X}(X\times X)$ or $\bar\CH^{\dim X}(X\times X)$ a
	projector (\emph{i.e.}, $p\circ p=p$) and $n\in \Z$. See \cite{MR2115000} for the basic notions.\medskip
	
	Let us introduce the following subcategories of $\CHM$ and $\NumM$ that will be
	relevant to our work.
	
	\begin{defn}[Motives of abelian type]\label{def:Mab}
		Let $\M^{ab}$ (\emph{resp.} $\bar {\M^{ab}}$) be the strictly\footnote{A full
			subcategory is called \emph{strictly full}, if it is closed under isomorphisms.}
		full, thick and rigid tensor subcategory of $\CHM$ (\emph{resp.} $\NumM$) generated by
		the motives of abelian varieties. A motive is said to be \emph{of abelian type}
		if it belongs to $\M^{ab}$; equivalently, if one of its Tate twists is
		isomorphic to the direct summand of the motive of an abelian variety.  We have
		the restriction of the projection functor\,:
		$$\pi:  \M^{ab}\to \bar{\M^{ab}}.$$
	\end{defn}
	
	\begin{ex}\label{ex}
		The Chow (\emph{resp.} numerical) motives of the following algebraic varieties
		belong to the category $\M^{ab}$ (\emph{resp.} $\bar{\M^{ab}}$)\,:
		\begin{enumerate}[(i)]
			\item projective spaces, Grassmannian varieties and more generally projective homogeneous
			varieties under a linear algebraic group and toric varieties\,;
			\item smooth projective curves\,;
			\item Kummer K3 surfaces\,; K3 surfaces with Picard numbers at least 19 as well
			as their (nested) Hilbert schemes\,;
			\item abelian torsors\,;
			\item Hilbert schemes of abelian surfaces\,; 
			%	\comm{(un espace des modules des faisceaux sur une surface ab\'elienne?)}
			\item generalized Kummer varieties\,; 
			%	\comm{(fibre d'Albanese d'un espace des modules des faisceaux sur une surface
			%ab\'elienne?)}
			\item Fermat hypersurfaces\,;
			\item projective bundles over and products of the examples above.
		\end{enumerate}
		As far as the authors know, all examples of motives that have been proven to be 
		(Kimura) finite dimensional (\cite{MR2107443})  belong\footnote{When $k$ has characteristic zero, there
			are many varieties whose motive is not in $\M^{ab}$, while
			conjecturally all varieties have finite dimensional motive.} to the category
		$\M^{ab}$.
	\end{ex}

	Let us state the following result of \cite{MR2795752}, which is built upon
	\cite{MR2107443} and \cite{MR1956434}\,:
	\begin{thm}[O'Sullivan {\cite[Theorem 5.5.3]{MR2795752}}]\label{thm:section}
		The projection $\otimes$-functor $\pi: \M^{ab}\to \bar{\M^{ab}}$ has a
		right-inverse $T$, which is unique up to a unique tensor isomorphism above the
		identity.
	\end{thm}
	
	\begin{rmk}
		See Theorem \ref{thm:SD} together with Remark \ref{rmk:Section}, for a
		down-to-earth understanding of Theorem \ref{thm:section}.
	\end{rmk}
	
	\begin{rmk}
		The existence of the right-inverse $\otimes$-functor $T$ is ensured by a general
		result of Andr\'e--Kahn \cite{MR1956434} concerning the so-called Wedderburn
		categories, and such a section is unique only up to a \emph{non-unique} tensor
		conjugacy. The Hopf algebra structure on the motive of an abelian variety, given
		by the diagonal embedding and the group structure (in particular the
		$(-1)$-involution), allows O'Sullivan to make the section $T$ unique up to a
		\emph{unique} tensor conjugacy above the identity.
	\end{rmk}
	
	\begin{rmk}\label{warning}
		The section $T$ in Theorem \ref{thm:section} cannot be defined uniquely.
		Indeed, let $B$ be a torsor under an abelian variety $A$ of dimension $g$.
		Obviously $A$ and $B$ have isomorphic Chow motives. If a canonical section $T$
		were constructed for morphisms between $\1(-g)$ and $\h(B)$, then we would have
		a canonical 1-dimensional subspace $\DCH_{0}(B)$ inside the infinite-dimensional
		space $\CH_{0}(B)$, hence a canonical degree-one 0-cycle of $B$. However, as the
		origin of $B$ is not fixed, there is neither a privileged point nor a privileged
		non-trivial  0-cycle. 
	\end{rmk}

	\subsection{Symmetrically distinguished cycles on abelian varieties}
	O'Sullivan defines the following concrete notion of symmetrically distinguished
	cycles on an abelian variety $A$, and shows (Theorem \ref{thm:SD}) that these
	provide a section to $$\CH(A) \twoheadrightarrow \overline{\CH}(A)$$ that is
	compatible with the intersection product. 
	
	\begin{defn}[Symmetrically distinguished cycles on abelian varieties
		\cite{MR2795752}]\label{def:SD}
		%\comm{Ici on veut sans doute dire "Definition" ? Il serait aussi peut-etre bien
		%de dire explicitement que les symmetrically distinguished cycles decrivent
		%explicitement l'unique section $T$ pour les motifs de varieties abeliennes.}
		Let $A$ be an abelian variety and $\alpha\in \CH(A)$. For each integer $m\geq
		0$, denote by $V_{m}(\alpha)$ the $\Q$-vector subspace of $\CH(A^{m})$
		generated
		by elements of the form
		$$p_{*}(\alpha^{r_{1}}\times \alpha^{r_{2}}\times \cdots\times
		\alpha^{r_{n}}),$$
		where $n\leq m$, $r_{j}\geq 0 $ are integers, and where $p : A^{n}\to A^{m}$ is a
		closed immersion each component $A^{n}\to A$ of which is either a projection or
		the composite of a projection with the involution $[-1]: A\to A$. Then $\alpha$ is
		\emph{symmetrically distinguished} if for every $m$ the restriction of the
		projection $\CH(A^{m})\to \overline\CH(A^{m})$ to $V_{m}(\alpha)$ is
		injective.
		The subgroup of symmetrically distinguished cycles is denoted by $\DCH(A)$.
		
	\end{defn}
	
	Here is the main result of O'Sullivan \cite{MR2795752}, which is the most important ingredient that we use throughout this paper\,:
	\begin{thm}[O'Sullivan {\cite[Theorem 6.2.5]{MR2795752}}]\label{thm:SD}
		Let $A$ be an abelian variety. Then the symmetrically distinguished cycles in
		$\CH(A)$ form a graded $\Q$-subalgebra $\DCH(A)$ that contains symmetric
		divisors and that is
		stable under pull-backs and push-forwards along homomorphisms of abelian
		varieties. Moreover
		the composition $$\DCH(A)\hookrightarrow \CH(A)\twoheadrightarrow \overline
		\CH(A)$$ is an isomorphism of $\Q$-algebras.
	\end{thm}
	
	\begin{rmk} \label{R:symdistbeauville}
		Given an abelian variety $A$, thanks to Theorem \ref{thm:SD}, it is easy to see by looking at the eigenvalues of multiplication-by-$m$ endomorphisms ($m\in \Z$) that
		$\DCH(A)$ is a subalgebra of $\CH(A)_{(0)}$, where
		$\CH(A)_{(*)}$ refers to Beauville's decomposition\footnote{Beauville's
			decomposition coincides with the decomposition induced, as in \eqref{E:grading},
			by the Chow--K\"unneth decomposition of Deninger--Murre \cite{DM}.}
		\cite{MR826463}. 
		Moreover, the inclusion $\DCH^i(A) \subseteq \CH^i(A)_{(0)}$ is an equality for
		$i\leq 1$ as well as for $i\geq \dim A - 1$ by the Fourier transform
		\cite{MR726428}.
		Beauville's conjectures on abelian varieties in \cite{MR826463} would imply
		that the subalgebra $\DCH(A)$ is equal to the direct summand $\CH(A)_{(0)}$. In
		this sense, O'Sullivan's work \cite{MR2795752} can be viewed as a step towards
		Beauville's conjectures. 
	\end{rmk}

	\subsection{... on abelian torsors with torsion structures}\label{subsect:atts}
	For later use, we give here a minor extension of O'Sullivan's theory. The main
	idea appeared in our previous work \cite{MHRCKummer}\,: to treat the Chow
	motives of some algebraic varieties like Hilbert schemes of abelian surfaces and
	generalized Kummer varieties, it is inevitable to deal with `disconnected
	abelian varieties' where there is no natural choice for the origins on the
	components, whence the notion of symmetrically distinguished cycles \emph{a
		priori} fails. However, a simple but crucial observation made in
	\cite{MHRCKummer} is that we have a canonical notion of torsion points on these
	components. 
	\begin{defn}[Abelian torsors with torsion structure
		\cite{MHRCKummer}]\label{def:atts}
		An \emph{abelian torsor with torsion structure}, or an \emph{a.t.t.s} for
		short, is a pair $(X, Q_{X})$ where $X$ is a connected smooth projective
		variety
		and $Q_{X}$ is a subset of closed points of $X$ such that there exists an isomorphism, as
		algebraic varieties, $f: X\xrightarrow{\isom} A$ from $X$ to an abelian variety
		$A$ which
		induces a bijection between $Q_{X}$ and $\mathrm{Tor}(A)$, the set of all
		torsion points of $A$. A choice of such an isomorphism $f$ is called a
		\emph{marking}. A
		morphism of a.t.t.s's $(X,Q_X) \to (Y,Q_Y)$ consists of a morphism of
		algebraic varieties $f : X\to Y$
		such that $f(Q_X) \subseteq Q_Y$.
	\end{defn}
	This notion of \emph{a.t.t.s} sits in between the notion of abelian variety
	(with fixed origin) and that of abelian torsor (without origin).
	
	\begin{defn}[Symmetrically distinguished cycles on a.t.t.s's]\label{def:SD2}
		Given an a.t.t.s $(X,Q_{X})$, an algebraic cycle $\gamma\in \CH(X)$ is called
		\emph{symmetrically distinguished} if, for a marking $f: X\xrightarrow{\isom}
		A$, the cycle
		$f_{*}(\gamma)\in \CH(A)$ is symmetrically distinguished in the sense of
		O'Sullivan (Definition \ref{def:SD}). By \cite[Lemma 6.7]{MHRCKummer}, this
		definition
		is independent of the choice of marking. An algebraic cycle on a disjoint
		union
		of a.t.t.s's is symmetrically distinguished if it is so on each component. We
		denote by $\DCH(X)$ the subspace of symmetrically distinguished cycles.
	\end{defn}
	
	We have the following generalization of Theorem \ref{thm:SD}\,; see
	\cite[Proposition 6.9]{MHRCKummer}.
	Its proof uses the fact that torsion points on an abelian variety are all
	rationally equivalent (with $\Q$-coefficients).
	
	\begin{thm}\label{thm:SD2}
		Let $(X,Q_{X})$ be an a.t.t.s. Then the symmetrically distinguished cycles in
		$\CH(X)$ form a graded $\Q$-subalgebra that  is stable under pull-backs and
		push-forwards along morphisms of a.t.t.s's. Moreover
		the composition $\DCH(X) \hookrightarrow \CH(X)\twoheadrightarrow \overline
		\CH(X)$ is an isomorphism of $\Q$-algebras.
	\end{thm}
	
	We refer to \cite[\S 6.2]{MHRCKummer} for more properties of symmetrically
	distinguished cycles on a.t.t.s's.

	\section{Symmetrically distinguished abelian motives}
	To make a more flexible use of O'Sullivan's Theorem \ref{thm:SD} in the
	language of motives, we introduce the following category $\M_{sd}^{ab}$. We refer to Remarks
	\ref{rmk:Section} and \ref{rmk:Envelope} for some motivations.
	
	\begin{defn}[The category $\M_{sd}^{ab}$]\label{def:SDMotive}
		The category of \emph{symmetrically distinguished abelian motives},
		denoted $\M^{ab}_{sd}$, is defined as follows\,:
		\begin{enumerate}[$(i)$]
			\item An object consists of the data of
			\begin{itemize}
				\item a positive integer $r\in \N^{*}$\,;
				\item a length-$r$ sequence of abelian varieties (thus with fixed
				origins) $A_{1},\ldots, A_{r}$\,;
				\item a length-$r$ sequence of integers $n_{1}, \ldots, n_{r}\in \Z$\,;
				%\item an $(r\times r)$ antisymmetric matrix of integers
				%$\left(a_{i,j}\right)_{1\leq i, j\leq r}$ satisfying the cocycle condition:
				%$a_{i,k}+a_{k,j}=a_{i,j}$ for all $1\leq i, j, k\leq r$.
				\item an $(r\times r)$-matrix $P:=\left(p_{i,j}\right)_{1\leq i, j\leq
					r}$ with $p_{i,j}\in \DCH^{\dim A_{i}+n_{j}-n_{i}}(A_{i}\times A_{j})$ a
				\emph{symmetrically distinguished} cycle (Definition \ref{def:SD}), such that
				$P\circ P=P$, that is, for all $1\leq i, j\leq r$, we have
				$$\sum_{k=1}^{r}p_{k,j}\circ p_{i,k}=p_{i,j} \text{   in    } \CH^{\dim
					A_{i}+n_{j}-n_{i}}(A_{i}\times A_{j})$$
			\end{itemize}
			Such an object is denoted in the sequel by a triple
			$$\left(A_{1}\sqcup\cdots \sqcup A_{r}, P=\left(p_{i,j}\right), (n_{1}, \ldots,
			n_{r})\right).$$
			\item The group of morphisms from $\left(A_{1}\sqcup\cdots \sqcup A_{r},
			P=\left(p_{i,j}\right), (n_{1}, \ldots, n_{r})\right)$ to another object
			$\left(B_{1}\sqcup\cdots \sqcup B_{s}, Q=\left(q_{i,j}\right), (m_{1}, \ldots,
			m_{s})\right)$ is defined to be the subgroup of 
			$$\bigoplus_{i=1}^{r}\bigoplus_{j=1}^{s} \CH^{\dim A_i +
				m_{j}-n_{i}}(A_i\times B_j)$$ (whose elements are viewed as an $(s\times
			r)$-matrix) given by $$Q\circ \left(\bigoplus_{i=1}^{r}\bigoplus_{j=1}^{s}
			\CH^{\dim A_i + m_{j}-n_{i}}(A_i\times B_j)\right)\circ P,$$
			where the multiplication law is the one between matrices.
			\item The composition is defined as usual by composition of
			correspondences.
			\item The category $\M^{ab}_{sd}$ is an additive category where the
			direct sum is given by 
			\begin{eqnarray*}
				&&\left(\bigsqcup_{i=1}^{r}A_{i}, P, (n_{1}, \ldots, n_{r})\right)\oplus
				\left(\bigsqcup_{j=1}^{s}B_{j}, Q, (m_{1}, \ldots, m_{s})\right)\\
				&=& \left(\bigsqcup_{i=1}^{r}A_{i}\sqcup \bigsqcup_{j=1}^{s}B_{j},
				P\oplus Q:=\begin{pmatrix} P & 0\\ 0 &Q\end{pmatrix}, (n_{1}, \ldots, n_{r},
				m_{1}, \ldots, m_{s}) \right)
			\end{eqnarray*}
			\item The category  $\M^{ab}_{sd}$ is a symmetric mono\"idal category
			where the tensor product is defined by 
			\begin{eqnarray*}
				&&\left(\bigsqcup_{i=1}^{r}A_{i}, P, (n_{1}, \ldots, n_{r})\right)\otimes
				\left(\bigsqcup_{j=1}^{s}B_{j}, Q, (m_{1}, \ldots, m_{s})\right)\\
				&=& \left(\bigsqcup_{i=1}^{r}\bigsqcup_{j=1}^{s}A_{i}\times B_{j},
				P\otimes Q, (n_{i}m_{j}; 1\leq i\leq r, 1\leq j\leq s) \right)
			\end{eqnarray*}
			where $P\otimes Q$ is the Kronecker product of two matrices.\\
			In particular, for any $m\in \Z$, the $m$-th \emph{Tate twist},
			\emph{i.e.}, the tensor product with the Tate object $\1(m):=(\Spec k, \Spec k,
			m)$ sends $\left(A_{1}\sqcup\cdots \sqcup A_{r}, P, (n_{1}, \ldots,
			n_{r})\right)$ to $\left(A_{1}\sqcup\cdots \sqcup A_{r}, P, (n_{1}+m, \cdots,
			n_{r}+m)\right)$. All Tate objects are $\otimes$-invertible.
			\item The category  $\M^{ab}_{sd}$ is rigid\,; the dual of  
			$\left(A_{1}\sqcup\cdots \sqcup A_{r}, P=\left(p_{i,j}\right), (n_{1},
			\ldots, n_{r})\right)$ is given by $\left(A_{1}\sqcup\cdots \sqcup A_{r},
			{}^{t}P:=({}^{t}p_{j,i}), (d_{1}-n_{1}, \ldots, d_{r}-n_{r})\right)$,
			where $d_{k}=\dim A_{k}$ and the $(i,j)$-th entry of ${}^{t}P$ is
			${}^{t}p_{j,i}\in \CH^{d_{i}+(d_{j}-n_{j})-(d_{i}-n_{i})}(A_{i}\times A_{j})$,
			the transpose of $p_{j,i}\in \CH^{d_{j}+n_{i}-n_{j}}(A_{j}\times A_{i})$.
		\end{enumerate}
		
		In a similar way, one can define the rigid symmetric mono\"idal additive
		category $\M^{atts}_{sd}$ by replacing in the above definition abelian varieties
		(thus with origin fixed) by abelian torsors with torsion structure (thus with
		only the subset of torsion points fixed, \emph{cf.} \S \ref{subsect:atts}). 
		With the notion and basic properties of symmetrically distinguished cycles
		extended to the case of abelian torsors with torsion structure in \S
		\ref{subsect:atts}, all the above constructions go through. It is important to point out
		that\footnote{We thank Peter O'Sullivan for reminding us of this subtle point.}
		$\M^{ab}_{sd}$ and $\M^{atts}_{sd}$ are \emph{not} subcategories of $\CHM$ since
		in the definition of motives, one uses varieties instead of \emph{pointed} varieties or
		varieties with additional structures. See however Lemma \ref{lemma:Equivalence} below.
		
		There are natural fully faithful additive tensor functors $$F:
		\M^{ab}_{sd}\to \M^{ab}\quad \text{and} \quad F': \M^{atts}_{sd}\to \M^{ab},$$ which send an object
		$\left(A_{1}\sqcup\cdots \sqcup A_{r}, P=\left(p_{i,j}\right), (n_{1}, \ldots,
		n_{r})\right)$ to the Chow motive $\im\left(P: \oplus_{i=1}^{r}\h(A_{i})(n_{i})
		\to \oplus_{i=1}^{r}\h(A_{i})(n_{i})\right)$. Here we use the facts that $\CHM$
		is pseudo-abelian and that $P$ induces an idempotent endomorphism of
		$\oplus_{i=1}^{r}\h(A_{i})(n_{i})$ by construction.
		
		For any object $M$ in $\M^{ab}_{sd}$ or $\M^{atts}_{sd}$ and any $i\in
		\Z$, the $i$-th \emph{Chow group} $\CH^{i}(M)$ is defined to be $\CH^{i}(F(M))$
		which is nothing but $\Hom_{\M^{ab}_{sd}}\left((\Spec k, \Spec k, -i),
		M\right)$.
		
	\end{defn}
	Despite the technical construction of the categories $\M^{ab}_{sd}$ and
	$\M^{atts}_{sd}$, they are, after all, not so different from the category
	$\M^{ab}$ of abelian motives (Definition \ref{def:Mab})\,:
	\begin{lem}[Relation with Chow motives of abelian type]\label{lemma:Equivalence}
		The functors $F: \M^{ab}_{sd}\to \M^{ab}$ and $F': \M^{atts}_{sd}\to
		\M^{ab}$ are equivalences of categories. 
	\end{lem}
	\begin{proof}
		These two functors are fully faithful by definition and we only have to show
		that they are essentially surjective. Consider  an object  in $\CHM$ isomorphic
		to $(A,p,n)$ with $A$ a $g$-dimensional abelian torsor, $p\in \CH^{g}(A\times
		A)$ a projector and $n\in \Z$. First we choose an origin for $A$ so that the
		symmetric distinguishedness makes sense in the rest of the proof. Using the existence of
		symmetrically distinguished cycles in each numerical cycle class (Theorem
		\ref{thm:SD}), one can find a symmetrically distinguished element $q\in
		\DCH^{g}(A\times A)$ such that $q$ is numerically equivalent to $p$. As $p$ is a
		projector, we know that $q\circ q$ is numerically equivalent to $q$. However, as
		$q\circ q$ and $q$ are both symmetrically distinguished, they must be equal by
		the uniqueness of symmetrically distinguished lifting in Theorem \ref{thm:SD},
		\emph{i.e.}, $q$ is a projector. Therefore $(A, p, n)$ is isomorphic, in $\CHM$,
		to $(A,q,n)$ which is in the image of the functor $F$. Finally, since $F$
		factorizes through $F'$,  $F'$ is also essentially surjective.
	\end{proof}

	Now we extend the notion of symmetrical distinguishedness from cycles on abelian
	varieties (Definition \ref{def:SD}) to morphisms in the category $\M^{ab}_{sd}$
	(and $\M^{atts}_{sd}$).
	\begin{defn}[Symmetrically distinguished morphisms in
		$\M^{ab}_{sd}$]\label{def:SDmorph}
		Given two objects in $\M^{ab}_{sd}$, say, $M:=\left(A_{1}\sqcup\cdots \sqcup
		A_{r}, P=\left(p_{i,j}\right), (n_{1}, \ldots, n_{r})\right)$ and\\ 
		$N:=\left(B_{1}\sqcup\cdots \sqcup B_{s}, Q=\left(q_{i,j}\right), (m_{1},
		\ldots, m_{s})\right)$, the subspace of \emph{symmetrically distinguished
			morphisms} from $M$ to $N$, denoted by $\DHom(M,N)$,  is defined to be
		$$\DHom(M,N):=Q\circ \bigoplus_{i,j} \DCH^{\dim A_i+m_{j}-n_{i}}(A_i\times
		B_j)\circ P\subseteq \Hom(M,N).$$ Similarly, one can define symmetrically
		distinguished morphisms in $\M^{atts}_{sd}$.
		Here $\DCH(A_i\times B_j)$ is in the sense of Definition \ref{def:SD} or
		\ref{def:SD2}.\\
		In particular, for any object $M$ in $\M^{ab}_{sd}$ (or $\M^{atts}_{sd}$) and any
		integer $i$,  $\DHom(\1(-i), M)$ is a canonical subgroup of
		$\CH^{i}(M)=\CH^{i}(F(M))$. We denote\footnote{Beware that our notation slightly
			conflicts with the notation of \cite{MR2187148}, where $\DCH^*(X)$ stands for
			the subalgebra generated by divisors, which is denoted by $R(X)$ in the present
			paper.} $$\DCH^{i}(M):=\DHom(\1(-i), M)$$ and call its elements
		\emph{symmetrically distinguished cycles} of $M$. 
	\end{defn}
	
	We collect some basic properties of symmetrically distinguished morphisms in the
	following lemma. Recall that $\pi: \M^{ab}\to \bar{\M^{ab}}$ is the natural
	projection functor (Definition \ref{def:Mab}).
	\begin{lem}[Relation with numerical motives of abelian type]\label{lemma:EquivalenceNum}
		
		In $\M^{ab}_{sd}$,
		\begin{enumerate}[$(i)$]
			\item the composition and the tensor product of two symmetrically distinguished
			morphisms is again symmetrically distinguished. Hence we have a tensor
			subcategory $\left(\M^{ab}_{sd},   \text{     s.d.\,morphisms}\right)$.
			\item For any two objects $M, N\in \M^{ab}_{sd}$, the functor $\pi\circ F$
			induces an isomorphism\,:
			$$\DHom(M, N)\xrightarrow{\isom} \Hom_{\NumM}(\pi(F(M)), \pi(F(N))).$$
			In particular, for any object $M\in \M^{ab}_{sd}$, the composition of the
			natural map $\DCH^{i}(M)\inj \CH^{i}(M)\surj \bar\CH^{i}(M):=\bar\CH^{i}(F(M))$
			is an isomorphism.
			\item The composed functor $\pi\circ F: \M^{ab}_{sd}\to \bar{\M^{ab}}$ induces
			an equivalence of categories $$\bar{F}: \left(\M^{ab}_{sd},   \text{    
				s.d.\,morphisms}\right) \xrightarrow{\isom} \bar{\M^{ab}}.$$
		\end{enumerate}
		Similar properties also hold for the category $\M^{atts}_{sd}$.
	\end{lem}
	\begin{proof}
		$(i)$ is a consequence of Theorem \ref{thm:SD}, which implies that symmetrically
		distinguished cycles on abelian varieties are closed under tensor product and
		symmetrically distinguished correspondences between abelian varieties are closed
		under composition.
		
		For $(ii)$, let $M:=\left(\sqcup_{i=1}^{r}A_{i}, P, (n_{1}, \ldots,
		n_{r})\right)$ and $N:=\left(\sqcup_{j=1}^{s}B_{j}, Q, (m_{1}, \ldots,
		m_{s})\right)$. Then, on the one hand, we have by definition $$\DHom(M,
		N):=Q\circ \bigoplus_{i,j} \DCH^{\dim A_i+m_{j}-n_{i}}(A_i\times B_j)\circ P\,;$$
		and, on the other hand, 
		\begin{eqnarray*}
			\Hom_{\NumM}\left(\pi(F(M)),
			\pi(F(N))\right) &=&\Hom_{\NumM}\left(\bar{P}\left(\oplus\bar{\h}(A_{i})(n_{i})\right),\bar{Q}\left(\oplus\bar{\h}(B_{j})(m_{j})\right)\right)\\
			&=& \bar{Q}\circ \bigoplus_{i,j} \bar{\CH}^{\dim A_i+m_{j}-n_{i}}(A_i\times
			B_j)\circ \bar{P}.
		\end{eqnarray*}
		By Theorem \ref{thm:SD}, for any $1\leq i\leq r, 1\leq j\leq s$,  the natural
		map induced by $\pi\circ F$ $$\DCH^{\dim A_i+m_{j}-n_{i}}(A_i\times
		B_j)\xrightarrow{\isom} \bar{\CH}^{\dim A_i+m_{j}-n_{i}}(A_i\times B_j)$$ is an
		isomorphism. Now the fact that $P$ and $Q$ are matrices of symmetrically
		distinguished cycles allows us to conclude.
		
		For $(iii)$, the full faithfulness is the content of $(ii)$ while the essential
		surjectivity follows from that of $F$ (Lemma \ref{lemma:Equivalence}) and
		$\pi$.
		
		The same argument also works for the category $\M^{atts}_{sd}$ by using Theorem
		\ref{thm:SD2} in place of Theorem \ref{thm:SD}.
	\end{proof}
	
	\begin{rmk}\label{rmk:Section}
		Lemma \ref{lemma:Equivalence} and Lemma \ref{lemma:EquivalenceNum} are
		contrasting\,: on the one hand, the whole category $\M^{ab}_{sd}$ is equivalent
		to $\M^{ab}$, the category of abelian Chow motives\,; on the other hand, the
		subcategory with same objects as $\M_{sd}^{ab}$ and with symmetrically distinguished
		morphisms (Definition \ref{def:SDmorph}) is equivalent to $\bar{\M^{ab}}$, the
		category of abelian numerical motives. Thus $\M^{ab}_{sd}$ fulfills exactly our
		purpose to make a bridge between $\M^{ab}$ and $\bar{\M^{ab}}$. More precisely,
		we have the commutative diagram
		\begin{equation*}
		\xymatrix{
			\M^{ab}_{sd} \ar[r]^{F}_{\isom}& \M^{ab} \ar@{->>}[d]^{\pi}\\
			(\M^{ab}_{sd}, \text{ s.d.\,morph}) \ar@{^{(}->}[u] \ar[r]^-{\isom}_-{\bar F}&
			\bar{\M^{ab}}
		}
		\end{equation*}
		which gives an explicit way to understand O'Sullivan's categorical Theorem
		\ref{thm:section} via his more down-to-earth Theorem \ref{thm:SD}. Namely, we no
		longer deal with the right-inverse tensor functor $T$, whose existence is proven in a
		somehow abstract way and whose uniqueness is up to a tensor conjugacy, but
		instead we have, via the equivalences $F$ and $\bar F$, a concrete subcategory
		of symmetrically distinguished morphisms inside $\M^{ab}_{sd}$, which plays the
		role of the section functor $T$. We think the construction and basic properties
		of $\M^{ab}_{sd}$ and its subcategory of symmetrically distinguished morphisms
		would have independent interest in the future study of algebraic cycles on
		abelian varieties, or more generally, varieties with motives of abelian type.
	\end{rmk}

	Finally, let us note the following simple consequence of Lemma \ref{lemma:EquivalenceNum} $(iii)$ , which will be crucial when dealing
	with quotients (or more generally, generically finite surjective morphisms) in
	\S \ref{subsect:quotient}.
	\begin{lem}\label{lemma:PsAb}
		The category $(\M^{ab}_{sd}, \text{  s.d.\,morphisms})$, with objects as in
		$\M^{ab}_{sd}$ but with morphisms restricted to symmetrically distinguished
		morphisms, is pseudo-abelian.
	\end{lem}
	\begin{proof}
	This follows from the equivalence of categories in Lemma \ref{lemma:EquivalenceNum}$(iii)$ and the fact that $\bar{\M^{ab}}$ is pseudo-abelian.
%	
%		Let $M:=\left(A_{1}\sqcup\cdots \sqcup A_{r}, P, (n_{1}, \ldots, n_{r})\right)$
%		be an object in $\M^{ab}_{sd}$, and let $f\in \End(M)$ be an idempotent
%		symmetrically distinguished endomorphism, \emph{i.e.} $f\in \DHom(M,M)$ and $f^{2}=f$. By definition,
%		there exists an $(r\times r)$-matrix $Q=\left(q_{i,j}\right)$
%		%_{1\leq i,j\leq r}$ 
%		with $q_{i,j}\in \DCH^{\dim A_{i}+n_{j}-n_{i}}(A_{i}\times A_{j})$, such that
%		$f=P\circ Q\circ P$. Therefore $(P\circ Q\circ P)^{\circ 2}=P\circ Q\circ P$ is
%		a idempotent matrix with symmetrically distinguished entries. We thus have the
%		canonical image of $f\in \End(M)$ given by $\left(A_{1}\sqcup\cdots \sqcup
%		A_{r}, P\circ Q\circ P, (n_{1}, \ldots, n_{r})\right)$.
	\end{proof}
	
	\begin{rmk}\label{rmk:Envelope}
		In fact, the category $(\M^{ab}_{sd}, \text{  s.d.\,morphisms})$ is the
		pseudo-abelian additive envelop of the category $Corr^{ab}_{sd}$ of
		symmetrically distinguished correspondences between abelian varieties\,: more
		precisely, an object of $Corr^{ab}_{sd}$ is a couple $(A, n)$ with $A$ an abelian variety (with
		fixed origin) and $n\in \Z$ and morphisms between two objects $(A,n), (B,m)$ are
		given by
		$$\Hom((A,n), (B, m)):= \DCH^{\dim A+m-n}(A\times B)\,;$$ and the composition is
		the usual one for correspondences.
	\end{rmk}

	%For later reference, let us state the following simple fact.
	%\begin{lem}\label{lemma:SDMorphProd}
	%In the category $\M^{ab}_{sd}$ or $\M^{atts}_{sd}$, the composition and the
	%tensor product of two symmetrically distinguished morphisms are still
	%symmetrically distinguished.
	%\end{lem}
	%\begin{proof}
	%This is an immediate consequence of Theorem \ref{thm:SD} (and its extension
	%Theorem \ref{thm:SD2}) which says that symmetrically distinguished cycles are
	%closed under tensor product, intersection product and pull-back and
	%push-forward along homomorphisms of abelian varieties or abelian torsors with
	%torsion structures.
	%\end{proof}

	\section{Distinguished cycles} \label{sect:distinguish}
	%With the theory of symmetrically distinguished cycles and morphisms being
	%recalled in \S \ref{sect:SD}, we can say that the philosophy of the paper is to
	%study the Chow rings of algebraic varieties with motive of abelian type
	
	\subsection{Definitions and basic properties}
	%To introduce the notion of distinguished cycles for varieties whose motive is
	%of abelian type, we need to consider the image of the section functor $T$ in
	%Theorem \ref{thm:section}, which is formalized in the following\,:

	%First we need to extend the notion of symmetrically distinguished cycles to
	%motives which are direct summands of motives of abelian varieties cut out by
	%symmetrically distinguished projectors, and their direct sums.

	%\begin{defn}[Symmetrically distinguished morphisms]\label{def:SDMotive}
	%	Let $A_i$ and $B_j$ be abelian varieties (or more generally a.t.t.s's), let
	%$p_i \in \DCH(A_i \times A_i)$ and $q_j \in \DCH(B_j\times B_j)$ be projectors
	%that are symmetrically distinguished in the sense of  Definition \ref{def:SD}
	%or \ref{def:SD2}, and let $n_i$ and $m_j$ be integers. Let $M$ and $N$ in
	%$\CHM$ be the following two motives\,:
	%	$$M = \bigoplus_i (A_i,p_i,n_i) \quad \text{and} \quad N = \bigoplus_j
	%(B_j,q_j,m_j).$$
	%	The space of \emph{symmetrically distinguished morphisms} from $M$ to $N$,
	%denoted
	%	 $\DCH(M,N)$,  
	%	is the sub-vector space of $\Hom_{\CHM}(M,N) := \bigoplus_{i,j} q_j\circ
	%\CH(A_i\times B_j)\circ p_i$, consisting of elements in $\bigoplus_{i,j}
	%q_j\circ \DCH(A_i\times B_j)\circ p_i$, where $\DCH(A_i\times B_j)$ is in the
	%sense of Definition \ref{def:SD} or \ref{def:SD2}.
	%	 Note that by Theorem \ref{thm:SD} or \ref{thm:SD2}, this subspace is mapped
	%isomorphically to $\Hom_{\NumM}(M, N)$. 
	%	
	%	
	%\end{defn}

	Here come the key notions of this paper\,:
	
	\begin{defn}[Marking]\label{def:marking}
		Let $X$ be smooth projective variety such that its Chow motive $\h(X)$ belongs
		to $\M^{ab}$.  
		A \emph{marking} for $X$ consists of an object  $M\in \M^{ab}_{sd}$ together
		with an isomorphism $$\phi: \h(X)\xrightarrow{\isom} F(M) \quad \text{in }
		\CHM,$$ where $F:\M^{ab}_{sd}\xrightarrow{\isom}\M^{ab}$ is the equivalence in Definition \ref{def:SDMotive}.
		%\begin{itemize}
		%\item an object $M = \bigoplus_{i} (A_i,p_i,n_i) \in \M^{ab}_{sd}$, together
		%with
		%\item  a correspondence $\phi \in \bigoplus_i p_i\circ \CH^{\dim X
		%+n_i}(X\times A_i)$ inducing
		%an isomorphism $$\phi: \h(X)\xrightarrow{\isom} F(M) \quad \text{in } \CHM,$$
		%where $F:\M^{ab}_{sd}\to \M^{ab}$ is the functor introduced in Definition
		%\ref{def:SDMotive}.
		%\end{itemize}
	\end{defn}
	By Lemma \ref{lemma:Equivalence}, a marking for a smooth projective variety $X$
	with motive of abelian type always exists. In practice, starting from
	Section \ref{sect:examples}, we will abusively ignore the difference between
	$\M^{ab}_{sd}$ and its image in $\M^{ab}$ by $F$ and write a marking as an
	isomorphism $\phi: \h(X)\xrightarrow{\isom} M$ for $M\in \M^{ab}_{sd}$.

	\begin{defn}[Distinguished cycles]\label{def:distinguished}
		Let $X$ be a smooth projective variety such that its Chow motive $\h(X)$
		belongs to $\M^{ab}$.  
		Given a marking $\phi: \h(X)\xrightarrow{\isom} F(M)$ with $M\in \M^{ab}_{sd}$,
		we define the subgroup of \emph{distinguished cycles} of codimension~$i$ of $X$,
		denoted by $\DCH^i_{\phi}(X)$, or sometimes $\DCH^i(X)$ if $\phi$ is clear from
		the context, to be the pre-image of $\DCH^i(M)$ (see Definition
		\ref{def:SDmorph}) \emph{via} the induced isomorphism
		$\phi_{*}:\CH^i(X)\xrightarrow{\isom}\CH^i(M)$.
	\end{defn}
	
	Almost by construction, we have\,:
	\begin{lem}\label{lemma:section}
		For any smooth projective variety $X$ such that $\h(X)\in \M^{ab}$ and any
		marking $\phi: \h(X)\xrightarrow{\isom} F(M)$ with $M\in \M^{ab}_{sd}$, the
		composition $$\DCH^i_{\phi}(X)\inj \CH^i(X)\surj \bar\CH^i(X)$$ is an
		isomorphism. In other words, $\phi$ provides a section (as graded vector spaces)
		of the natural projection $\CH(X)\surj \bar\CH(X)$.
	\end{lem}
	\begin{proof}
		In the commutative diagram
		\begin{equation*}
		\xymatrix{
			\DCH^i_{\phi}(X) \ar@{^{(}->}[r] \ar[d]_{\isom}^{\phi_{*}}&
			\CH^i(X)\ar@{->>}[r]\ar[d]_{\isom}^{\phi_{*}} &\bar\CH^i(X)
			\ar[d]_{\isom}^{\bar\phi_{*}}\\
			\DCH^i(M) \ar@{^{(}->}[r]& \CH^i(M)\ar@{->>}[r] &\bar\CH^i(M)
		}
		\end{equation*}
		the composition of the bottom line is an isomorphism by Lemma
		\ref{lemma:EquivalenceNum}. Therefore the composition of the top line is also an
		isomorphism, hence $\DCH^i_{\phi}(X)$ provides a section.
	\end{proof}

	\begin{rmk}[Fundamental class]\label{rmk:FundCl}
%		The condition $(\starC)$ above says in particular that the fundamental class
%		$\1_{X}$, which is $c_{0}(X)$, is distinguished. However, we want to point out
%		that 
%\comm{J'ai deplace cette remarque.}
Given a smooth projective variety $X$, its fundamental class $\1_{X}$ is always distinguished for any choice of marking. Indeed, we can
		assume that $X$ is connected, thus $\CH^{0}(X)=\Q\cdot\1_{X}$, and Lemma
		\ref{lemma:section} ensures that $\1_X$ is distinguished.
	\end{rmk}

	Distinguished cycles behave well with respect to tensor products and projections\,:
	\begin{prop}[Tensor products and projections]\label{prop:tensor}
		Let $X, Y$ be two smooth projective varieties with motive of abelian type,
		endowed with markings $\phi: \h(X)\xrightarrow{\isom} F(M)$ and $\psi:
		\h(Y)\xrightarrow{\isom} F(N)$. Then $$\phi\otimes \psi :\h(X\times Y)
		\xrightarrow{\isom} F(M\otimes N)$$ provides a marking for $X\times Y$, and the
		exterior product $\CH^i(X)\times \CH^j(Y)\xrightarrow{\otimes} \CH^{i+j}(X\times
		Y)$ respects distinguished cycles\,:
		$$\DCH^i_{\phi}(X)\times \DCH^j_{\psi}(Y)\xrightarrow{\otimes}
		\DCH^{i+j}_{\phi\otimes \psi}(X\times Y).$$
		Moreover, denoting $p:X\times Y\to X$ the natural projection, we have
		$$p_*\DCH_{\phi\otimes \psi}^i(X\times Y) \subseteq \DCH_{\phi}^{i-\dim Y}(X) \quad \mbox{and}\quad
		p^* \DCH_{\phi}^{i}(X) \subseteq \DCH_{\phi\otimes \psi}^i(X\times Y),$$
		and similarly for the natural projection $q: X\times Y\to Y$.  
		%\comm{J'ai ajoute ce resultat sur les projections, qui est utilise pour prouver que la diagonale est distinguee.}
	\end{prop}
	\begin{proof}
		That $\phi \otimes \psi$ provides a marking for $X\times Y$ such that exterior product respects distinguished cycles follows directly from Lemma
		\ref{lemma:EquivalenceNum}$(i)$, which says that the tensor product of two
		symmetrically distinguished morphisms is symmetrically distinguished.
		%	Theorem \ref{thm:SD} ensures that the exterior product of two symmetrically
		%distinguished cycles is symmetrically distinguished.
		%Therefore, if $A$ and $B$ are two abelian varieties, $p\in \DCH_{\dim
		%A}(A\times A)$, $q\in \DCH_{\dim B}(B\times B)$ are symmetrically distinguished
		%projectors, then, since $(A, p, m)\otimes (B, q, n)=(A\times B, p\otimes q,
		%m+n)$,  $\phi \otimes \psi$ provides a marking for $X\times Y$. Likewise, the
		%exterior product of two cycles in $\DCH^i_{\phi}(X)$ and $\DCH^j_{\psi}(Y)$,
		%respectively, provide a cycle in $\DCH^{i+j}_{\phi\otimes \psi}(X\times Y)$.
		To see that push-forwards and pull-backs along projections respect distinguished cycles,  
		it is enough, by Lemma
		\ref{lemma:EquivalenceNum}$(i)$, to see that $ \id_{M} \otimes f: M\otimes
		N \to M\otimes \1(-d)$ is a symmetrically distinguished
		morphism (Definition \ref{def:SDmorph}), where $d:=\dim Y$ and $f:N\to \1(-d)$
		is induced by the morphism $\h(Y)\to \1(-d)$ determined by the fundamental class
		of $Y$. By Lemma \ref{lemma:EquivalenceNum}$(i)$, we only have to see that $f$
		is a symmetrically distinguished morphism, which is explained in Remark~\ref{rmk:FundCl}. 
	\end{proof}

	\subsection{The main questions and the key condition $(\star)$}
	
	\begin{ques}\label{Question}
		Here are the most important properties of the distinguished cycles that we are
		going to investigate\,:
		\begin{itemize}
			\item When does $\bigoplus_i\DCH^i_{\phi}(X)$ form a (graded) $\Q$-subalgebra
			of $\CH(X)$ ?
			\item When do the Chern classes of $X$ belong to $\bigoplus_i\DCH^i_{\phi}(X)$
			?
		\end{itemize}
	\end{ques}

	To this end, let us introduce the following condition for smooth projective
	varieties whose Chow motive is of abelian type\,:
	\begin{defn}\label{def:Star}
		We say that a smooth projective variety $X$ with $\h(X)\in \M^{ab}$ satisfies
		the condition $(\star)$ if\,:\\ There exists a marking $\phi:
		\h(X)\xrightarrow{\isom} F(M)$ (with $M\in \M^{ab}_{sd}$) such that 
		\begin{itemize}
			%\item[$(\star1)$] \textbf{(Auto-duality)} the diagonal $\Delta_{X}$ belongs to
			%$\DCH_{\phi^{\otimes 2}}(X^{2})$, 
			\item[$(\starM)$] \textbf{(Multiplicativity)} the small diagonal $\delta_{X}$
			belongs to $\DCH_{\phi^{\otimes 3}}(X^{3})$, that is, under the induced
			isomorphism $\phi^{\otimes 3}_{*}: \CH(X^{3})\xrightarrow{\isom} \CH(M^{\otimes
				3})$, the image of $\delta_{X}$ is symmetrically distinguished, \emph{i.e.}, in
			$\DCH(M^{\otimes 3})$\,;
			\item[$(\starC)$] \textbf{(Chern classes)} all Chern classes of $T_{X}$ belong
			to $\DCH_{\phi}(X)$.
		\end{itemize}
		More generally, if $X$ is a smooth projective variety equipped with the action
		of a finite group $G$, we say that $(X,G)$ satisfies $(\star)$ if there exists a
		marking $\phi: \h(X)\xrightarrow{\isom} F(M)$ that satisfies, in addition to
		$(\starM)$ and $(\starC)$ above\,:
		\begin{itemize}
			\item[$(\star_G)$] \textbf{($G$-invariance)} the graph $g_X$ of $g: X\to X$
			belongs to $\DCH_{\phi^{\otimes 2}}(X^{2})$ for any $g\in G$. 
		\end{itemize}
		%Here $g_{X}\in \CH(X\times X)$ denotes the graph of $g: X\to X$.
	\end{defn}
	
	We will see in Corollary \ref{cor:TopChern} that the condition $(\starM)$ implies that the \emph{top} Chern class of $T_{X}$ is distinguished.	
	
	\begin{lem}[Diagonal]\label{lemma:Diagonal}
	Notation is as before.
		\begin{enumerate}[$(i)$]
			\item The condition $(\starM)$ implies that the diagonal $\Delta_{X}$ belongs to
			$\DCH_{\phi^{\otimes 2}}(X^{2})$.
			\item The condition that $\Delta_{X}$ is distinguished is equivalent to saying
			that the isomorphism $\sigma: M\xrightarrow{\isom} M^{\vee}(-d_{X})$, given by
			the commutativity of the following diagram\footnote{Recall that $F$ is an
				equivalence (Lemma \ref{lemma:Equivalence}), so $F(\sigma)$ determines
				$\sigma$.}, is symmetrically distinguished in the sense of Definition
			\ref{def:SDmorph}, where the top morphism is the Poincar\'e duality in $\CHM$
			(induced by $\Delta_{X}$).
			\begin{equation}\label{diag:star1}
			\xymatrix{
				\h(X)\ar[d]^{\phi}_{\isom}\ar[r]^-{PD_{X}}_-{\isom} & \h(X)^{\vee}(-d_{X})   \\
				F(M)\ar[r]_-{F(\sigma)}^-{\isom}&F(M^{\vee}(-d_{X}))
				\ar[u]_{\phi^{\vee}(-d_{X})}^{\isom} 
			}
			\end{equation}
		\end{enumerate}
		
		% Nevertheless, we prefer to keep $(\star1)$ in Definition \ref{def:Star},
		%because in practice to prove $(\starM)$, we always first prove $(\star1)$ as a
		%preparation step and use the interpretation of $(\starM)$ in the following
		%Lemma \ref{lemma:InterpretationStar}, which is only available after $(\star1)$
		%is proven\,!
	\end{lem}
	\begin{proof}
		Statement $(ii)$ is tautological, and statement $(i)$ follows from Proposition~\ref{prop:tensor} together with the observation that
		 $\Delta_{X}$ is the push-forward of $\delta_{X}$ along the projection $\pr_{1,2}:
		X\times X\times X\to X\times X$.
%		 Via the marking, it is enough, by Lemma
%		\ref{lemma:EquivalenceNum}$(i)$, to see that $\id\otimes \id\otimes f: M\otimes
%		M\otimes M\to M\otimes M\otimes \1(-d)$ is a symmetrically distinguished
%		morphism (Definition \ref{def:SDmorph}), where $d:=\dim X$ and $f:M\to \1(-d)$
%		is induced by the morphism $\h(X)\to \1(-d)$ determined by the fundamental class
%		of $X$. By Lemma \ref{lemma:EquivalenceNum}$(i)$, we only have to see that $f$
%		is a symmetrically distinguished morphism, which is explained in Remark
%		\ref{rmk:FundCl}. 
%\comm{J'ai modifie la preuve pour prendre en compte la nouvelle version de la Proposition \ref{prop:tensor}.}
	\end{proof}

	%
	%\begin{lem}\label{lemma:1-dim}
	%Let $A$ be an abelian variety of dimension $g$, $p\in\CH^{g}(A\times A)$ be a
	%symmetrically distinguished projector and $n$ an integer. If the image $p_{*}:
	%\CH^{n}(A)\to \CH^{n}(A)$ is a 1-dimensional $\Q$-vector space, then it is
	%generated by a symmetrically distinguished element. 
	%\end{lem}
	%\begin{proof}
	%\comm{Je crois que c'est vrai mais je ne sais pas le d\'emontrer...}
	%\end{proof}

	%The following interpretations of the conditions in Definition \ref{def:Star}
	%will be useful\,:
	\begin{lem}[Equivalent formulation of
		$(\starM)$]\label{lemma:InterpretationStar}
		Let $\phi: \h(X)\xrightarrow{\isom} F(M)$ be a marking as above and $d_{X}$ be the
		dimension of $X$.
		%\begin{enumerate}[(a)]
		%\item The condition that the diagonal $\Delta_{X}$ belongs to
		%$\DCH_{\phi^{\otimes 2}}(X^{2})$, is equivalent to saying that the isomorphism
		%$\sigma$, given by the commutativity of the following diagram, is symmetrically
		%distinguished in the sense of Definition \ref{def:SDmorph}, where the top
		%morphism is the Poincar\'e duality in $\CHM$ (induced by $\Delta_{X}$).
		%\begin{equation*}%\label{diag:star}
		%\xymatrix{
		%\h(X)(d_{X})\ar[d]^{\phi(d_{X})}_{\isom}\ar[r]^{PD_{X}}_{\isom} & \h(X)^{\vee}
		%  \\
		% M(d_{X})\ar[r]_-{\sigma}^-{\isom}&M^{\vee} \ar[u]_{\phi^{\vee}}^{\isom} 
		%}
		%\end{equation*}
		%\item 
		The condition $(\starM)$ is equivalent to saying that the morphism $\mu:
		M^{\otimes 2}\to M$, determined by the commutativity of the following
		diagram\footnote{Recall that $F$ is an equivalence (Lemma
			\ref{lemma:Equivalence}), so $F(\mu)$ determines $\mu$.}, is a symmetrically
		distinguished morphism, where the top morphism is the intersection product in
		$\CHM$ induced by the small diagonal.
		\begin{equation}\label{diag:starM}
		\xymatrix{
			\h(X)^{\otimes 2}\ar[d]_{\phi^{\otimes 2}}^{\isom} \ar[r]^{\delta_{X}} &
			\h(X)\ar[d]^{\phi}_{\isom}\\
			F(M^{\otimes 2}) \ar[r]_{F(\mu)}& F(M)
		}
		\end{equation}
		%\end{enumerate}
	\end{lem}
	\begin{proof}
		First we claim\footnote{We thank Peter O'Sullivan for mentioning this to us.}
		that the condition that $\mu$ is symmetrically distinguished implies that
		$\sigma$ in Diagram (\ref{diag:star1}) is symmetrically distinguished (or
		equivalently, $\Delta_{X}$ is distinguished by Lemma \ref{lemma:Diagonal}$(ii)$). Indeed, 
		consider the commutative diagram 
		\begin{equation*}
		\xymatrix{
			\h(X)^{\otimes 2}\ar[d]_{\phi^{\otimes 2}}^{\isom} \ar[r]^{\delta_{X}} &
			\h(X)\ar[d]^{\phi}_{\isom}  \ar[r]^-{\1_{X}}& \1(-d_{X})\ar@{=}[d]\\
			F(M^{\otimes 2}) \ar[r]_{F(\mu)}& F(M) \ar[r]_-{F(\nu)}& F(\1(-d_{X}))
		}
		\end{equation*}
		where the left square is (\ref{diag:starM}), the top right morphism is induced
		by the fundamental class of $X$ and $\nu$ is the morphism determined by the
		commutativity of the right square. By Remark \ref{rmk:FundCl}, $\nu$ is a
		symmetrically distinguished morphism. Now the outer square of the previous
		diagram gives the right square in the following diagram
		\begin{equation*}
		\xymatrix{
			\h(X)  \ar[d]_{\phi}^{\isom} \ar[rr]^-{\eta_{\h(X)}\otimes
				\id_{\h(X)}}&&\h(X)^{\vee}\otimes \h(X)^{\otimes
				2}\ar[d]_{(\phi^{\vee})^{-1}\otimes\phi^{\otimes 2}}^{\isom}
			\ar[rr]^-{\id_{h(X)^{\vee}}\otimes (\1_{X}\circ\delta_{X})} &&
			\h(X)^{\vee}(-d_{X}) \\
			F(M)\ar[rr]_-{F(\eta_{M}\otimes \id_{M})} && F(M^{\vee}\otimes M^{\otimes 2})
			\ar[rr]_{F(\id_{M^{\vee}}\otimes (\nu \circ \mu))}&& F(M^{\vee})(-d_{X})
			\ar[u]^{\phi^{\vee}(-d_{X})}_{\isom}
		}
		\end{equation*}
		where in the left square, $\eta_{M}: \1\to M^{\vee}\otimes M$ is the unit of the
		duality for $M$ and similarly for $\h(X)$.
		Therefore, by definition, the isomorphism $\sigma$ in Diagram (\ref{diag:star1})
		is given by $$\sigma=\left(\id_{M^{\vee}}\otimes (\nu\circ \mu)\right)\circ
		\left(\eta_{M}\otimes \id_{M}\right).$$ As $\mu$, $\nu$ and $\eta_{M}$ are all
		symmetrically distinguished morphisms, so is $\sigma$ by Lemma~\ref{lemma:EquivalenceNum}$(i)$.
		
		Now let us show the equivalence between $(\starM)$ and the symmetric
		distinguishedness of $\mu$. Thanks to the above Claim and to Lemma
		\ref{lemma:Diagonal}, for both directions of implication one can suppose that
		$\sigma$ is symmetrically distinguished. Thus the following isomorphism,
		induced by composing with $\sigma^{\otimes 2}\otimes \id_{M}$, preserves the
		symmetrically distinguished elements\,:
		$$\CH^{2d_{X}}(M^{\otimes 3})=\Hom(\1, M(d_{X})^{\otimes 2}\otimes
		M)\xrightarrow{\isom}\Hom(\1, (M^{\vee})^{\otimes 2}\otimes M)=\Hom(M^{\otimes
			2}, M).$$
		We can conclude by observing that this isomorphism sends $\phi^{\otimes
			3}_{*}(\delta_{X})$ to $\mu$.
	\end{proof}
	
	Let us also mention the following convenient sufficient condition for
	$(\star_{G})$\,:
	\begin{lem}[$G$-invariant marking]\label{lemma:StarG}
		Let $X$ be a smooth projective variety endowed with an action of a finite group
		$G$. Let $\phi: \h(X)\lra{\isom} F(M)$ be a marking as above. If $\Delta_{X}$ is
		distinguished 
		and if for any $g\in G$, we have $\phi\circ g=F(\overline{g})\circ \phi$ for some symmetrically distinguished cycle $\bar g$, then $\phi$ satisfies
		$(\star_{G})$, where $g:\h(X)\to\h(X)$ is the automorphism induced by $g$.
	\end{lem}
	\begin{proof}
		For any $g\in G$, consider the composition $$\1(-\dim X)\lra{\Delta_{X}}
		\h(X)\otimes \h(X)\lra{\id\otimes g}\h(X)\otimes \h(X)\lra{\phi\otimes \phi}
		M\otimes M.$$
		We obtain that $(\phi\otimes\phi)_* \Gamma_{g}=(\phi\otimes\phi)\circ
		(\id\otimes g)\circ\Delta_{X}=F(\id \otimes \bar g)\circ (\phi\otimes \phi)\circ \Delta_{X}$, where the latter term is
		symmetrically distinguished from the  assumption on $\Delta_X$. This means exactly that the graph
		$\Gamma_{g}$ is distinguished.
	\end{proof}

	\begin{rmk}[Another formulation]
		%	The following interpretation of the conditions in $(\star)$ using the section
		%$\tilde T$ in Theorem \ref{thm:section} was kindly suggested to us by Peter
		%O'Sullivan\,; this is partially indicated in the concluding remarks of his
		%paper \cite[\S 6.3]{MR2795752}. 
		%	For an algebraic variety $X$ with $\h(X)\in \M^{ab}$, the existence of a
		%marking satisfying $(\star1)$ and $(\starM)$ is equivalent to the existence of
		%an isomorphism of algebra objects $$\varphi: \h(X)\xrightarrow{\simeq}\tilde
		%T(\bar \h(X)),$$ such that the induced morphism \emph{via} Poincar\'e duality
		%$$\varphi^{\vee}: \tilde T\left(\bar \h(X)^{\vee}\right)\xrightarrow{\simeq}
		%\h(X)^{\vee}$$ is the inverse of $\varphi\otimes \id_{\1(\dim X)}$. As such an
		%isomorphism induces a section of the epimorphism $\CH(X)\to \bar\CH(X)$, we can
		%also translate the condition $(\starC)$ as the Chern classes belong to the
		%image of the section. Similarly, in the presence of a $G$-action, the condition
		%$(\star_{G})$ can be spelled out by its graphs. 
		The following interpretation of the condition $(\starM)$ using the section $T$
		in Theorem \ref{thm:section} was kindly suggested to us by Peter
		O'Sullivan\footnote{
			The condition that $\h(X)\in \M^{ab}$ corresponds to the condition $X\in
			\mathscr V^{0}$ in \cite[\S 6.3]{MR2795752} and for such $X$, the existence of a
			marking satisfying $(\starM)$ corresponds to the condition $X\in \mathscr
			V^{00}$ in \cite[\S 6.3]{MR2795752}.}.
		%\,; this is partially indicated in the concluding remarks of his paper \cite[\S
		%6.3]{MR2795752}. 
		For an algebraic variety $X$ with $\h(X)\in \M^{ab}$, the existence of a marking
		satisfying  $(\starM)$ is equivalent to the existence of an isomorphism of
		algebra objects $$\varphi: \h(X)\xrightarrow{\simeq} T(\bar \h(X)).$$ 
		%Such a marking satisfies $(\star1)$ if and only if  the induced morphism
		%\emph{via} Poincar\'e duality $$\varphi^{\vee}: T\left(\bar
		%\h(X)^{\vee}\right)\xrightarrow{\simeq} \h(X)^{\vee}$$ is the inverse of
		%$\varphi\otimes \id_{\1(\dim X)}$. 
		As such an isomorphism induces a section of the epimorphism $\CH(X)\to
		\bar\CH(X)$. The condition $(\starC)$ can be translated into saying that the
		Chern classes belong to the image of the section. Similarly, in the presence of
		a $G$-action, the condition $(\star_{G})$ can be spelled out by its graphs.

		This formulation of $(\star)$ has the obvious advantage of being both natural and intrinsic.
		However, to work out examples, which is the main objective of this paper, as
		well as to prove theorems in practice (\S \ref{sect:examples},
		\S\ref{sect:Examples}), we find it more convenient to stick to Definition
		\ref{def:Star} together with its interpretation given in Lemma
		\ref{lemma:InterpretationStar}.
	\end{rmk}
	
	The motivation to study the condition $(\star)$ is the following\,:
	\begin{prop}[Subalgebra]\label{prop:Disting}
		Let $X$ be a smooth projective variety with motive of abelian type.
		If $X$ satisfies the condition $(\starM)$, then there is a section, as
		\emph{graded algebras}, for the natural surjective morphism $\CH(X)\surj
		\bar\CH(X)$. If moreover $(\starC)$ is satisfied, then all Chern classes of $X$
		are in the image of this section.
		
		In other words, under $(\star)$, we have a graded $\Q$-subalgebra $\DCH(X)$ of
		the Chow ring $\CH(X)$, which contains all the Chern classes of $X$ and is
		mapped isomorphically to $\bar\CH(X)$. We call elements of $\DCH(X)$
		\emph{distinguished cycles} of $X$.
	\end{prop}
	\begin{proof}
		Let $\phi: \h(X)\xrightarrow{\isom} F(M)$ be a marking, where $M\in \M^{ab}_{sd}$.
		If $\phi$ satisfies $(\star)$, then we define $\DCH(X):=\DCH_{\phi}(X)$ as in
		Definition \ref{def:distinguished}, and this provides a section to the epimorphism
		$\CH(X)\surj \bar\CH(X)$ as graded vector spaces by Lemma~\ref{lemma:section}.
		To show that it provides a section as \emph{algebras}, one has to show that
		$\DCH_{\phi}(X)$ is closed under the intersection product of $X$ (the unit
		$\1_{X}$ is automatically distinguished by Remark \ref{rmk:FundCl}). Let
		$\alpha\in \DCH^{i}_{\phi}(X)$ and $\beta\in \DCH^{j}_{\phi}(X)$. Then by
		definition the morphisms $\phi\circ\alpha:\1(-i)\to F(M)$ and
		$\phi\circ\beta:\1(-j)\to F(M)$ determine symmetrically distinguished morphisms.
		By Lemma \ref{lemma:EquivalenceNum}$(i)$, $(\phi^{\otimes 2})\circ
		(\alpha\otimes \beta)=(\phi\circ\alpha)\otimes(\phi\circ\beta): \1(-i-j)\to
		F(M^{\otimes 2})$ also determines a symmetrically distinguished morphism. 
		\begin{equation*}
		\xymatrix{
			\1(-i-j) \ar[r]^{\alpha \otimes \beta}\ar[dr] &\h(X)^{\otimes
				2}\ar[r]^{\delta_{X}} \ar[d]^{\phi^{\otimes 2}}_{\isom}&
			\h(X)\ar[d]^{\phi}_{\isom}\\
			& F(M^{\otimes 2}) \ar[r]^{F(\mu)} & F(M)
		}
		\end{equation*}
		Condition $(\star)$ implies that $\mu$, which is determined by the above
		commutative diagram, is a symmetrically distinguished morphism. Therefore, the
	 composition $\phi\circ\delta_{X}\circ(\alpha\otimes \beta)$ in the above
		diagram determines a symmetrically distinguished morphism, which means that
		$\alpha\cdot\beta=\delta_{X,*}(\alpha\otimes \beta)$ is in $\DCH_{\phi}(X)$. The
		assertion concerning Chern classes is tautological.
	\end{proof}
	
	We deduce that the condition $(\starM)$ actually already implies all the
	analogous statements for all sorts of diagonals on higher powers (note the
	analogy with \cite[Proposition 8.7$(iii)$]{SV} in the context of self-dual
	multiplicative Chow--K\"unneth decompositions)\,:
	\begin{cor}[Other diagonals]\label{C:diag}
		Let $X$ be a smooth projective variety with $\h(X)\in \M^{ab}$. If $X$ satisfies
		the condition $(\starM)$, then all the classes of the partial diagonals\footnote{A \emph{partial
				diagonal} of a self-product $X^{n}$ is a subvariety of the form $\{(x_{1},
			\cdots, x_{n}) \in X^{n}~\vert~~ x_{i}=x_{j}~~~ \text{for all}~~~ i\sim j \}$
			for an equivalence relation $\sim$ on $\{1, \cdots, n\}$.} in a self-product of
		$X$ are distinguished.
	\end{cor}
	\begin{proof}
		Let us fix a marking $\phi: \h(X)\xrightarrow{\isom} F(M)$ satisfying the
		condition $(\starM)$ and write $\DCH$ for $\DCH_{\phi^{\otimes ?}}$. Observe
		that any partial diagonal can be written as the intersection product of several
		big diagonals\footnote{A \emph{big diagonal} of a self-product $X^{n}$ is a
			subvariety of the form $\{(x_{1}, \cdots, x_{n}) \in X^{n}~\vert~~
			x_{i}=x_{j}\}$ for some $1\leq i\neq j\leq n$.}. By Proposition
		\ref{prop:Disting}, we only have to show that any big diagonal of a self-product
		is distinguished. However, a big diagonal is the exterior product of the distinguished
		class $\Delta_{X}\in \DCH(X\times X)$ (by Lemma \ref{lemma:Diagonal}) with copies of the fundamental class
		$\1_{X}\in \DCH(X)$ (see Remark \ref{rmk:FundCl}), and is henceforth distinguished,
		thanks to Proposition \ref{prop:tensor}.
	\end{proof}

\subsection{Distinguished morphisms and distinguished correspondences}

%\comm{Nouvelle section.}

\begin{defn}[Distinguished morphisms and distinguished correspondences]\label{def:DisCor}
	Let $X$ and $Y$ be two smooth projective varieties equipped respectively with markings 
	$\phi: \h(X)\xrightarrow{\isom} F(M)$ and $\psi: \h(Y)\xrightarrow{\isom} F(N)$ with  
	$M,N \in \M^{ab}_{sd}$. A correspondence $\Gamma \in \CH(X\times Y)$
	 is said to be \emph{distinguished} if it is distinguished with respect to the product marking on $X\times Y$, \emph{i.e.}, $\Gamma \in \DCH_{\phi\otimes \psi} (X\times Y)$ or equivalently the morphism $(\phi\otimes \psi)(\Gamma) : M\to N$ is symmetrically distinguished in the sense of Definition \ref{def:SDmorph}. A morphism $f: X \to Y$ is said to be \emph{distinguished} if its graph belongs to $\DCH_{\phi\otimes \psi} (X\times Y)$.
\end{defn}
	
	The notion of distinguished morphisms and distinguished correspondences is only really relevant in the case where the markings satisfy the condition $(\starM)$\,:
	
	\begin{prop}\label{prop:distmorphism}
	Let $X$, $Y$ and $Z$ be smooth projective varieties equipped with markings 
that satisfy $(\starM)$, and let $\Gamma \in \DCH(X\times Y)$ and $\Gamma' \in \DCH(Y\times Z)$ be distinguished correspondences. Then
\begin{enumerate}[(i)]
	\item $\Gamma_*\DCH(X) \subseteq \DCH(Y)$ and $\Gamma^*\DCH(Y) \subseteq \DCH(X)$\,;
	\item $\Gamma'\circ \Gamma \in \DCH(X\times Z)$.
\end{enumerate}
	\end{prop}
	
\begin{proof}
This is a direct consequence of Proposition \ref{prop:tensor} and Proposition \ref{prop:Disting}.
\end{proof}
	
	\begin{cor}[Top Chern class]\label{cor:TopChern}
	Let $X$ be an $n$-dimensional smooth projective variety equipped with a marking satisfying $(\starM)$. Then the top Chern class of $X$ is distinguished, \emph{i.e.}, $c_{n}(T_{X})\in \DCH_{0}(X)$.\\
	In particular, for a smooth projective curve, $(\starC)$ is implied by $(\starM)$.
	\end{cor}
	\begin{proof} Observe that the small diagonal $\delta_{X}$, viewed as a correspondence between $X$ and $X\times X$, is distinguished by hypothesis and it transforms $\Delta_X$ to $c_{n}(X)$\,:
	$$\delta_{X}^{*}\left(\Delta_{X}\right)=c_{n}(X).$$ Under the hypothesis $(\starM)$, we know that $\Delta_{X}\in \DCH(X\times X)$ by Lemma \ref{lemma:Diagonal}. Hence Proposition \ref{prop:distmorphism}$(i)$ yields that the top Chern class $c_{n}(X)$ is distinguished.\\
	As for the case of curves, it suffices to recall moreover that the fundamental class is automatically distinguished by Remark \ref{rmk:FundCl}.	
	\end{proof}

	\section{Operations preserving the condition $(\star)$}\label{sect:examples}
	In this section, we provide some standard operations on varieties that preserve
	$(\star)$. From now on, we systematically omit the functor $F:\M^{ab}_{sd}\to
	\M^{ab}$, which is an equivalence of categories (Lemma \ref{lemma:Equivalence}),
	in the notation of a marking.
	
	\subsection{Product varieties}
	
	Given two smooth projective varieties $X$ and $Y$ with markings $\phi:
	\h(X)\xrightarrow{\isom} M$ and $\psi: \h(Y)\xrightarrow{\isom} N$, their
	product will always be understood to be endowed with the marking 
	$$\phi \otimes \psi: \h(X\times_{k} Y)\cong \h(X)\otimes \h(Y)
	\xrightarrow{\isom} M\otimes N,$$
	which we will refer to as the \emph{product marking}. If $X$ and $Y$ are endowed
	with the action of a finite group $G$, then $X\times Y$ is endowed with the
	natural diagonal action of $G$.
	Our condition $(\star)$ (see Definition \ref{def:Star}) behaves well with
	respect to products\,:
	\begin{prop}[Products] \label{P:products}
		Assume $X$ and $Y$ are two smooth projective varieties satisfying the condition
		$(\star)$. Then  the natural marking on the product $X\times Y$ satisfies
		$(\star)$ and has the additional property that the graphs of the two natural
		projections are distinguished. 
		
		If in addition $X$ and $Y$ are equipped with the action of a finite group $G$
		and the respective markings satisfy $(\star_G)$,  then the product marking on
		$X\times Y$ satisfies $(\star_G)$.
	\end{prop}
	\begin{proof}
		By assumption, there are markings $\phi: \h(X)\xrightarrow{\isom} M$ and $\psi:
		\h(Y)\xrightarrow{\isom} N$ satisfying $(\star)$.
		The assertion $(\starM)$ (\emph{resp.} $(\star_G)$) follows from Proposition
		\ref{prop:tensor} applied to $X$ and $Y$ replaced by $X^{3}$ and $Y^{3}$
		(\emph{resp.} $X^{2}$ and $Y^{2}$). Indeed, $\delta_{X\times
			Y}=\delta_{X}\otimes \delta_{Y}$ (\emph{resp.} $g_{X\times Y} = g_X\otimes
		g_Y$).\\ 
		The assertion $(\starC)$ concerning the Chern classes follows directly from the
		formula $$c_{i}(X\times Y)=\sum_{j=0}^{i}c_{j}(X)\otimes c_{i-j}(Y)$$ and
		Proposition \ref{prop:tensor}.\\
		Finally, as the diagonal $\Delta_{X}\in \CH(X\times X)$ and fundamental class
		$\1_{Y}$ of $Y$ are distinguished (Lemma \ref{lemma:Diagonal}, Remark \ref{rmk:FundCl}),
		Proposition \ref{prop:tensor} tells us that the graph of the projection $X\times
		Y\to X$, which is equal to $\Delta_{X}\otimes \1_{Y}\in \CH(X\times X\times Y)$,
		is distinguished. The proof is similar for the other projection $X\times Y\to
		Y$.
	\end{proof}

	\begin{rmk}[Permutations]\label{R:perm}
		Suppose $X$ has a marking that satisfies $(\star)$. Then any permutation of the
		factors of $X^n$ defines a distinguished correspondence in $\DCH(X^{2n})$ for
		the product marking by Corollary \ref{C:diag}.
	\end{rmk}
	
	\begin{rmk}
		Assume $X$ and $Y$ are two smooth projective varieties endowed with the action
		of the finite groups $G$ and $H$, respectively. The product $G\times H$ acts
		naturally on the product $X\times Y$. Suppose  $X$ and $Y$  satisfy $(\star_G)$
		and $(\star_H)$, respectively.  Then the same arguments as above show that
		product marking on $X\times Y$ satisfies $(\star_{G\times H})$.
	\end{rmk}
	
	\subsection{Projective bundles}
	We show in this subsection that the condition $(\star)$ is stable by forming
	projective bundles as long as the Chern classes of the vector bundle are
	distinguished. 
	
	Let $X$ be a smooth projective variety of dimension $d$ and $E$ be a vector
	bundle over $X$ of rank $(r+1)$. Let $\pi: \PP(E)\to X$ be the associated
	projective bundle\footnote{The $\PP$ we are using here is the space of
		1-dimensional subspaces, thus different from Grothendieck's convention.}.  Let
	$\xi$ be the first Chern class of $\calO_{\pi}(1)$.

	Recall the \emph{projective bundle formula} (see \cite[\S 4.3.2]{MR2115000})\,:
	\begin{equation}\label{eqn:projbun}
	b: \bigoplus_{k=0}^{r}\h(X)(-k)\xrightarrow{\isom}\h(\PP E) ,
	\end{equation}
	which is given factor-wise by $\xi^{k}\cdot \pi^{*}:\h(X)(-k)\to \h(\PP E)$ for $0\leq
	k\leq r$.
	
	The following lemma\footnote{This should be known but the authors could not
		find a proper reference.} computes the small diagonal for $\PP E$. A piece of
	notation is convenient\,: for an element $\omega\in\CH^{k}(X)$, viewed as a
	morphism $\1\to \h(X)(k)$, we will talk about the morphism \emph{multiplication
		by $\omega$}, denoted by $\cdot\omega:\h(X)\to \h(X)(k)$, which is by definition
	the following composition\,:
	$$\h(X)\xrightarrow{\id\otimes \omega}
	\h(X)\otimes\h(X)(k)\xrightarrow{\delta_{X}(k)}\h(X)(k).$$ With a
	marking being fixed, if $\omega$ belongs to $\DCH(X)$ and $X$ satisfies
	$(\starM)$, then by Proposition~\ref{prop:distmorphism} multiplication by $\omega$ is a distinguished morphism.
	
	%\begin{lem}[Diagonal of projective bundles]\label{lemma:diagPB}
	%Notation is as above. Via the projective bundle formula (\ref{eqn:projbun}),
	%the Poincar\'e duality $PD_{\PP E}: \h(\PP E)(d+r)\xrightarrow{\isom} \h(\PP
	%E)^{\vee}$ induces an isomorphism 
	%$$b^{\vee}\circ PD_{\PP E}\circ b(d+r):
	%\bigoplus_{k=0}^{r}\h(X)(d+r-k)\xrightarrow{\isom}
	%\bigoplus_{l=0}^{r}\h(X)^{\vee}(l)$$ where for any $0\leq k,l\leq r$, the
	%morphism $\h(X)(d+r-k)\to \h(X)^{\vee}(l)$ is described as\,:
	%\begin{itemize}
	%\item if  $l+k\geq r$, it is the composition\,: 
	%$$\h(X)(d+r-k)\xrightarrow{\cdot
	%s_{k+l-r}(E)}\h(X)(d+l)\xrightarrow{PD_{X}(l)}\h(X)^{\vee}(l),$$ where the
	%first morphism is multiplication by the Segre class\footnote{The total Segre
	%class is by definition the inverse of the total Chern class, \emph{cf.}
	%\cite[Chapter 3]{MR1644323}.}.
	%\item  If $k+l<r$, it is the zero map.
	%\end{itemize}
	%\end{lem}
	%\begin{proof}
	%By Manin's identity principle (\cite[\S 4.3.1]{MR2115000}), it suffices to
	%prove the lemma for Chow groups. Note that the Poincar\'e dualities induce
	%simply the identity morphisms on Chow groups. Thus we have to compute the
	%composition  ${}^{t}b\circ b$ whose $(k,l)$-component for any $0\leq k,l\leq r$
	%is 
	%$$\CH^{*+d+r-k}(X)\xrightarrow{\xi^{k}\cdot\pi^{*}(-)}\CH^{*+d+r}(\PP
	%E)\xrightarrow{\pi_{*}(\xi^{l}\cdot -)}\CH^{*+d+l}(X).$$
	%Now for any $z\in\CH(X)$, $\pi_{*}(\xi^{l}\cdot\xi^{k}\cdot\pi^{*}(z))=z\cdot
	%\pi_{*}(\xi^{k+l})=z\cdot s_{k+l-r}(E)$ by the definition of Segre class. We
	%conclude by remarking that all negative Segre classes vanish.
	%\end{proof}
	
	\begin{lem}[Small diagonal of projective bundles]\label{lemma:sdiagPB}
		Notation is as above. The intersection product $$\delta_{\PP E}: \h(\PP
		E)\otimes \h(\PP E)\to\h(\PP E)$$ induces, \emph{via} (\ref{eqn:projbun}), a
		morphism $\left(\bigoplus_{k=0}^{r}\h(X)(-k)\right)^{\otimes 2}\to
		\bigoplus_{m=0}^{r}\h(X)(-m)$, such that for any $0\leq k, l, m\leq r$, the
		morphism $$\h(X)(-k)\otimes \h(X)(-l)\to \h(X)(-m)$$ is described as\,:
		\begin{itemize}
			\item If $m>k+l$ or $m>r$, it is the zero map.
			\item If $m=k+l\leq r$, it is induced by the intersection product of
			$X$, namely,~$\delta_{X}$.
			\item If $k+l\leq r$ and $m\neq k+l$, it is the zero map.
			\item If $m\leq r< k+l$, then it is the composition $$\h(X)(-k)\otimes
			\h(X)(-l)\xrightarrow{\delta_{X}(-k-l)}\h(X)(-k-l)\xrightarrow{\cdot \omega}
			\h(X)(-m),$$ where the second morphism is the multiplication by the following
			characteristic class  (with $s$ being the Segre class\footnote{The total Segre class is by definition the inverse of the total Chern class, \emph{cf.} \cite[Chapter 3]{MR1644323}.}) 
			$$\omega:=\sum_{t=0}^{r-m}c_{t}(E)s_{k+l-m-t}(E) \in
			\CH^{k+l-m}(X).$$
		\end{itemize}
	\end{lem}
	\begin{proof}
		By Manin's identity principle (\cite[\S 4.3.1]{MR2115000}), we only have to
		prove the lemma for Chow groups. Let us first compute the inverse $b^{-1}$ of
		the isomorphism in the projective bundle formula $$b:
		\bigoplus_{k=0}^{r}\CH^{*-k}(X)\xrightarrow{\isom} \CH^{*}(\PP E).$$ Assume
		$\gamma\in \CH^{*}(\PP E)$ is the image of $(z_{0}, z_{1}, \cdots, z_{r}) \in
		\oplus_{k=0}^{r}\CH^{*-k}(X)$, \emph{i.e.},
		$$\gamma=\sum_{k=0}^{r}\pi^{*}(z_{k})\cdot \xi^{k}.$$
		For any $t\geq 0$, $\pi_{*}(\gamma\cdot
		\xi^{t})=\sum_{k=0}^{r}\pi_{*}(\pi^{*}(z_{k})\cdot
		\xi^{k+t})=\sum_{k=0}^{r}z_{k}\cdot s_{k+t-r}(E)$. Since the total Segre class
		is the inverse of the total Chern class, we have for any $0\leq k \leq r$,
		$$z_{k}=\sum_{t=0}^{r-k}c_{t}(E)\cdot \pi_{*}(\gamma\cdot \xi^{r-k-t}).$$ This
		gives $b^{-1}$. Now let us go back to the product formula. We have to compute
		the composition $b^{-1}\circ (b\otimes b)$, whose $(k,l,m)$-th component for any
		$0\leq k, l, m\leq r$ is the composition\,:
		$$\CH(X)\otimes \CH(X)\xrightarrow{(\xi^{k}\cdot\pi^{*}, \xi^{l}\cdot\pi^{*})}
		\CH(\PP E)\otimes \CH(\PP E)\xrightarrow{\cdot} \CH(\PP
		E)\xrightarrow{b^{-1}_{m}} \CH(X),$$ where the last morphism is
		$\sum_{t=0}^{r-m}c_{t}(E)\cdot \pi_{*}(\bullet\cdot \xi^{r-m-t})$ by the formula
		for $b^{-1}$.
		Now, for any $z, z'\in \CH(X)$, the $m$-th component of 
		$\pi^{*}(z)\cdot \xi^{k}\cdot \pi^{*}(z')\cdot \xi^{l}=\pi^{*}(z\cdot z')\cdot
		\xi^{k+l}$ is 
		$\sum_{t=0}^{r-m}c_{t}(E)\cdot \pi_{*}(\pi^{*}(z\cdot z')\cdot \xi^{k+l}\cdot
		\xi^{r-m-t})=z\cdot z'\cdot(\sum_{t=0}^{r-m}c_{t}(E)s_{k+l-m-t}(E))$. We can
		conclude in all cases easily.
	\end{proof}
	
	\begin{prop}[$(\star)$ and projective bundles]\label{prop:PB}
		
		Let $X$ be a smooth projective variety and let $E$ be a vector bundle over $X$
		of rank $(r+1)$. Let $\pi: \PP(E)\to X$ be the associated projective bundle. If we
		have a marking for $X$ satisfying $(\star)$ such that all Chern classes of $E$
		are distinguished, then $\PP E$ has a natural marking such that $\PP E$
		satisfies $(\star)$ and such that the projection $\pi: \PP E \to X$
		is distinguished.
		
		If in addition  $X$ is equipped with the action of a finite group $G$ such that
		$E$ is $G$-equivariant and such that the marking of $X$ satisfies $(\star_G)$,
		then the natural marking of $\PP E$ satisfies $(\star_G)$.
	\end{prop}
	\begin{proof}
		Let $\phi: \h(X)\xrightarrow{\isom} M$ be a marking that satisfies $(\star)$ and
		is such that $c_{k}(E)\in \DCH(X)$. Using the projective bundle formula
		(\ref{eqn:projbun}), we obtain a marking for $\PP E$\,:
		$$\lambda: \h(\PP E) \xrightarrow{\isom} \bigoplus_{k=0}^{r}M(-k).$$
		Let us show that $\lambda$ satisfies $(\star)$.\\
		%For $(\star1)$, it is equivalent, by Lemma \ref{lemma:InterpretationStar}, to
		%say that the bottom isomorphism in the following commutative diagram is
		%symmetrically distinguished\,:
		%\begin{equation*}
		%\xymatrix{
		%\h(\PP E)(d+r)\ar[r]^{PD_{\PP E}} \ar@/_6pc/[dd]_{\lambda(d+r)} & \h(\PP
		%E)^{\vee} \ar[d]^{b^{\vee}}_{\isom}\\
		%\bigoplus_{k=0}^{r}\h(X)(d+r-k) \ar[u]^{b(d+r)}_{\isom} \ar[r]
		%\ar[d]^{\isom}_{\oplus\phi(\cdot)}&  \bigoplus_{l=0}^{r}\h(X)^{\vee}(l)\\
		%\bigoplus_{k=0}^{r}M(d+r-k)  \ar[r] &
		%\bigoplus_{l=0}^{r}M^{\vee}(l)\ar[u]^{\isom}_{\oplus\phi^{\vee}(\cdot)}
		%\ar@/_6pc/[uu]_{\lambda^{\vee}}
		%}
		%\end{equation*}
		%However, this is a consequence of the description of the middle morphism in
		%Lemma \ref{lemma:diagPB}, since all the Chern classes (and hence Segre classes)
		%of $E$ as well as $PD_{X}$ is distinguished by assumption.\\
		For $(\starM)$, one uses the interpretation of $(\starM)$ given in Lemma
		\ref{lemma:InterpretationStar}. Since $\delta_{X}$ as well as the
		Chern classes and Segre classes of $E$ are distinguished, the condition $(\starM)$ follows from Lemma~\ref{lemma:sdiagPB}.\\
		For $(\starC)$, we first claim that for any $k$, the cycle $\pi^{*}(\alpha)\cdot
		\xi^{k}$ is distinguished if $\alpha\in \CH(X)$ is so. For $k\leq r$, this
		is by definition, while for $k>r$, we use the equality
		$\xi^{r+1}+\pi^{*}(c_{1}(E))\xi^{r}+\cdots+\pi^{*}(c_{r+1}(E))=0$ and the
		distinguishedness of the Chern classes of $E$ to reduce to the treated cases. Now
		from the short exact sequences
		$$0\to \calO_{\PP E}\to \pi^{*}(E)\otimes \calO_{\pi}(1)\to T_{\PP E/X}\to
		0\,;$$
		$$0\to T_{\PP E/X}\to T_{\PP E}\to \pi^{*}T_{X}\to 0,$$
		we see that all the Chern characters of $\PP E$ are linear combinations of terms of
		the form $\pi^{*}(\alpha)\cdot \xi^{k}$, where $\alpha$ is a polynomial of
		Chern classes of $X$ and of $E$. By assumption $\alpha$ is distinguished hence
		so are the Chern characters of $\PP E$. With $(\starM)$ being proven for $\PP E$,
		we know that $\DCH(\PP E)$ is a subalgebra by Proposition~\ref{prop:Disting}.
		We are then done because Chern classes are polynomials of Chern characters.\\
		%\footnote{In Proposition \ref{prop:Disting}, the assumption on Chern classes is
		%not used for the proof that $\DCH$ forms a subalgebra.} 
		The distinguishedness of (the graph of) the projection $\pi: \PP(E)\to X$ is
		obvious\,: \emph{via} the markings $\phi$ and $\lambda$, it is equivalent to saying
		that the inclusion of the first summand $$M\inj M\oplus M(-1)\oplus\cdots\oplus
		M(-r)$$ is a symmetrically distinguished morphism.\\
		Finally, assume that $X$ is equipped  with the action of a finite group $G$ such
		that $E$ is $G$-equivariant. Note that with the induced action of $G$ on $\PP
		E$, we have that $\pi$ is $G$-equivariant and we have that $(g_{\PP E})_*\xi =
		\xi$ (since $G$ preserves $\mathcal{O}_\pi(1)$). Thus the action of $G$ commutes
		with $b$ and $b^\vee$. Since we are assuming that the marking $\phi$ of $X$
		satisfies $(\star_G)$, we find that the marking $\lambda$ satisfies $(\star_G)$.
	\end{proof}
	
	\begin{ex} \label{ex:Schur} 
%	If $X$ is a smooth projective variety with a marking that satisfies
%		$(\star)$, then natural
%		examples of vector bundles with distinguished Chern classes are given by the
%		tangent bundle $T_{X}$ as well as other vector bundles obtained from it by
%		performing duals, tensor products, direct sums and direct summands \emph{etc.}
%		Concretely, these are direct sums of vector bundles of the form $\mathbb
%		S_{\lambda}T_{X}$, where $\lambda$ is a non-increasing sequence of integers and
%		$\mathbb S_{\lambda}$ is the associated Schur functor. By Proposition
%		\ref{prop:PB}, the projective bundle associated to any such vector bundle has a marking that
%		satisfies~$(\star)$.
	If $X$ is a smooth projective variety with a marking that satisfies
$(\star)$, then natural
examples of vector bundles with distinguished Chern classes are given by the
tangent bundle $T_{X}$ as well as other vector bundles obtained from it by
performing duals, tensor products, and direct sums. More generally, one may consider direct sums of vector bundles of the form $\mathbb
	S_{\lambda}T_{X}$, where $\lambda$ is a non-increasing sequence of integers and
		$\mathbb S_{\lambda}$ is the associated Schur functor. By Proposition
\ref{prop:PB}, the projective bundle associated to any such vector bundle has a marking that
satisfies~$(\star)$.
	\end{ex}

	\subsection{Blow-ups}
	We will show in this subsection that the condition $(\star)$ in Definition~\ref{def:Star} passes to a blow-up in the expected way.
	
	We fix the following notation throughout this subsection. Let $X$ be a smooth
	projective variety of dimension $d$, $i:Y\inj X$ be a closed immersion of a smooth subvariety
	of codimension $c$ and $\calN:=\calN_{Y/X}$ be the normal bundle. Let $\tilde X$
	be the blow-up of $X$ along $Y$ and $E$ the exceptional divisor, which is
	identified with $\PP(\calN)$. Denote by $\xi$ the first Chern class of
	$\calO_{p}(1)=\calN^{\vee}_{E/{\tilde X}}$. The names of some relevant morphisms
	are in the following cartesian diagram\,:
	\begin{equation}\label{E:cartblowup}
	\xymatrix{
		E\ar@{^{(}->}[r]^{j} \ar[d]_{p} & \tilde X\ar[d]^{\tau}\\
		Y \ar@{^{(}->}[r]^{i} & X
	}
	\end{equation}
	Recall the blow-up formula (see \cite[\S 4.3.2]{MR2115000})
	\begin{equation}\label{eqn:blowup}
	b: \h(X)\oplus \bigoplus_{k=1}^{c-1}\h(Y)(-k)\xrightarrow{\isom}\h(\tilde X) ,
	\end{equation}
	which is given by\,:
	\begin{itemize}
		\item $\tau^{*}: \h(X)\to \h(\tilde X)$\,;
		\item for any $1\leq k\leq c-1$, $  j_{*}(\xi^{k-1}\cdot p^{*}(-)):\h(Y)(-k)\to
		\h(\tilde X)$.
	\end{itemize}
	The following lemma\footnote{This should be known but the authors could not
		find a proper reference.} computes the small diagonal of  $\tilde X^{3}$.  

	\begin{lem}[Small diagonal of blow-ups]\label{lemma:sdiagBl}
		The intersection product $$\delta_{\tilde X}: \h(\tilde X)\otimes \h(\tilde
		X)\to\h(\tilde X)$$ is described \emph{via} the isomorphism (\ref{eqn:blowup})
		as follows\,:
		\begin{itemize}
			\item $\h(X)\otimes \h(X)\to \h(X)$ is the intersection product (induced by
			$\delta_{X}$)\,;
			\item For any $1\leq k\leq c-1$, $\h(X)\otimes \h(Y)(-k)\to \h(Y)(-k)$ is the
			composition $$\h(X)\otimes \h(Y)(-k)\xrightarrow{i^{*}\otimes \id} \h(Y)\otimes
			\h(Y)(-k) \xrightarrow{\delta_{Y}(-k)} \h(Y)(-k)\,;$$
			\item For any $1\leq k, l\leq c-1$, $$\h(Y)(-k)\otimes \h(Y)(-l)\to \h(X)$$ is
			the composition\,:
			$$\h(Y)(-k)\otimes
			\h(Y)(-l)\xrightarrow{\delta_{Y}(-k-l)}\h(Y)(-k-l)\xrightarrow{-\cdot
				s_{k+l-c}(\calN)}\h(Y)(-c)\xrightarrow{i_{*}}\h(X)$$
			where in second morphism, $s$ stands for the Segre class.
			\item For any $1\leq k, l, m\leq c-1$,  $$\h(Y)(-k)\otimes \h(Y)(-l)\to
			\h(Y)(-m)$$ is as follows\,:
			\begin{itemize}
				\item if $m\geq c$ or $m>k+l$, it is the zero map.
				\item if $m=k+l\leq c-1$, then it is induced by $-\delta_{Y}$.
				\item if $m\neq k+l \leq c-1$, then it is the zero map.
				\item if $m\leq c-1<k+l$, it is the composition $$\h(Y)(-k)\otimes
				\h(Y)(-l)\xrightarrow{\delta_{Y}(-k-l)} \h(Y)(-k-l)\xrightarrow{\cdot \omega}
				\h(Y)(-m),$$
				where the second morphism is the multiplication by the following characteristic
				class with $s$ standing for the Segre class.
				$$\omega:=-\sum_{t=1}^{c-m}s_{k+l-m-t+1}(\calN)\cdot c_{t-1}(\calN)\in
				\CH^{k+l-m}(Y).$$
			\end{itemize}
		\end{itemize}
	\end{lem}
	\begin{proof}
		We only have to prove the lemma for Chow groups thanks to Manin's identity principle
		(\cite[\S 4.3.1]{MR2115000}). As in Lemma \ref{lemma:sdiagPB}, we compute the
		inverse of 
		$$b: \CH^{*}(X)\oplus \bigoplus_{k=1}^{c-1}\CH^{*-k}(Y) \to \CH^{*}(\tilde X).$$
		Assume $\gamma=\tau^{*}(z_{0})+\sum_{k=1}^{c-1}j_{*}(p^{*}(z_{k})\cdot
		\xi^{k-1})$ where $z_{0}\in \CH(X)$ and $z_{k}\in \CH(Y)$ for all $1\leq k\leq
		c-1$. Then $b^{-1}$ is given by
		$z_{0}=\tau_{*}(\gamma)$\,; and for all $1\leq k\leq c-1$,
		$$z_{k}=-\sum_{t=1}^{c-k}p_{*}(j^{*}(\gamma)\cdot \xi^{c-k-t})\cdot
		c_{t-1}(\calN).$$ Now concerning intersection products, we have to compute $b^{-1}\circ
		(b\otimes b)$. We only give the computation of the $(k,l,m)$-th component
		when $1\leq k, l, m\leq c-1$ and leave the other cases to the reader. Let $z,
		z'\in \CH(Y)$, then the $m$-th component of the product $j_{*}(p^{*}(z)\cdot
		\xi^{k-1})\cdot j_{*}(p^{*}(z')\cdot \xi^{l-1})=j_{*}\left(p^{*}(z)\cdot
		\xi^{k-1}\cdot j^{*}(j_{*}(p^{*}(z')\cdot
		\xi^{l-1}))\right)=-j_{*}\left(p^{*}(z\cdot z')\cdot \xi^{k+l-1}\right)$ is 
		\begin{eqnarray*}
			&&\sum_{t=1}^{c-m}p_{*}\left(j^{*}j_{*}\left(p^{*}(z\cdot z')\cdot
			\xi^{k+l-1}\right)\cdot \xi^{c-m-t}\right)\cdot c_{t-1}(\calN)\\
			&=&-\sum_{t=1}^{c-m}p_{*}\left(p^{*}(z\cdot z')\cdot \xi^{k+l+c-m-t}\right)\cdot
			c_{t-1}(\calN)\\
			&=& -\sum_{t=1}^{c-m} z\cdot z'\cdot s_{k+l-m-t+1}(\calN)\cdot c_{t-1}(\calN).
		\end{eqnarray*}
		Then all cases follow easily.
	\end{proof}
	
	\begin{prop}[$(\star)$ and blow-ups]\label{prop:Blup}
		Let $X$ be a smooth projective variety and let $i:Y\inj X$ be a smooth closed 
		subvariety. If we have markings satisfying the condition $(\star)$ for $X$ and
		$Y$ such that the inclusion morphism $i:Y\inj X$ is distinguished (Definition \ref{def:DisCor}), then $\tilde X$, the
		blow-up of $X$ along $Y$, has a natural marking that satisfies $(\star)$ and is
		such that the morphisms in Diagram \eqref{E:cartblowup} are all
		distinguished \footnote{The exceptional divisor $E$ is endowed with the natural
			marking of Proposition \ref{prop:PB} by its projective bundle structure over
			$Y$.}.
		
		If in addition $X$ is equipped with the action of a finite group $G$ such that
		$G\cdot Y = Y$ and such that the markings of $X$ and $Y$ satisfy $(\star_G)$,
		then the natural marking of $\tilde X$ also satisfies $(\star_G)$.
	\end{prop}
	\begin{proof}
		Let $\phi: \h(X)\xrightarrow{\isom} M$ and $\psi: \h(Y)\xrightarrow{\isom} N$ be
		markings satisfying $(\star)$. 
%		The distinguishedness of the inclusion morphism $i$ means that the
%		induced morphism $ \psi\circ i^{*}\circ\phi^{-1}:M\to N$ is symmetrically
%		distinguished, or equivalently, the graph of the inclusion $\Gamma_{i}\in
%		\DCH(X\times Y)$ is distinguished. 
		Using the blow-up formula (\ref{eqn:blowup}),
		$\phi$ and $\psi$ induce a marking for $\tilde X$\,:
		\begin{equation}\label{eqn:MarkingBl}
		\lambda: \h(\tilde X)\xrightarrow{\isom} M\oplus \bigoplus_{k=1}^{c-1}N(-k),
		\end{equation}
		which we will show to satisfy $(\star)$.
		
%		We first point out that the distinguishedness of $i^{*}$ implies
%		$i^{*}(\DCH(X))\subset \DCH(Y)$. By Lemma \ref{lemma:Diagonal}$(ii)$, applied to the Poincar\'e dualities for $X$ and $Y$,
%		$i_{*}:\h(Y)\to \h(X)(-c)$ is also distinguished. In particular,
%		$i_{*}(\DCH(Y))\subset \DCH(X)$.
		
		Using the short exact sequence $0\to T_{Y} \to T_{X}\vert_{Y}\to \calN\to 0$, we
		see that the Chern classes of $\calN$ can be expressed as polynomials of Chern classes of $Y$
		and Chern classes of $X$ restricted to $Y$, which are all in $\DCH(Y)$ by
		hypothesis $(\starC)$ for $X$ and $Y$. Since $\DCH(Y)$ is a subalgebra
		(Proposition \ref{prop:Disting}), all Chern classes of $\calN$ are distinguished
		on $Y$. 
%		
		%For $(\star1)$, it amounts to show that the bottom isomorphism in the following
		%diagram determined by the commutativity is symmetrically distinguished\,:
		%\begin{equation*}
		%\xymatrix{
		%\h(\tilde X)(d)\ar[r]^{PD_{\tilde X}} \ar@/_6pc/[dd]_{\lambda(d)} & \h(\tilde
		%X)^{\vee} \ar[d]^{b^{\vee}}_{\isom}\\
		%\h(X)(d)\oplus \bigoplus_{k=1}^{c-1}\h(Y)(d-k) \ar[u]^{b(d)}_{\isom} \ar[r]
		%\ar[d]^{\isom}_{\phi, \psi}& \h(X)^{\vee}\oplus
		%\bigoplus_{l=1}^{c-1}\h(Y)^{\vee}(l)\\
		%M(d)\oplus \bigoplus_{k=1}^{c-1}N(d-k)  \ar[r] & M^{\vee}\oplus
		%\bigoplus_{l=1}^{c-1}N^{\vee}(l)\ar[u]^{\isom}_{\phi^{\vee}, \psi^{\vee}}
		%\ar@/_6pc/[uu]_{\lambda^{\vee}}
		%}
		%\end{equation*}
		%However, this is a consequence of the description of the middle morphism given
		%in Lemma \ref{lemma:diagBl}, since $PD_{X}$, $PD_{Y}$ and the Segre classes of
		%$\calN$ are all distinguished by assumption.\\
		%With $(\star1)$ being proven, 
		The condition $(\starM)$ then follows from Lemma~\ref{lemma:sdiagBl} (together with Proposition~\ref{prop:distmorphism}), since all Segre
		and Chern classes as well as the morphisms $i^{*}:\h(X)\to \h(Y)$,
		$i_{*}:\h(Y)\to \h(X)(c)$, the intersection products $\delta_{X}: \h(X)^{\otimes
			2}\to \h(X)$ and $\delta_{Y}:\h(Y)^{\otimes 2}\to \h(Y)$ are all distinguished
		by assumption.
		%Note that the `quick' proof here does not mean $(\starM)$ is tautological\,:
		%all the difficulties are put into the proof of $(\star1)$ and two preparation
		%Lemmas \ref{lemma:diagBl} and \ref{lemma:sdiagBl}, without which the equivalent
		%interpretation of $(\starM)$ by the intersection product in Lemma
		%\ref{lemma:InterpretationStar} is not available.
		
	That the morphisms in Diagram \eqref{E:cartblowup} are all
	distinguished in the sense of Definition \ref{def:DisCor} is straightforward\,:	the inclusion morphism $i:Y\inj X$ is distinguished by assumption\,; the projective bundle $p: E\to
		Y$ is distinguished thanks to Proposition \ref{prop:PB}\,; the distinguishedness
		of of $\tau$ is equivalent to say that (\emph{via} the markings $\phi$
		and $\lambda$) the inclusion of the first summand $M\inj M\oplus
		\bigoplus_{k=1}^{c-1}N(-k)$ is symmetrically distinguished, which is obvious\,;
		finally, one checks easily that \emph{via} the natural markings, the morphism
		$j^{*}:\h(\tilde X)\to\h(E)$ corresponds to the morphism $$(i^{*}, -\id, \cdots,
		-\id): M\oplus N(-1)\oplus \cdots \oplus N(-c+1)\to N\oplus N(-1)\oplus\cdots
		\oplus N(-c+1),$$ which is obviously symmetrically distinguished. 
		
		Now for $(\starC)$, we use the formula for Chern classes of a blow-up given in
		\cite[Theorem 15.4]{MR1644323}. Given the distinguishedness of the Chern classes of
		$T_{X}$, $T_{Y}$ and $\calN$, we only have to show that for any $\alpha\in
		\DCH(Y)$ and $k\in \N$, the class $j_{*}(p^{*}(\alpha)\cdot \xi^{k})\in
		\CH(\tilde X)$ is distinguished. But that is immediate, because each of $j, p, \alpha$, and $\xi = -j^*j_*(1)$ is distinguished by the above.
		
%		For $k<c-1$, this is by definition. For $k=c-1$,
%		by the excess intersection formula\footnote{In this case, it is called the `key
%			formula' in \cite[Proposition 6.7(a)]{MR1644323}.} (\cite[\S6.3]{MR1644323}),
%		$j_{*}(p^{*}(\alpha)\cdot \xi^{k})$ is a linear combination of
%		$\tau^{*}(i_{*}(\alpha))$ with some $j_{*}(p^{*}(\alpha)\cdot \xi^{l})$ for
%		$l<k$. Since $i$ is distinguished by assumption and since $\tau$ is distinguished by definition of the marking $\lambda$, we get that
%	    $\tau^{*}(i_{*}(\alpha))$ is distinguished by Proposition~\ref{prop:distmorphism} . Finally, for
%		$k>c-1$, one uses the equality
%		$\xi^{c}+p^{*}(c_{1}(\calN))\xi^{c-1}+\cdots+p^{*}(c_{c}(\calN))=0$ to reduce to
%		the treated cases.
		
		Finally, assume that  $X$ is equipped with the action of a finite group $G$ such
		that $G\cdot Y = Y$. Note that with the induced action of $G$ on $E$ and $\tilde
		X$, we have that the morphisms in diagram \eqref{E:cartblowup} are
		$G$-equivariant. Thus the action of $G$ commutes with $b$ and $b^\vee$. Since we
		are assuming that the markings  of $X$ and $Y$ satisfy $(\star_G)$, we find that
		the marking $\lambda$ satisfies $(\star_G)$.
	\end{proof}

	\subsection{Generically finite morphism}\label{subsect:quotient}
	In this subsection, we show that the condition $(\star)$ passes from the source
	variety of a surjective and generically finite morphism to the target variety
	under natural assumptions.
	
	\begin{prop}\label{prop:cover}
		Let $\pi: X\to Y$ be a generically finite and surjective morphism between smooth
		projective varieties. If $X$ has a marking satisfying $(\starM)$ and such that
		the cycle ${}^{t}\Gamma_{\pi}\circ\Gamma_{\pi}$ is distinguished in $\CH(X\times
		X)$, then $Y$ has a natural marking that satisfies $(\starM)$ and is such that
		the graph of $\pi$ is distinguished.
	\end{prop}
	\begin{proof}
		Let $d$ be the degree of $\pi$ and $n$ be the dimension of $X$ and $Y$.
		Let $M\in \M^{ab}_{sd}$ and $\phi: \h(X)\xrightarrow{\isom} M$ be a marking
		satisfying $(\starM)$. The graph of $\pi$ and its transpose induce respectively 
		$$\pi_{*}:=\Gamma_{\pi}: \h(X)\to \h(Y),$$
		$$\pi^{*}:={}^{t}\Gamma_{\pi}: \h(Y)\to \h(X)$$
		such that $\pi_{*}\circ\pi^{*}=d\cdot \Delta_{Y}\in \End(\h(Y))$.
		Therefore
		$\frac{1}{d}{}^{t}\Gamma_{\pi}\circ\Gamma_{\pi}=\frac{1}{d}\pi^{*}\circ
		\pi_{*}\in \End(\h(X))$ is a projector and $\pi^{*}$ induces an isomorphism of
		Chow motives $$\h(Y)\lra{\isom} \left(X,
		\frac{1}{d}{}^{t}\Gamma_{\pi}\circ\Gamma_{\pi}, 0\right).$$
		Consider the projector $$q:=\phi\circ\frac{1}{d}{}^{t}\Gamma_{\pi}\circ\Gamma_{\pi}\circ
		\phi^{-1}$$ in $\End(M)$.
		Since ${}^{t}\Gamma_{\pi}\circ\Gamma_{\pi}$ is distinguished by assumption, $q$
		is a symmetrically distinguished idempotent endomorphism of $M$. By Lemma
		\ref{lemma:PsAb}, we have a canonical image $$N:=\im\left(q: M\to M\right),$$
		with $N\in \M^{ab}_{sd}$ and such that the projection $p:M\surj N$ and the inclusion
		$i: N\inj M$ are symmetrically distinguished morphisms in $\M^{ab}_{sd}$. 
		
		By definition, we have $p\circ i=\id$ and $i\circ p=q=\phi\circ\frac{1}{d}\pi^{*}\circ \pi_{*}\circ\phi^{-1}$. Therefore the composition $$\lambda:=p\circ \phi\circ \pi^{*}: \h(Y)\to N$$ is an isomorphism with inverse $\lambda^{-1}=\frac{1}{d}\pi_{*}\circ \phi^{-1}\circ i$.
		%The projection $M\surj N$ is also denoted by $p$. Since
		%${}^{t}\Gamma_{\pi}\circ\Gamma_{\pi}$ is distinguished by assumption, $p$ is a
		%symmetrically distinguished idempotent endomorphism of $M$ and therefore $N\in
		%\M^{ab}_{sd}$ and $p:M\surj N$, as well as the inclusion $N\inj M$, is a
		%symmetrically distinguished morphism.
		Note that $\lambda$ is nothing else but the following composition of isomorphisms\,:
		$$\h(Y)\xrightarrow{\isom} \left(X,
		\frac{1}{d}{}^{t}\Gamma_{\pi}\circ\Gamma_{\pi}, 0\right)\xrightarrow{p\circ\phi}
		N.$$
		We now show that the marking for $Y$ provided by the isomorphism $\lambda$ satisfies $(\starM)$.
		%For $(\star1)$, we consider the following commutative diagram
		%\begin{equation*}
		%\xymatrix{
		%\h(Y)(n)\ar[r]^{PD_{Y}}_{\isom} \ar@{^{(}->}[d]^{\pi^{*}}
		%\ar@/_4pc/[ddd]_{\lambda(n)} & \h(Y)^{\vee} \\
		%\h(X)(n)  \ar[r]^{PD_{X}}_{\isom} \ar[d]_{\isom}^{\phi}&
		%\h(X)^{\vee}\ar@{->>}[u]_{(\pi^{*})^{\vee}}\\
		%M(n)\ar[r]^{\sigma_{X}}_{\isom} \ar@{->>}[d]^{p} &
		%M^{\vee}\ar[u]^{\isom}_{\phi^{\vee}} \\
		%N(n) \ar[r]^{\sigma_{Y}}_{\isom}& N^{\vee}\ar@/_4pc/[uuu]_{\lambda^{\vee}}
		%\ar@{^{(}->}[u]_{p^{\vee}}
		%}
		%\end{equation*}
		%By assumption (and Lemma \ref{lemma:InterpretationStar} $(a)$), $\sigma_{X}$,
		%$p$ and $p^{\vee}$ are all symmetrically distinguished, hence so is
		%$\sigma_{Y}$. Again by Lemma \ref{lemma:InterpretationStar} $(a)$, $\lambda$
		%satisfies $(\star1)$.\\
		We consider the following commutative diagram,
		\begin{equation*}
		\xymatrix{
			\h(Y)\otimes\h(Y) \ar[r]^-{\delta_{Y}} \ar@{^{(}->}[d]^{\pi^{*}\otimes \pi^{*}}
			& \h(Y) \ar@{^{(}->}[d]^{\pi^{*}}
			\ar@/^4pc/[ddd]_{\lambda}\\
			\h(X)\otimes\h(X)  \ar[r]^-{\delta_{X}} \ar[d]_{\isom}^{\phi\otimes \phi}&
			\h(X)\ar[d]_{\isom}^{\phi}\\
			M\otimes M\ar[r]^-{\mu_{X}}  & M
			\ar@{->>}[d]^{p} \\
			N\otimes N \ar@{^(->}[u]^{i\otimes i} \ar[r]^-{\mu_{Y}}& N
		}
		\end{equation*}
		where $\mu_{Y}:=p\circ \mu_{X}\circ (i\otimes i)$ is clearly symmetrically distinguished as $\mu_{X}$, $i\otimes i$ and $p$ are so. 	
		By Lemma \ref{lemma:InterpretationStar}, it suffices to check that $\mu_{Y}\circ (\lambda\otimes \lambda)=\lambda\circ \delta_{Y}$. This is straightforward\,:
		\begin{eqnarray*}
		 \mu_{Y}\circ (\lambda\otimes \lambda)
		 &=& p\circ \mu_{X}\circ (i\otimes i)\circ (\lambda\otimes\lambda)\\
		 &=& p\circ \mu_{X}\circ (\phi\otimes \phi)\circ (\pi^{*}\otimes \pi^{*})\\
		 &=& p\circ \phi\circ \pi^{*}\circ \delta_{Y}\\
		 &=& \lambda\circ \delta_{Y},
		\end{eqnarray*}
		where the second equality uses $i\circ \lambda=i\circ p\circ \phi\circ \pi^{*}=q\circ \phi\circ \pi^{*}=\phi\circ \pi^{*}$ and the third equality uses the commutativity of the previous diagram.\\
		That the graph of $\pi: X\to Y$ is distinguished is equivalent to the condition
		that the natural inclusion $N\inj M$,  or equivalently $p:M\surj N$, is a
		symmetrically distinguished morphism.
		%, which follows directly from the assumption that $p=\phi\circ\frac{1}{d}\pi^{*}\circ \pi_{*}\circ \phi^{-1}$ is	symmetrically distinguished.
	\end{proof}
	
	\begin{rmk}[$(\star)$ and semi-small morphisms]\label{rmk:semismall}
		When $\pi: X\to Y$ is semi-small (\emph{cf.} \S\ref{subsubsect:semismall}), then
		the condition on the cycle ${}^{t}\Gamma_{\pi}\circ\Gamma_{\pi}$ in Proposition
		\ref{prop:cover} is equivalent to the more explicit condition that the class of
		$X\times_{Y}X$ in $\CH_{n}(X\times X)$ is distinguished. 
	\end{rmk}

	\begin{prop}[$(\star)$ and \'etale covers]\label{prop:etale}
		Notation and assumptions are as in Proposition \ref{prop:cover}.
		If moreover,  $\pi$ is \'etale and the marking for $X$ satisfies $(\starC)$,
		then the natural marking for $Y$ also satisfies $(\starC)$.
	\end{prop}
	\begin{proof}
		Let $d$ be the degree of $\pi$. For any $i\in \N$, 
		$c_{i}(Y)=\frac{1}{d}\pi_{*}\pi^{*}(c_{i}(Y))=\frac{1}{d}\pi_{*}c_{i}(X)$ is
		distinguished since $c_{i}(X)$ is distinguished and $\pi$ is a distinguished
		morphism. 
	\end{proof}
	
	\begin{prop}[$(\star)$ and finite group quotients]\label{P:quotient}
		Let $X$ be a smooth projective variety endowed with an action of a finite group
		$G$, such that the quotient $Y:=X/G$ is smooth. If there is a marking for
		$(X,G)$ satisfying $(\starM)$ and $(\star_G)$, then $Y$ has a natural marking
		that satisfies $(\starM)$ and is such that the quotient morphism
		$\pi : X \to Y$ is distinguished.
		
		Moreover, if $\pi: X\to Y$ is \'etale or a cyclic covering along a divisor $D$
		such that $D\in \DCH(X)$ and if the marking for $X$ satisfies $(\starC)$, then
		the natural marking for $Y$ also satisfies $(\starC)$.
	\end{prop}
	\begin{proof}
		The assertions concerning $(\starM)$ and the distinguishedness of ${\pi}$
		follow from Proposition \ref{prop:cover}. Indeed, by Remark \ref{rmk:semismall},
		in order to apply Proposition \ref{prop:cover}, it suffices to check that the class of
		$X\times_{Y} X$ is distinguished. In the present situation of finite group
		quotient, $X\times_{Y} X$ is nothing but $\sum_{g}\Gamma_{g}$, which is
		distinguished in $\CH(X\times X)$ by $(\star_{G})$.
		
		As for the condition $(\starC)$, the \'etale case is treated in Proposition
		\ref{prop:etale}.
		Suppose $\pi: X\to Y$ is a degree $d$ cyclic covering branched along a divisor $D$ such that
		$D\in \DCH(X)$.
		% With $(\starM)$ having been established, by Proposition \ref{prop:Disting}, 
		In order to show that the natural marking on $Y$ satisfies
		$(\starC)$, it suffices to show by the projection formula that
		$\pi^*\operatorname{ch}(T_Y)$ is distinguished. We have a short exact sequence
		$$0 \longrightarrow T_X \longrightarrow \pi^*T_Y \longrightarrow O_D(dD)
		\longrightarrow 0.$$ Since $X$ satisfies $(\starC)$, it is enough to show that
		$\operatorname{ch}(O_D(dD))$ belongs to $\DCH(X)$. Now $O_D(dD)$ fits into the
		short exact sequence 
		$$0\longrightarrow O_X((d-1)D) \longrightarrow O_X(dD) \longrightarrow O_D(dD)
		\longrightarrow 0.$$  Since the class of the divisor $D$ is assumed to belong to
		the $\Q$-subalgebra $\DCH(X)$, we find that indeed $\operatorname{ch}(O_D(dD)) =
		\operatorname{ch}(O_X(dD)) - \operatorname{ch}(O_X((d-1)D))$ belongs to
		$\DCH(X)$, which concludes the proof.
	\end{proof}

	\subsection{Hilbert squares and nested Hilbert schemes}
	
	\begin{prop}[Hilbert squares]\label{prop:Hilb2}
		Assume $X$ is a smooth projective variety with a marking that satisfies
		$(\star)$. Then $X^{[2]}$ has a natural marking that satisfies $(\star)$ and is
		such that the universal family $\{(x,z) : x\in \mathrm{Supp}(z) \} \subseteq
		X\times X^{[2]}$ is distinguished (with respect to the product marking).
	\end{prop}
	\begin{proof}
		The product $X\times X$ is naturally endowed with the action of $G:=\Z/2$ that
		switches the factors, and the locus of fixed points is the diagonal, which is
		isomorphic to $X$. By Remark \ref{R:perm}, the product marking on $X\times X$
		satisfies $(\star_G)$. Therefore, we may apply Proposition \ref{prop:Blup} to
		obtain a marking on the blow-up $\widetilde{X\times X}$ of $X\times X$ along the
		diagonal that satisfies $(\star)$ and $(\star_G)$. Now $X^{[2]}$ is the quotient
		of the latter blow-up by the cyclic action of $\Z/2$. Thus Proposition
		\ref{P:quotient} provides a marking for $X^{[2]}$ that satisfies $(\star)$. 
		
		Finally, we show that the universal family $Y:= \{(x,z) : x\in \mathrm{Supp}(z)
		\}$ is distinguished. First note that $Y$ is isomorphic to $\widetilde{X\times
			X}$, so that $Y$ is endowed with the natural marking coming from that of $X$. In
		order to conclude, we only need to show that the graph $\Gamma$ of the inclusion
		morphism $i:Y \hookrightarrow X\times X^{[2]}$ is distinguished. 
%		Since the
%		quotient morphism $\widetilde{X\times X} \to X^{[2]}$ is distinguished
%		(Proposition~\ref{P:quotient}), it is enough to show that the pull-back
%		$\Gamma'$ of $\Gamma$ to $Y\times X \times Y$ is distinguished. But then, this
%		is clear since $\Gamma'$ consists of the  two irreducible components
%		$\{((x,y),x,(x,y) : x,y\in X \}$ and $\{(y,x),x,(x,y) : x,y\in X\}$, and the
%		cycle classes of both components are distinguished by Corollary \ref{C:diag} and
%		Proposition \ref{prop:Blup}.
 This is clear because the components $Y\to X$ and $Y\to X^{[2]}$ of $i$, which consist of the composition $\widetilde{X\times
 	X} \to X\times X \to X$ and the quotient morphism $\widetilde{X\times X} \to X^{[2]}$, are distinguished.
	\end{proof}
	
	Recall that by a result of Cheah \cite{MR1616606}, for a smooth projective
	variety $X$ of dimension $\geq 3$, the only smooth nested Hilbert schemes of
	finite length subschemes on $X$ are $X^{[2]}$, $X^{[3]}$, $X^{[1, 2]}$ and
	$X^{[2,3]}$. By the same method, we have\,:
	\begin{prop}[Nested Hilbert schemes]\label{prop:NestHilb}
		Assumption is as in Proposition \ref{prop:Hilb2}.  Then $X^{[1,2]}$ and
		$X^{[2,3]}$ have natural markings satisfying $(\star)$ and are such that the
		classes of the universal subschemes are distinguished.
	\end{prop}
	\begin{proof}It is clear that $X^{[1,2]}$ is isomorphic to $\widetilde{X\times
			X}$, the blow-up of $X\times X$ along the diagonal, hence satisfies $(\star)$ by
		Proposition \ref{prop:Blup}. Similarly, $X^{[2,3]}$ is isomorphic to the blow-up
		of $X\times X^{[2]}$ along the universal subscheme $Y$. As is mentioned in the
		proof of the previous proposition, $Y$ is isomorphic to $X^{[1,2]}$ hence to
		$\widetilde{X\times X}$, thus it has a marking satisfying $(\star)$. As
		$X^{[2]}$ is endowed with the marking in Proposition \ref{prop:Hilb2}, $X\times
		X^{[2]}$ is endowed with the product marking satisfying $(\star)$ by Proposition
		\ref{P:products}. Moreover, the Chern classes of the normal bundle of $Y$ in
		$X\times X^{[2]}$ are distinguished since they are polynomials of the Chern classes
		of $T_{Y}$, of $T_{X}$ pulled-back to $Y=\widetilde{X\times X}$ \emph{via} the
		first projection and of $T_{X^{[2]}}$ pulled-back to $Y$ \emph{via} the $\Z/2$
		quotient map (\emph{cf.} the computation in \cite[Theorem 6.1]{SV2}), which are
		all distinguished by Propositions \ref{prop:Blup} and \ref{P:quotient}. Again by
		Proposition \ref{prop:Blup}, $X^{[2,3]}$ has a marking satisfying $(\star)$. The
		assertions about the universal subschemes follow from Corollary \ref{C:diag}.
	\end{proof}
	\begin{rmk}[Hilbert cubes]
		An argument similar as above combined with the explicit description of the
		Hilbert cube $X^{[3]}$ in \cite{SV2} shows that $X^{[3]}$ satisfies $(\star)$
		once $X$ does. Indeed, $X^{[3]}$ is constructed from $X^{3}$ in five steps
		(\emph{cf.} \cite{SV2} or \cite{MR1219698})\,: the first three are successive
		blow-ups of $X^{3}$, each time along a center satisfying $(\star)$ with normal
		bundle having distinguished Chern classes\,; the fourth step is a quotient map
		by a distinguished cyclic $\Z/3$-action\,; the final step is a blow-down of
		divisor with distinguished normal bundle to a center satisfying $(\star)$. Thus
		using Propositions \ref{P:products}, \ref{prop:PB}, \ref{prop:Blup},
		\ref{P:quotient} and Corollary \ref{C:diag} repeatedly in the first four steps,
		and using in the final step the analogue of the technical \cite[Lemma~6.4]{SV2}
		(with $\CH(-)_{(0)}$ replaced by $\DCH(-)$), one can obtain a marking of
		$X^{[3]}$ satisfying $(\star)$. The details are left to the interested reader.
	\end{rmk}
	
	\subsection{Birational transforms for hyper-K\"ahler varieties}\label{s:birational}
	Using Huybrechts' fundamental result \cite{MR1664696} on deformation equivalence
	between birational hyper-K\"ahler varieties, Rie{\ss}  \cite{MR3268859} shows
	that the Chow rings of birational hyper-K\"ahler varieties are isomorphic.
	Actually her proof yields the following more precise result\,:
	
	\begin{thm}[{Rie{\ss}  \cite[\S 3.3 and Lemma 4.4]{MR3268859}}]\label{thm:Riess}
		Let $X$ and $Y$ be $d$-dimensional irreducible holomorphic symplectic varieties.
		If they are birational, then there exists a correspondence $Z\in \CH_{d}(X\times
		Y)$ such that 
		\begin{enumerate}[$(i)$]
			\item $(Z\times Z)_{*}: \CH_{d}(X\times X)\to \CH_{d}(Y\times Y)$ sends
			$\Delta_{X}$ to $\Delta_{Y}$;
			\item $(Z\times Z\times Z)_{*}: \CH_{d}(X\times X\times X)\to \CH_{d}(Y\times
			Y\times Y)$ sends $\delta_{X}$ to $\delta_{Y}$.
			\item $Z_{*}: \CH(X)\to \CH(Y)$ sends $c_i(X)$ to $c_{i}(Y)$ for any $i\in \N$;
			\item $Z$ induces an isomorphism of algebra objects $\h(X)\to \h(Y)$ in $\CHM$
			with inverse given by ${}^{t}Z$.
		\end{enumerate}
		In particular, $Z$ induces an isomorphism between their Chow rings (\emph{resp.}
		cohomology rings).
	\end{thm}
	
	\begin{cor}\label{cor:BirHK}
		Let $X$ and $Y$ be $d$-dimensional irreducible holomorphic symplectic varieties
		that are birationally equivalent. If $X$ has a marking that satisfies $(\star)$,
		then so does $Y$.
	\end{cor}
	\begin{proof}
		Let $Z\in \CH_{d}(X\times Y)=\Hom(\h(X), \h(Y))$ be the correspondence in
		Theorem \ref{thm:Riess}.
		Let $\phi: \h(X)\lra{\isom} M$ be a marking satisfying $(\star)$. Then we
		consider the marking $\psi=\phi\circ Z^{*}: \h(Y)\lra{\isom} M$. The fact that
		$\psi$ satisfies the condition $(\starM)$ and $(\starC)$ follows from Theorem
		\ref{thm:Riess} $(ii), (iii)$ respectively, together with the corresponding
		property of $\phi$.
	\end{proof}

	\section{Examples of varieties satisfying the condition
		$(\star)$}\label{sect:Examples}
	We provide in this section some examples of varieties satisfying the condition
	$(\star)$. Together with the operations in Section \ref{sect:examples}, we obtain
	even more examples. Thanks to Proposition \ref{prop:Disting}, the rational Chow
	ring of each of them possesses a subalgebra consisting of distinguished cycles,
	which is mapped isomorphically to the numerical Chow ring and contains all Chern
	classes of the variety.

	\subsection{Easy examples}\label{subsect:EasyEx}

	First of all, as $(\star)$ is certainly a property preserved by isomorphisms of
	algebraic varieties, we have by O'Sullivan's Theorem \ref{thm:SD}\,:
	\begin{lem}
		Any abelian torsor, that is, a variety isomorphic to an abelian variety,
		satisfies $(\star)$.
	\end{lem}

	Another set of examples generalizes the projective spaces\,:
	\begin{prop} \label{prop:trivial}
		Let $X$ be a smooth projective variety over a field $k$ and let $\Omega$ be a universal domain containing $k$. Assume that $X$ satisfies at least one of the following
		conditions\,:
		\begin{enumerate}
			\item $X\isom G/P$ is a homogeneous variety, where $G$ is a linear algebraic
			group and $P$ is a parabolic subgroup.
			\item $X$ is a toric variety.
			\item The bounded derived category $D^{b}_{coh}(X)$ has a full exceptional collection.
			\item The cycle class map $\CH^{*}(X_\Omega)\to H^{*}(X_\Omega, \Q_\ell)$ is injective for some prime $\ell \neq \mathrm{char}\, k$.
			\item The Chow group $\CH^*(X_\Omega)$ is a finite-dimensional $\Q$-vector space.
		\end{enumerate}
		Then $X$ satisfies $(\star)$.
	\end{prop}
	\begin{proof}
		Actually any of these conditions ensures that the Chow motive of $X$ is of
		Lefschetz-Tate type\,: 
		$$\h(X)\isom \bigoplus_{i}\1(a_{i}),$$
		with $a_{i}\in \Z$. It is well-known for (1) and (2)\,; while for (3) it is
		established in \cite{MR3050698} and \cite{MR3331729}. For (4), it is the main
		result of \cite{MR1316539}, see also \cite[\S 2.2]{MR3186044} for a recent
		account, and for (5), it is proven in \cite{KimuraRep}, \cite{VialRep}. 
	\end{proof}

	\subsection{Curves}
	Recall that the smooth projective curves of genus 0 and 1 are covered in \S
	\ref{subsect:EasyEx}. We consider in this subsection curves of higher genera.
	
	Let $C$ be a smooth projective curve with genus $g\geq 2$.  Its Jacobian variety
	$JC$ is a principally polarized abelian variety of dimension $g$ with origin
	denoted by $O$ and theta divisor denoted by $\Theta\in \CH^{1}(JC)$, which is
	always assumed to be symmetric. By choosing a base point $z\in C$,
	% \comm{(or just a degree-one 0-cycle $z\in \CH^{1}(C)$?)}, 
	we have the Abel--Jacobi embedding\,:
	\begin{eqnarray*}
		\iota_{z}:C&\inj& JC\\
		x&\mapsto& \calO_{C}(x-z).
	\end{eqnarray*}
	Associated to $z$, there is also the motivic decomposition of $C$\,:
	$$\h(C)=\h^{0}(C)\oplus \h^{1}(C)\oplus \h^{2}(C),$$
	where $\h^{0}(C):=(C, z\times C,0)\isom \1$, $\h^{2}(C):=(C, C\times z,0)\isom
	\1(-1)$ and $\h^{1}(C):=(C, \Delta_{C}-z\times C-C\times z,0)$.
	
	\begin{prop}\label{prop:curve}
		Let $C$ be a smooth projective curve of genus $g\geq 2$. If there exists a
		point $z\in C$ 
		%	\comm{( or a degree-one 0-cycle $z$ ?)} 
		such that $\iota_{z}(C)\in \CH_{1}(JC)$ is symmetrically
		distinguished \footnote{By Remark \ref{R:symdistbeauville},
			this condition is equivalent to $\iota_z(C) \in \CH_1(JC)_{(0)}$.}, then $C$ satisfies the condition $(\star)$.
		%in Definition \ref{def:Star}.
	\end{prop}
	\begin{proof}
		Let us fix $z$ and simply write $\iota:=\iota_{z}$ and $C:=\iota_{z}(C)$.
		Assume that $C\in \CH_{1}(JC)$ is symmetrically distinguished. Since the
		1-cycles $C$ and $ \frac{1}{(g-1)!}\Theta^{g-1}$ are numerically equivalent and
		symmetrically distinguished, they are actually equal (\emph{i.e.}, rationally
		equivalent), thanks to Theorem \ref{thm:SD}. 
		
		Deninger and Murre construct in \cite{DM} a canonical motivic
		decomposition $$\h(JC)=\oplus_{i=0}^{2g}\h^{i}(JC).$$ Let $\pi^{i} \in
		\CH^{g}(JC\times JC)$ be the projector corresponding to $\h^{i}(JC)$. For
		example, $\pi^{0}=[O]\times JC$ and $\pi^{2g}=JC\times [O]$. See
		\cite{MR1265530} for the explicit formulae of the other projectors $\pi^{i}$.
		One important feature, easily seen from Theorem \ref{thm:SD}, is that they are
		all symmetrically distinguished.
		
		We claim that $\Gamma_{\iota}=:\iota_{*}: \h(C)\to \h(JC)(g-1)$ induces
		isomorphisms\,:
		\begin{itemize}
			\item $\h^{2}(C)\xrightarrow{\isom} \h^{2g}(JC)(g-1):=(JC, JC\times [O],
			g-1)$\,;
			\item $\h^{1}(C)\xrightarrow{\isom} \h^{2g-1}(JC)(g-1):=(JC, \pi^{2g-1},
			g-1)$\,;
			\item $\h^{0}(C)\xrightarrow{\isom} L^{g-1}\h^{0}(JC)(g-1):=(JC,
			\frac{1}{g!}\Theta\times\Theta^{g-1}, g-1)$\,; the latter is a direct summand of
			$\h^{2g-2}(JC)(g-1)$ in the Lefschetz decomposition constructed by K\"unnemann
			in \cite{MR1223225},
		\end{itemize}
		where $L$ is the Lefschetz operator (see \cite{MR1223225}). Indeed, all these
		morphisms are in the Kimura category $\M^{ab}$ (see \cite{MR2107443}). The
		functor $\M^{ab}\to \bar{\M^{ab}}$ is therefore conservative (\emph{cf.}
		\cite[Corollary 3.16]{MR2167204}). One checks easily that these morphisms are
		isomorphisms modulo homological, thus \emph{a fortiori} numerical, equivalence. 
		
		Putting them together, we have a marking for $C$\,:
		$$\phi:=\iota_{*}: \h(C)\xrightarrow{\isom} M:=(JC, JC\times
		[O]+\pi^{2g-1}+\frac{1}{g!}\Theta\times\Theta^{g-1}, g-1).$$
		Let us show $(\starM)$\,: since the inclusion of the direct summand $M$ into
		$\h(JC)$ is clearly symmetrically distinguished, to show that $\phi^{\otimes
			3}_{*}(\delta_{C})$ is symmetrically distinguished, it suffices to show that
		$\iota^{3}_{*}:\CH_{1}(C^{3})\to \CH_{1}(JC^{3})$ sends the small diagonal
		$\delta_{C}$ to a symmetrically distinguished cycle of $JC\times JC\times JC$.
		However, by the following commutative diagram
		\begin{equation*}
		\xymatrix{
			C \ar@{^{(}->}[r]^-{\delta_{C}} \ar@{^{(}->}[d]_{\iota} & C\times C\times C
			\ar@{^{(}->}[d]^{\iota^{3}}\\
			JC \ar@{^{(}->}[r]^-{\delta_{JC}}& JC\times JC\times JC
		}
		\end{equation*}
		we have that $\iota^{3}_{*}(\delta_{C})=\delta_{JC,*}(\iota(C))$ is symmetrically
		distinguished by the assumption and Theorem \ref{thm:SD}.  
		
		The condition $(\starC)$ on Chern classes follows from $(\starM)$ since $C$ is a curve (Corollary \ref{cor:TopChern}). 
%		It remains to show the condition $(\starC)$ on Chern classes. The fundamental
%		class $\1_{C}$ being clearly distinguished (Remark
%		\ref{rmk:FundCl}), it is enough to show that in $\CH_{0}(JC)$ we have
%		\begin{equation}\label{eqn:pushK}
%		\iota_{*}(K_{C})=(2g-2) [O].
%		\end{equation}
%		It is well-known that
%		$$\iota^{*}(\Theta)=\frac{1}{2}K_{C}+z \in \CH_{0}(C).$$ 
%		Applying $\iota_{*}$ to this equality,
%		we obtain by the projection formula
%		$$\frac{1}{2}\iota_{*}(K_{C})+\iota_{*}(z)=\Theta\cdot C=\Theta\cdot
%		\frac{\Theta^{g-1}}{(g-1)!}=g[O].$$
%		The desired equality (\ref{eqn:pushK}) then follows by observing that
%		$\iota(z)=O$. 
%		%	\comm{(Is this true if $z$ is not a point of $C$?) }
	\end{proof}
	
	\begin{cor}\label{cor:hyperelliptic}
		All hyperelliptic curves satisfy the condition $(\star)$.
	\end{cor}
	\begin{proof}
		For a hyperelliptic curve $C$, choose any Weierstrass point to define the
		Abel--Jacobi embedding, then the involution $[-1]$ on $JC$ preserves $C$ and acts
		on $C$ by the hyperelliptic involution. By \cite[Proposition 2.1]{Tavakol}, in
		the Beauville decomposition of $\CH^{g-1}(JC)$, the class of $C$ belongs to
		$\CH^{g-1}(JC)_{(0)}$. On the other hand, $\CH^{g-1}(JC)_{(0)}$ is the Fourier
		transform \cite{MR826463} of $\CH^{1}(JC)_{(0)}$ which maps isomorphically to
		$\bar\CH^{1}(JC)$. Therefore, the natural cycle class map
		$\CH^{g-1}(JC)_{(0)}\to \bar\CH^{g-1}(JC)$ is also an isomorphism. Consequently,
		all cycles in $\CH^{g-1}(JC)_{(0)}$, in particular the class of $C$, are
		symmetrically distinguished. One can now conclude by invoking Proposition
		\ref{prop:curve}.
	\end{proof}
	\begin{rmk}
		The case of hyperelliptic curves is mentioned in \cite[\S 6.3]{MR2795752}.
	\end{rmk}

	\begin{rmk}[Hilbert schemes of a hyperelliptic curve] \label{rem:hyperelliptic}
		Recall that the Hilbert scheme of length-$n$ subschemes on a smooth curve $C$ is
		nothing but the $n$-th symmetric power $C^{(n)}$ of $C$. Now if $C$ satisfies
		$(\starM)$, then by Proposition \ref{P:products}, $C^n$ satisfies $(\starM)$ and
		by Proposition \ref{P:quotient}, $C^{(n)}$ satisfies $(\starM)$. By Corollary \ref{cor:TopChern}, $C$ also satisfies $(\starC)$, and the same computation as in \cite[p.\,95]{SV} shows
		that $C^{(n)}$ satisfies $(\starC)$. Therefore, it follows from Corollary
		\ref{cor:hyperelliptic} that the Hilbert schemes of a hyperelliptic curve
		satisfy $(\star)$.
	\end{rmk}

	\subsection{Fermat hypersurfaces}\label{s:Fermat}
	An important class of (higher dimensional) varieties whose motive is known to be
	of abelian type is provided by the Fermat hypersurfaces, by using the inductive
	structure discovered by Shioda--Katsura \cite{MR526513}. Note that Proposition
	\ref{prop:trivial} implies that smooth quadric hypersurfaces satisfy $(\star)$
	since their Chow groups are finite dimensional vector spaces.
	
	In the sequel of this subsection, we fix a degree $d\geq 3$ and, for any $r\in
	\N$, we let $X_{r}$ denote the Fermat hypersurface of degree $d$ in
	$\PP^{r+1}$\,:
	$$X_{r}:=\{(x_{0}, \cdots, x_{r+1}) ~\vert~ x_{0}^{d}+\cdots+x_{r+1}^{d}=0\}\subset \PP^{r+1}.$$
	
	Recall the inductive structure\, (\emph{cf.} \cite[Theorem 1]{MR526513})\,: let
	$\epsilon$ be a (fixed) $d$-th root of $-1$ and $\zeta$ be a (fixed) $d$-th root
	of unity. For any $r, s\in \N$, we have the following commutative diagram\,:
	\begin{equation}\label{diag:Induction}
	\xymatrix{
		E \cart\ar@{^{(}->}[r]  \ar[d]& Z  \ar[r] \ar[dr]^{\psi}\ar[d]^{\beta}&
		Z/\mu_{d}  \cart\ar[d]^{\tau}& \ar@{_{(}->}[l] X_{r-1}\times \PP^{s} \coprod
		\PP^{r} \times X_{s-1} \ar[d]\\
		X_{r-1}\times X_{s-1} \ar@{^{(}->}[r]^-{i_{r}\times i_{s}}& X_{r}\times X_{s}
		\ar@{-->}[r]^{\varphi}&  X_{r+s} & \ar@{_{(}->}[l] X_{r-1} \coprod X_{s-1} 
	}
	\end{equation}
	where 
	$i_{r}: X_{r-1} \inj X_{r}$ is the embedding given by
	$(x_{0}, \ldots, x_{r})\mapsto (x_{0}, \ldots, x_{r}, 0)$\,;
	$\varphi : \left((x_{0}, \ldots, x_{r+1}), (y_{0}, \ldots,
	y_{s+1})\right)\mapsto (y_{s+1}x_{0}, \ldots, y_{s+1}x_{r}, \epsilon
	x_{r+1}y_{0}, \ldots, \epsilon x_{r+1}y_{s})$\,;  \linebreak 
	$\beta$ and $\tau$ are blow-ups\,; the action of $\mu_{d}$ on the blow-up $Z$ is
	induced by its action on $X_{r}$ and $X_{s}$ given by $(x_{0}, \ldots,
	x_{r+1})\mapsto (x_{0}, \ldots, x_{r}, \zeta x_{r+1})$ and $(y_{0}, \ldots,
	y_{s+1}) \mapsto (y_{0}, \ldots, y_{s},\zeta y_{s+1})$, respectively.\medskip
	
	The main result of this subsection is the following.
	
	\begin{prop}[Fermat cubics]\label{prop:Fermat}
		If $d=3$, then there exist, for all $r\in \N$, a marking
		$\phi_{r}: \h(X_{r})\lra{\isom} M_{r}$, for the cubic Fermat hypersurface
		$X_{r}$, such that 
		\begin{enumerate}[$(i)$]
			\item The embedding $i_{r}: X_{r-1}\inj X_{r}$ is distinguished (Definition \ref{def:DisCor})\,;
%			\footnote{that
%				is, its graph is distinguished, or equivalently, the induced morphism
%				$\phi_{r-1}\circ i_{r}^{*}\circ\phi_{r}^{-1}: M_{r}\to M_{r-1}$ is symmetrically
%				distinguished, \emph{cf.} the beginning of the proof of Proposition
%				\ref{prop:Blup}.}
			\item The action of $\mu_{d}$ on $X_{r}$ is distinguished, \emph{i.e.}, $\phi_r$ satisfies
			$(\star_{\mu_{d}})$\,;
			\item $\phi_{r}$ satisfies the condition $(\star)$ of Definition \ref{def:Star}.
		\end{enumerate}
		In particular, all Fermat cubic hypersurfaces satisfy the condition $(\star)$.
	\end{prop}
	\begin{proof}
		We proceed by induction on $r$.
		For $r=1$, $X_{1}=\{x_{0}^{3}+x_{1}^{3}+x_{2}^{3}=0\}$ is a cubic curve in
		$\PP^{2}$\,; by fixing an origin, it becomes an elliptic curve. We fix $(-1, 1,
		0)$ as its origin. Trivially, $X_1$ satisfies $(\star)$
		(\S\ref{subsect:EasyEx}). The embedding  $X_{0}\inj X_{1}$ is given by three
		points $(-1, 1,0), (-\zeta, 1, 0), (-\zeta^{2}, 1, 0)$, which are of 3-torsion\footnote{In fact, the nine 3-torsion points of the Fermat elliptic curve are
			exactly its intersection with the coordinate axes $(x_{0}=0), (x_{1}=0)$ and
			$(x_{2}=0)$. Indeed, these nine points lie on 12 lines. Each line contains three
			of these points and each point lies on four lines. Now use the fact that the sum
			of the three points in the intersection of any line with the elliptic curve is
			the hyperplane section class, we easily deduce that 3 times any of the nine
			points is the hyperplane section class. Hence they are all 3-torsion points if
			any one of them is fixed as the origin.}, therefore distinguished. As for the
		action of $\mu_{d}$, which is given by $(x_{0}, x_{1}, x_{2})\mapsto (x_{0},
		x_{1}, \zeta x_{2})$, it is clearly an automorphism of abelian variety hence
		also distinguished.
		
		Assuming the assertions $(i) - (iii)$ for $r \leq n$, let us establish them for
		$r = n+1$. We set in the sequel $s=1$ in the diagram (\ref{diag:Induction}) and also $\epsilon=-1$. By
		the induction hypothesis and the fact that distinguished morphisms are stable
		under products, the embedding $X_{n-1}\times X_{0}\inj X_{n}\times X_{1}$ is
		distinguished. Therefore $Z$ satisfies $(\star)$ by Proposition \ref{prop:Blup}.
		Again by the induction hypothesis, the action of $\mu_{d}$ on $X_{n}\times
		X_{1}$ is distinguished with distinguished ramification locus, which implies by
		Proposition \ref{P:quotient} that $Z/\mu_{d}$ satisfies $(\star)$. 
		We now claim that the marking on $X_{n+1}$ defined via~$\tau$ satisfies $(\star)$. We thank the referee for providing the following argument. For $(\starM)$ it is enough by Proposition~\ref{prop:cover} to show that
		$${}^t\Gamma_\tau \circ \Gamma_\tau =(\tau \times \tau)^*(\Delta_{X_{n+1}})$$
		is distinguished. The exceptional divisors for $\tau$ are $E_0, E_1, E_2, E_3$ with
		$$E_0 =X_{n-1} \times \PP^1 = X_{n-1} \times (X_1/\mu_3)$$ 
		and $E_i$ for $i>0$ a component $\PP^n$ of $\PP^n \times X_0$. We have
		$$(\tau \times \tau)^*(\Delta_{X_{n+1}}) = \Delta_{Z/\mu_3} +\alpha_0+\alpha,$$		
		where $\alpha_0$ is the push-forward along $E_0 \times E_0 \to (Z/\mu_3) \times (Z/\mu_3)$ of 
		$$E_0 \times_{X_{n-1}} E_0 = \Delta_{X_{n-1}} \times \PP^1\times \PP^1,$$
		and $\alpha$ is supported on $\coprod_{i>0} E_i \times E_i$. Both $\Delta_{Z/\mu_3}$ and $\alpha_0$ are distinguished, and $\alpha$ is distinguished because for $i > 0$ every cycle on $E_i \times E_i = \PP^n \times \PP^n$ is. Finally $(\starC)$ follows from  \cite[Theorem~15.4]{MR1644323}.
%		
%		
%		
%		
%		As $Z/\mu_{d}$ is a blow-up of $X_{n+1}$ along a smooth subscheme, 
%		$\tau^{*}\colon \h(X_{n+1})\hookrightarrow\h(Z/\mu_{d})$ is injective and is retracted by $\tau_{*}$. In other words, $\h(Z/\mu_{d})$ is canonical a direct summand of $\h(X_{n+1})$. Therefore by the idempotent completeness of the category $\left(\M^{ab}_{sd}, \text{ s.d. morphisms }\right)$ (Lemma \ref{lemma:PsAb}), the marking for $Z/\mu_{d}$ satisfying $(\star)$ determines a marking for $X_{n+1}$ such that $\tau^{*}$ and $\tau_{*}$ are distinguished. As $\h(X_{n+1})$ is moreover a subalgebra
%		object of $\h(Z/\mu_{d})$ (\emph{cf.} Lemma \ref{lemma:sdiagBl}), we have the following commutative diagram :
%		\begin{eqnarray*}
%		\xymatrix{
%		 \h(Z/\mu_{d})\otimes \h(Z/\mu_{d}) \ar[r]^-{\delta_{Z/\mu_{d}}}& \h(Z/\mu_{d})\ar@{->>}[d]_{\tau_{*}}\\
%		 \h(X_{n+1})\otimes \h(X_{n+1})\ar@{^(->}[u]^{\tau^{*}\otimes \tau^{*}} \ar[r]^-{\delta_{X_{n+1}}}& \h(X_{n+1}).
%		}
%		\end{eqnarray*}
%		By Lemma \ref{lemma:InterpretationStar}, $\delta_{Z/\mu_{d}}$ is a distinguished morphism (and by construction the natural inclusion and projection in the above diagram are distinguished), and  $\delta_{X_{n+1}}$ is henceforth distinguished too. Therefore this marking of $X_{n+1}$ satisfies $(\starM)$ and the morphism $\psi$ is
%		distinguished. Alternately, one could have also argued by using Proposition \ref{prop:cover}. That this marking satisfies $(\starC)$ follows from \cite[Theorem~15.4]{MR1644323}. 
%		
%		
%		
		In particular, $(iii)$ for $r=n+1$ is proven.\\
		For $(i)$, we have the following commutative diagram, where $i$ is the embedding
		determined by the point $(1,0,-\zeta)\in X_{1}$.
		\begin{equation*}
		\xymatrix{
			Z \ar[d]^{\beta} \ar[dr]^{\psi}&\\
			X_{n}\times X_{1} \ar@{-->}[r]^{\varphi} & X_{n+1}\\
			X_{n} \ar@{^{(}->}[u]^{i} \ar@{^{(}->}[ur]_{i_{n+1}}&
		}
		\end{equation*}
		Since $(1,0,-\zeta)$ is a torsion point of $X_{1}$, $i^{*}$ is distinguished.
		Therefore, with $\psi$ and $\beta$ being distinguished by construction,
		$i_{n+1}^{*}=i_{n+1}^{*}\circ \psi_{*}\circ\psi^{*}=i^{*}\circ\beta_{*}\circ
		\psi^{*}$ is also distinguished.\\
		Finally for $(ii)$, the action of $\mu_{d}$ on $X_{n+1}$ comes, \emph{via} the
		diagram (\ref{diag:Induction}), from the action of $\mu_{d}$ on $X_{1}$ which is
		given by
		$(y_{0}, y_{1}, y_{2})\mapsto (y_{0}, \zeta y_{1}, y_{2})$. It is clearly an
		automorphism of abelian variety hence is distinguished.
	\end{proof}
	
	So far, we are not able to determine whether other Fermat hypersurfaces satisfy
	$(\star)$ but we would like to make the following conjecture\,:
	\begin{conj}\label{conj:Fermat}
		The Fermat hypersurfaces which are Calabi--Yau or Fano, \emph{i.e.}, $d\leq r+2$,
		satisfy the condition $(\star)$.
	\end{conj}

	\begin{rmk}
		The conclusion of Conjecture \ref{conj:Fermat} can not hold in general for Fermat hypersurfaces
		of general type\,; \emph{cf.} Proposition \ref{prop:Fermatnot} (together with
		Proposition \ref{prop:multmarking}) below for counter-examples in the case of Fermat curves
		starting from degree 4.
	\end{rmk}
	
	\begin{rmk}
		It is interesting to notice that for $d=4$, we know that the quartic Fermat
		surface satisfies $(\star)$ for a different reason\,: it is a Kummer surface
		(\emph{cf.} \cite[Chapter 14, Example 3.18]{MR3586372}) and Proposition
		\ref{prop:kummer} applies. One could therefore show Conjecture~\ref{conj:Fermat} for $d=4$
		by a similar induction argument  as in Proposition \ref{prop:Fermat} once we
		know the case of Fermat quartic threefold (and some natural compatibilities with
		the Fermat quartic surface). \end{rmk}

	\subsection{K3 surfaces with large Picard number}\label{s:k3large}
	
	While K3 surfaces are expected to have motive of abelian type \emph{via} the
	Kuga--Satake construction, this has only been established in scattered cases.
	This includes Kummer surfaces, and \cite[Theorem~2]{Pedrini}  K3 surfaces with
	Picard number $\geq 19$.

	\subsubsection{Kummer surfaces} \label{S:kumsurface}
	By definition the Kummer surface $K_1(A)$ attached to the abelian surface $A$ is
	the fiber over $0$ of the morphism $A^{[2]} \to A^{(2)} \to A$, which is the
	composition of the sum morphism $A^{(2)}\to A$ with the Hilbert--Chow morphism
	$A^{[2]} \to A^{(2)}$.

	\begin{prop}\label{prop:kummer}
		A Kummer surface admits a marking that satisfies $(\star)$.
	\end{prop}
	\begin{proof}
		The Kummer surface $K_1(A)$ has the following alternative description\,: the
		$[-1]$-involution on $A$ induces an involution, denoted $\iota$, on the blow-up
		$\widetilde{A}$ of $A$ along its subgroup of $2$-torsion points, and $K_1(A)$ 
		is the $\Z/2$-quotient of $\widetilde{A}$ for that action. By Proposition
		\ref{prop:Blup}, $(\widetilde{A},\Z/2)$ has a marking that satisfies $(\star)$.
		We can then conclude from  Proposition \ref{P:quotient} that $K_1(A)$ has a
		marking that satisfies $(\star)$.
	\end{proof}
	
	Later on (\emph{cf.} Proposition \ref{P:genK}), we will generalize Proposition
	\ref{prop:kummer} by establishing that generalized Kummer varieties admit a
	marking that satisfies $(\star)$.
	
	\subsubsection{ K3 surfaces with Picard number $\geq 19$}
	Such K3 surfaces admit \cite{Morrison} a Nikulin involution (that is, a
	symplectic involution) with quotient birationally equivalent to a Kummer
	surface. 
	
	\begin{prop}\label{prop:k3large}
		A K3 surface with Picard number $\geq 19$ admits a marking that satisfies
		$(\star)$.
	\end{prop}
	\begin{proof} Let $X$ be a K3 surface with a Nikulin involution\,; by \cite[\S
		5]{Nikulin} $X$ has eight isolated fixed points, which we denote $Q_1,\ldots,
		Q_8$. Let $\pi : X \to X/\iota$ be the quotient morphism\,; $X/\iota$ has
		ordinary double points at the points $P_i := \pi(Q_i)$, so that if $f : Y \to
		X/\iota$ denotes the minimal resolution, then the exceptional divisors of $f$
		are smooth rational $(-2)$-curves $C_i := f^{-1}(P_i)$.
		
		Let $X$ now be a K3 surface with Picard number $\geq 19$. According to
		\cite[Corollary~6.4]{Morrison}, $X$ admits a Shioda--Inose structure, meaning
		that $X$ admits a Nikulin involution $\iota$ such that $Y$ is a Kummer surface
		and such that $f^*\pi_*$ induces a Hodge isometry $T_X(2) \simeq T_Y$, where
		$T_X$ refers to the transcendental lattice of $X$. The latter was upgraded to an
		isomorphism of Chow motives by Pedrini \cite[Theorem~2]{Pedrini}. Precisely,
		given $S$ a K3 surface, let us denote $o_S$ the Beauville--Voisin zero-cycle\,;
		\emph{cf.}~\cite{BVK3}. We fix a basis $\{D_j\}$ of $\CH^1(S)$, and denote
		$\{D_j^\vee\}$ the dual basis with respect to the intersection product. We then
		define the idempotent correspondences  $\pi^0_S := o_S\times S$, $\pi^4_S :=
		S\times o_S$,  $\pi_S^{2,alg} := \sum_j D_j^\vee\times D_j$, and $\pi^{2,tr}_S
		:= \Delta_S - \pi_S^0 - \pi_S^4 - \pi_S^{2,alg}$. The motive $\h^{alg}(S) :=
		(S,\pi^0_S + \pi^{2,alg}_S + \pi_S^4)$ is the algebraic motive of $S$ (it is
		isomorphic to a direct sum of Lefschetz--Tate motives), and the motive
		$\mathfrak{t}^2(S) := (S,\pi_S^{2,tr})$ is the transcendental motive of $S$.
		Pedrini \cite{Pedrini} showed that $f^*\pi_*$ induces an isomorphism of motives
		$\mathfrak{t}^2(X) \simeq \mathfrak{t}^2(Y)$ (with inverse
		$\frac{1}{2}\pi^*f_*$).
		
		We fix a marking for the Kummer surface $Y$ that satisfies $(\star)$\,; such a
		marking does exist by Proposition \ref{prop:kummer}. Since $\DCH^1(Y) = \CH^1(Y)$,
		we have that the classes of the smooth rational curves $C_i$ are distinguished,
		and we also have that
		the projectors $\pi^0_Y, \pi^4_Y, \pi_Y^{2,alg}$ and $\pi^{2,tr}_Y$ are
		distinguished.
		Then we claim that the marking given by the decomposition $\h(X) \simeq
		\pi^*f_* \mathfrak{t}^2(Y) \oplus \h^{alg}(X)$ satisfies $(\star)$. That it
		satisfies $(\starC)$ is obvious since $c_1(X) = 0$ and since by \cite{BVK3},
		$c_2(X)$ is a multiple of $o_X$ and hence is mapped to zero in
		$\CH^2(\mathfrak{t}^2(Y))$. 
		%	By refined intersection \cite{MR1644323}, the cycle
		%$(f,f)^*(\pi,\pi)_*\Delta_X$ is supported on $(f,f)^{-1}(\pi,\pi)(\Delta_X) =
		%\Delta_Y \cup \bigcup_i C_i\times C_i$. Therefore the cycle
		%$(f,f)^*(\pi,\pi)_*\Delta_X$ is a linear combination of $\Delta_Y$ and the
		%$C_i\times C_i$. These cycles belong to $\DCH(Y\times Y)$, thereby establishing
		%$(\star1)$, \emph{i.e.} that $\Delta_X$ is distinguished. 
		%	Again, 
		By refined intersection \cite{MR1644323}, the cycle
		$(f,f,f)^*(\pi,\pi,\pi)_*\delta_X$ is supported on
		$(f,f,f)^{-1}(\pi,\pi,\pi)(\delta_X) = \delta_Y \cup \bigcup_i C_i\times
		C_i\times C_i$. Since $C_i$ is a smooth rational curve, we have that
		$\CH_{2}(C_i \times C_i\times C_i)$ admits $c_i\times C_i\times C_i$, $C_i\times
		c_i\times C_i$ and $C_i\times C_i\times c_i$ as a basis, where $c_i$ is any
		point on $C_i$. 
		The cycle $(f,f,f)^*(\pi,\pi,\pi)_*\delta_X$ is therefore a linear combination
		of $\delta_Y$ and, for $1\leq i\leq 8$, of $c_i\times C_i\times C_i$, $C_i\times
		c_i\times C_i$ and $C_i\times C_i\times c_i$. By \cite{BVK3}, the class of $c_i$
		in $\CH^2(Y)$ is the Beauville--Voisin zero-cycle $o_Y$\,; thus $c_i \in
		\DCH^2(Y)$. The cycles $c_i\times C_i\times C_i$, $C_i\times c_i\times C_i$ and
		$C_i\times C_i\times c_i$ therefore belong to $\DCH(Y\times Y\times Y)$ by
		Proposition \ref{prop:tensor}. Since $\delta_Y$ is distinguished, this
		establishes $(\starM)$, \emph{i.e.}, that $\delta_X$ is distinguished. 
	\end{proof}

	\subsection{(Nested) Hilbert schemes of surfaces, generalized Kummer varieties}
	In this subsection, we produce series of varieties satisfying $(\star)$.
	The first series of examples is given by the Hilbert schemes and (two-step)
	nested Hilbert schemes of points on a surface that satisfies $(\star)$,
	\emph{e.g.} an abelian surface, a Kummer surface (Proposition
	\ref{prop:kummer}), a K3 surface with Picard rank $\geq 19$ (Proposition
	\ref{prop:k3large}) or the product of two hyperelliptic curves (Corollary
	\ref{cor:hyperelliptic}). Note that by a result of Cheah \cite{MR1616606} the
	only nested Hilbert schemes of a smooth surface $S$ that are smooth are the
	Hilbert schemes $S^{[n]}$ and the nested Hilbert schemes $S^{[n,n+1]}$ for $n\in \N$.
	
	\begin{prop}\label{P:hilb}
		Let $S$ be a smooth projective surface that satisfies $(\star)$. Then, for any
		$n\in \N$, the Hilbert scheme of length-$n$ subschemes on $S$, denoted
		$S^{[n]}$, and the nested Hilbert scheme $S^{[n, n+1]}$, satisfy the condition
		$(\star)$.
	\end{prop}
	
	The second series of example is built from an abelian surface $A$\,: the
	associated Kummer K3 surface as well as its higher dimensional generalizations.
	Recall that the $n$-th \emph{generalized Kummer} variety (see \cite{MR730926})
	is the symplectic resolution of the quotient $A^{n+1}_{0}/{\fS_{n+1}}$, where
	$A^{n+1}_{0}$ is the abelian variety $\ker\left(+: A^{n+1}\to A\right)$, upon
	which the symmetric group acts naturally by permutations. 
	\begin{prop}\label{P:genK}
		For any $n\in \N$, the generalized Kummer variety $K_{n}(A)$ associated to an abelian surface $A$ satisfies the
		condition $(\star)$.
	\end{prop}

	The proofs of Propositions \ref{P:hilb}  and \ref{P:genK} will be given
	concomitantly in full in \S \ref{S:proof}. Note that the case of Kummer surfaces
	(which are the generalized Kummer varieties of dimension 2) was already treated
	in Proposition \ref{prop:kummer}. We start by recalling some results of de
	Cataldo and Migliorini \cite{cm} concerning the motives of Hilbert schemes of
	surfaces, or more generally that of a semi-small resolution.

	\subsubsection{The motive of semi-small resolutions}\label{subsubsect:semismall}
	
	Recall that a morphism $f: Y\to X$ is called \emph{semi-small} if for all
	integer $k\geq 0$, the codimension of the locus $\left\{x\in X : \dim
	f^{-1}(x)\geq k\right\}$ is at least $2k$. In particular, $f$ is generically
	finite. In \cite{cm}, assuming $f : Y\to X$ is a semi-small resolution with $Y$
	smooth and projective, de Cataldo and Migliorini computed the Chow motive of $Y$
	in terms of the Chow motives of projective compactifications of \emph{relevant
		strata} of
	$f$ provided these are finite group quotients of smooth varieties\,; we refer to
	\cite{cm} for a precise statement. In our case of interest, this has the
	following consequence. Suppose $S$ is a smooth projective surface and suppose
	$A$ is an abelian surface. Let us make some standard construction and fix the
	notation.
	
	Given a partition $\lambda=(\lambda_{1}\geq \cdots \geq
	\lambda_{|\lambda|})=(1^{a_{1}}\cdots n^{a_{n}})$ of a positive integer $n$
	where
	$a_{i}=\#\{j :  1\leq j\leq n\,;\lambda_{j}=i\}$ and where $|\lambda|
	:=a_1+\cdots+a_{n}$ denotes the length of $\lambda$, we define
	$\mathfrak{S}_\lambda :=  
	\mathfrak{S}_{a_{1}}\times\cdots\times\mathfrak{S}_{a_{n}}$. 
	We define   $S^{\lambda}$ to be $S^{|\lambda|}$, equipped with the natural
	action of $\mathfrak{S}_{\lambda}$ and with the natural morphism to $S^{(n)}$ by
	sending $(x_{1}, \cdots, x_{|\lambda|})$ to
	$\sum_{j=1}^{|\lambda|}\lambda_{j}[x_{j}]$.
%	We define $A^\lambda$  similarly\,; in addition $A^\lambda$ comes equipped with
%	the natural morphism $A^\lambda \to A^{(n)}:=
%	A^{n}/\mathfrak{S}_{n}, (x_1,\ldots, x_{|\lambda|}) \mapsto \sum_i a_i[x_i]$.
	We denote the quotient $$S^{(\lambda)} := S^\lambda/\mathfrak{S}_\lambda\isom
	S^{(a_{1})}\times \cdots \times S^{(a_{n})}$$  and we
	define the incidence correspondence $$\Gamma^\lambda := (S^{[n]}
	\times_{S^{(n)}} S^\lambda)_{\mathrm{red}} \subset S^{[n]} \times S^\lambda.$$
	The correspondence $\Gamma^{(\lambda)} \subset S^{[n]}\times S^{(\lambda)}$ is
	then the quotient $\Gamma^{\lambda}/\mathfrak{S}_\lambda$. Similarly, the
	correspondence $\Gamma^{(\lambda, j)}_{1}\subset S^{[n,n+1]}\times
	S^{(\lambda)}\times S$ is defined to be the incidence subvariety
	$$\Gamma^{(\lambda, j)}_{1}:=\left\{(\xi\subset\xi',z, x)~\vert~(\xi, z)\in
	\Gamma^{(\lambda)}\,; x=\xi'/\xi \text{ has multiplicity} \geq j \text{ in }
	\xi\right\}.$$ For an integer $a\geq 0$, the motive of the quotient $S^{(a)}$ is
	thought of as the direct summand of the motive of $S^a$ with respect to the
	idempotent $\frac{1}{a!}\sum_{\sigma \in \mathfrak{S}_a} \sigma$. When $S=A$ is
	an abelian surface, this idempotent is symmetrically distinguished, while in the
	case when $S$ is a smooth projective surface satisfying $(\star)$ it is also
	distinguished (see Remark \ref{R:perm}).
	In the case $S=A$ an abelian surface, taking the fiber over $0$ of the sum map
	$A^{n} \to A$ and of the sum map
	composed with the Hilbert--Chow morphism $A^{[n]} \to A^{(n)} \to A$, we
	define likewise $A_0^\lambda, A_0^{(\lambda)},  \Gamma_0^\lambda$, and
	$\Gamma_0^{(\lambda)}$.
	
	Then the strata associated to the semi-small resolutions
	$$S^{[n]} \to S^{(n)},   \quad K_{n-1}(A) \to A^{(n)}_0 \quad \text{and} \quad
	S^{[n, n+1]}\to S^{(n)}\times S$$ 
	are indexed by the set $\mathscr{P}(n)$ of partitions of $n$ in the first two
	cases and $\coprod_{\lambda\in \mathscr{P}(n)}I_{\lambda}$ with
	$I_{\lambda}=\{0\}\coprod \{j~|~a_{j}\neq 0\}$ in the last case\,; and we have
	morphisms (in fact, isomorphisms by Theorem \ref{T:dCM} below) of Chow motives 
	\begin{equation}\label{eq:hilb}
	\Gamma := \bigoplus_{\lambda\in\mathscr{P}(n)} \Gamma^{(\lambda)} \ :\
	\mathfrak{h}(S^{[n]}) \stackrel{}{\longrightarrow}
	\bigoplus_{\lambda\in\mathscr{P}(n)}\mathfrak{h}(S^{(\lambda)})(|\lambda|-n)
	\end{equation}
	
	\begin{equation}\label{eq:gk}
	\Gamma_0 := \bigoplus_{\lambda\in\mathscr{P}(n)} \Gamma^{(\lambda)}_0\ : \
	\mathfrak{h}(K_{n-1}(A)) \stackrel{}{\longrightarrow}
	\bigoplus_{\lambda\in\mathscr{P}(n)}\mathfrak{h}(A^{(\lambda)}_{0})(|\lambda|-n).
	\end{equation}
	
	\begin{equation}\label{eq:nestedhilb}
	\Gamma_{1} := \bigoplus_{\lambda\in\mathscr{P}(n)}\bigoplus_{j\in I_{\lambda}}
	\Gamma_{1}^{(\lambda, j)} \ :\
	\mathfrak{h}(S^{[n, n+1]}) \stackrel{}{\longrightarrow}
	\bigoplus_{\lambda\in\mathscr{P}(n)}\bigoplus_{j\in
		I_{\lambda}}\mathfrak{h}(S^{(\lambda)}\times S)(|\lambda|-n-\delta_{0,j}),
	\end{equation}
	where $\delta_{0,j}$ is 0 if $j=0$ and is 1 if $j\neq 0$.
	
	\begin{thm}[de Cataldo and Migliorini] \label{T:dCM}
		The morphisms of Chow motives $\Gamma$, $\Gamma_0$ and $\Gamma_{1}$ are
		isomorphisms with
		inverses given respectively by 
		$$\Gamma' := \sum_{\lambda\in\mathscr{P}(n)} \frac{1}{m_\lambda} {}^t
		\Gamma^{(\lambda)},  \quad \Gamma'_0 :=
		\sum_{\lambda\in\mathscr{P}(n)} \frac{1}{m_\lambda} {}^t \Gamma_0^{(\lambda)}
		\quad \text{and} \quad \Gamma_{1}' := \sum_{\lambda\in\mathscr{P}(n)} \sum_{j\in
			I_{\lambda}}\frac{1}{m_{\lambda,j}} {}^t
		\Gamma_{1}^{(\lambda, j)}$$
		where the superscript `$t$' indicates transposition, and where $m_\lambda := 
		(-1)^{n-|\lambda|} \prod_{i=1}^{|\lambda|} \lambda_i$ and $m_{\lambda,
			j}:=(-1)^{n-|\lambda|} a_{j}\prod_{i=1}^{|\lambda|} \lambda_i$ are non-zero
		constants, where $a_{j}=1$ if $j=0$ and $a_{j}=\#\{i :  1\leq i\leq
		n\,;\lambda_{i}=j\}$ if $j\neq 0$.
	\end{thm}
	\begin{proof}
		The proof that the morphism \eqref{eq:hilb} is an isomorphism with inverse
		given by
		$\Gamma'$ can be found in \cite{MR1919155} (or \cite{cm}),  the proof that  the
		morphism
		\eqref{eq:gk}  is an isomorphism with inverse given by $\Gamma'_0$ can be found
		in
		\cite[Corollary 6.3]{MHRCKummer} and the proof that the morphism
		\eqref{eq:nestedhilb}  is an isomorphism in \cite[Theorem 3.3.1]{cm} while the
		fact that its inverse is given by $\Gamma'_1$ follows from the proof of
		\cite[Theorem 2.3.8]{cm}.
	\end{proof}

	\subsubsection{Proof of Propositions \ref{P:hilb} and \ref{P:genK}}
	\label{S:proof}
	The argument is based on Voisin's universally defined cycle  theorem on
	self-products of surfaces \cite[Theorem 5.12]{VoisinDiag}. Let us write $X$ for
	either (i) the Hilbert scheme of length-$n$ subschemes on a surface $S$
	satisfying $(\star)$ (Proposition \ref{P:hilb}), (ii) the $n$-th nested Hilbert
	scheme of a surface $S$ satisfying $(\star)$ (Proposition \ref{P:hilb}), or
	(iii) a generalized Kummer variety $K_n(A)$ (Proposition \ref{P:genK}).
	We are going to show that the markings given by \eqref{eq:hilb}, \eqref{eq:gk}
	and \eqref{eq:nestedhilb} satisfy $(\star)$. For that purpose, we have to show
	that the class of the small diagonal $\delta_X$ (resp. the Chern classes of $X$)
	are mapped  in cases (i) and (ii) to a distinguished cycle on self-products of
	$S$ under the correspondences $\Gamma \otimes
	\Gamma\otimes \Gamma$ and $\Gamma_{1} \otimes
	\Gamma_{1}\otimes \Gamma_{1}$ (resp. $\Gamma$ and $\Gamma_{1}$), where $\Gamma$
	(resp. $\Gamma_{1}$) is the isomorphism \eqref{eq:hilb} (resp.
	\eqref{eq:nestedhilb}), and in case (iii)  to a symmetrically distinguished
	cycle  on an a.t.t.s. under the correspondence   $\Gamma_0 \otimes
	\Gamma_0\otimes \Gamma_0$ (resp. $\Gamma_0$), where $\Gamma_0$ is the
	isomorphism \eqref{eq:gk}.

	In cases (i) and (ii), one argues as in \cite[\S 3.2]{VialHilb} or as in
	\cite[Proposition 5.7]{MHRCKummer}. 
	The main idea is that, thanks to Voisin's theorem \cite[Theorem
	5.12]{VoisinDiag}, $\Gamma_* c_i(X)$ and  $(\Gamma \otimes \Gamma \otimes
	\Gamma)_*\delta_X$ (resp. $\Gamma_{1,*} c_i(X)$ and  $(\Gamma_1 \otimes \Gamma_1
	\otimes \Gamma_1)_*\delta_X$) are cycles that are polynomials in pull-backs
	along projections of Chern classes of $S$ and the diagonal $\Delta_S$. Since $S$
	is assumed to satisfy $(\star)$, diagonals and Chern classes are distinguished,
	and hence the above cycles are all distinguished.

	In case (iii), this is achieved for the small diagonal by arguing as in the
	proof of
	\cite[Proposition~6.12]{MHRCKummer} and for the Chern classes as in the proof of
	\cite[Proposition~7.13]{MHRCKummer}.  A key point to establish  $(\starM)$ is
	that the small diagonal $\delta_{K_n(A)}$ is the restriction of the small
	diagonal $\delta_{A^{[n+1]}}$ under the 3-fold product of the inclusion
	$K_{n-1}(A) \to A^{[n]}$. The proof of $(\starC)$ is similar once one has
	observed that 
	%the diagonal $\Delta_{K_{n-1}(A)}$ is the restriction of the diagonal
	%$\Delta_{A^{[n]}}$ and that 
	the Chern classes $c_i(K_{n-1}(A))$ are the restrictions of the Chern classes
	$c_i({A^{[n]}})$. One cannot invoke Voisin's theorem directly here, and one has
	to utilize the commutativity of the following diagram, whose squares are all
	cartesian and without excess intersections,
	\begin{equation*}\label{eqn:diagram}
	\xymatrix{
		(A^{[n]})^{3} \cart& \Gamma^{\lambda}\times \Gamma^{\mu}\times \Gamma^{\nu}
		\cart\ar[l]_-{p''}\ar[r]^-{q''} & A^{\lambda}\times A^{\mu}\times A^{\nu}\\
		(A^{[n]})^{3/A}\cart \ar@{^{(}->}[u]& \Gamma^{\lambda}\times_{A}
		\Gamma^{\mu}\times_{A} \Gamma^{\nu} \cart\ar[l]_-{p'}\ar[r]^-{q'}
		\ar@{^{(}->}[u] &
		A^{\lambda}\times_{A}A^{\mu}\times_{A}A^{\nu}\ar@{^{(}->}[u]_{j}\\
		K_{n-1}(A)^{3} \ar@{^{(}->}[u] & \Gamma_0^{\lambda}\times \Gamma_0^{\mu}\times
		\Gamma_0^{\nu} \ar@{^{(}->}[u] \ar[l]_-{p}\ar[r]^-{q} & A_0^{\lambda}\times
		A_0^{\mu}\times A_0^{\nu} \ar@{^{(}->}[u]_{i}
		%& & V^{\lambda}_{\tau_{1}}\times V^{\mu}_{\tau_{2}}\times
		%V^{\nu}_{\tau_{3}}\ar@{^{(}->}[u] \ar[ul]_-{p}\ar[r]^-{q} &
		%M^{\lambda}_{\tau_{1}}\times M^{\mu}_{\tau_{2}}\times M^{\nu}_{\tau_{3}}
		%\ar@{^{(}->}[u]
	}
	\end{equation*}
	Here $\lambda, \mu, \nu$ are partitions of $n$\,; all fiber products in the
	second row are over $A$\,; the second row is the base change by the inclusion of
	small diagonal $A\inj A^{3}$ of the first row\,; the third row is the base change by $\{O_{A}\}\inj A$ of the second row.
	
	We need to show that $ (\Gamma_0^{\lambda}\times \Gamma_0^{\mu}\times
	\Gamma_0^{\nu})_*(\delta_{K_{n-1}(A)}) = q_*p^*(\delta_{K_{n-1}(A)})$ is
	symmetrically distinguished on the a.t.t.s. $A_0^{\lambda}\times A_0^{\mu}\times
	A_0^{\nu}$ for all partitions $\lambda, \mu, \nu$ of $n$.
	
	As in the proof of \cite[Proposition 6.12]{MHRCKummer}, we have thanks to
	\cite[Lemma 6.6]{MHRCKummer} that 
	$A^{\lambda}\times_{A}A^{\mu}\times_{A}A^{\nu}$ and $A_0^{\lambda}\times
	A_0^{\mu}\times A_0^{\nu} $ are naturally disjoint unions of a.t.t.s. and the
	inclusions $i$ and $j$ are morphisms of a.t.t.s. on each component in the sense of Definition \ref{def:atts}.
	
	Denote $\delta_{A^{[n]}/A}$ the small diagonal inside the relative fiber product
	$(A^{[n]})^{3/A}$.
	Now by functorialities and the base change formula (\emph{cf.} \cite[Theorem
	6.2]{MR1644323}), we have
	$$j_{*}\circ q'_{*}\circ p'^{*} (\delta_{A^{[n]}/A})=q''_{*}\circ
	p''^{*}(\delta_{A^{[n]}}),$$
	which is a polynomial of big diagonals of $A^{|\lambda|+|\mu|+|\nu|}$ by
	Voisin's result \cite[Proposition 5.6]{VoisinDiag}, thus symmetrically
	distinguished in particular. By \cite[Lemma 6.10]{MHRCKummer}, $q'_{*}\circ
	p'^{*}(\delta_{A^{[n]}/A})$ is symmetrically distinguished on each component of
	$A^{\lambda}\times_{A}A^{\mu}\times_{A}A^{\nu}$.
	Again by functorialities and the base change formula,  we have 
	$$q_{*}\circ p^{*}(\delta_{K_{n-1}(A)})=i^{*}\circ q'_{*}\circ
	p'^{*}(\delta_{A^{[n]}/A}).$$
	Since $i$ is a morphism of a.t.t.s on each component, one concludes that
	$q_{*}\circ p^{*}(\delta_{K_{n-1}(A)})$ is symmetrically distinguished on each
	component, which concludes the proof. \qed
	
\begin{rmk}[Self-dual multiplicative Chow--K\"unneth decomposition for nested Hilbert schemes] The 
arguments of the proof of Proposition \ref{P:hilb} can be used to show that if a smooth projective surface $S$ has a self-dual multiplicative Chow--K\"unneth decomposition (see Section~\ref{S:mck} for the definition), then so do the nested Hilbert schemes $S^{[n,n+1]}$. Thus one may add the operation of taking nested Hilbert schemes of surfaces to \cite[Theorem~2]{SV2}.
\end{rmk}

	%\subsection{\comm{Fano variety of lines on the Fermat cubic fourfold}}
	
	\section{Link with multiplicative Chow--K\"unneth decompositions}\label{S:mck}
	
	%	\begin{rmk}[Analogy with \emph{self-dual multiplicative Chow--K\"unneth
	%decompositions}] 
	A \emph{Chow--K\"unneth decomposition} on a smooth projective variety $X$ of
	dimension $d$ is a set  $\{\pi^i_X : 0\leq i\leq 2d\}$ of mutually orthogonal
	idempotent correspondences in $X\times X$ that add up to $\Delta_X$ and whose
	cohomology classes in $H^{2d}(X\times X)$ are the components of the diagonal in
	$H^{2d-i}(X)\otimes H^i(X)$ for the K\"unneth decomposition.
	The notion of Chow--K\"unneth decomposition was introduced by Murre, who
	conjectured that all smooth projective varieties should admit such a
	decomposition \cite{Murre}. Murre's conjecture is intimately linked to the
	conjectures of Beilinson and Bloch\,; \emph{cf.} \cite{MR2115000}.
	
	The notion of \emph{multiplicative} Chow--K\"unneth (MCK) decomposition was
	introduced in \cite{SV} and further studied in \cite{MHRCKummer}, \cite{SV2},
	\cite{VialHilb} and \cite{MHRCK3}.
	A Chow--K\"unneth decomposition $\{\pi^i_X : 0\leq i\leq 2d\}$ on a smooth
	projective variety $X$ of dimension $d$ induces a bigrading decomposition of the
	Chow groups of self-powers of $X$ \emph{via} the formula 
	\begin{equation}\label{E:grading}
	\CH^i(X^n)_{(j)} := (\pi_{X^n}^{2i-j} )_* \CH^i(X^n),
	\end{equation}
	where by definition $X^n$ is endowed with the product Chow--K\"unneth
	decomposition
	$$\pi_{X^n}^{k} := \sum_{k_1+\cdots + k_n = k} \pi_X^{k_1}\otimes\cdots
	\otimes \pi_X^{k_n}. $$
	A Chow--K\"unneth decomposition $\{\pi^i_X : 0\leq i\leq 2d\}$ is
	\emph{multiplicative} if
	$\delta_X$ belongs to $\CH^{2d}(X\times X\times X)_{(0)}$. 
	As pointed out by the referee, this multiplicative condition 
	implies\footnote{Indeed, if $a$ is the structural morphism of $X$, we have $a_*\circ {\pi^{2d}_X}_* = a_*$, so that projecting $\delta_X = (\pi_{X^3}^{4d})_*\delta_X$ onto $X\times X$ gives $\Delta_X= (\pi^{2d}_{X^2})_*\Delta_X$. From the latter, it follows that $\pi_X^i = (\Delta_X\otimes \pi_X^i)_*\Delta_X = (\pi^{2d-i}_X\otimes \pi_X^i)_*\Delta_X =  (\pi^{2d-i}_X\otimes\Delta_X)_*\Delta_X = {}^t\pi^{2d-i}_X$\,; and conversely from $\pi_X^i = {}^t\pi^{2d-i}_X$ for all $i$, it follows that $\Delta_X = \sum_i\pi^i_X \circ \pi_X^i = \sum_i  ({}^t\pi^{i}_X\otimes \pi_X^i)_*\Delta_X = \sum_i  (\pi^{2d-i}_X\otimes \pi_X^i)_*\Delta_X = (\pi^{2d}_{X^2})_*\Delta_X$.}
	 that the diagonal $\Delta_X$ belongs to $\CH^d(X\times X)_{(0)}$, or equivalently, that the Chow--K\"unneth
	decomposition $\{\pi^i_X : 0\leq i\leq 2d\}$ is 
	\emph{self-dual}, meaning that $\pi_X^i = {}^t\pi_X^{2d-i}$ for all $i$.
	(In particular, the above remark makes it possible to simplify some of the arguments of \cite[\S 3]{SV2}). The existence of a multiplicative Chow--K\"unneth decomposition for $X$ ensures that $\CH^*(X)_{(0)}$ defines a graded subalgebra of $\CH^*(X)$. Finally,
	a natural condition that appeared in \cite{SV2} is that the Chern classes of $X$
	belongs to $\CH^*(X)_{(0)}$. As is apparent from the above and from the previous sections, the
	theory for $\DCH^*$ is in every way similar to that of $\CH^*(-)_{(0)}$ (compare
	with \cite{SV2}). 
	
	According to Murre's conjecture (D), for any choice of a Chow--K\"unneth
	decomposition $\{\pi^i_X : 0\leq i\leq 2d\}$, we should have that the
	restriction of the projection morphism $\CH^*(X) \to \overline\CH^*(X)$ to
	$\CH^*(X)_{(0)}$ is an isomorphism\,; see \cite{Murre}. Thus conjecturally the
	existence of a self-dual multiplicative Chow--K\"unneth decomposition for $X$
	provides a splitting to the algebra homomorphism $\CH^*(X) \to
	\overline\CH^*(X)$, in the same that a marking that satisfies $(\star)$ does. 
	%	\end{rmk}
	
	\begin{prop}[$(\star)$ and MCK decomposition]\label{prop:multmarking}
		Let $X$ be a smooth projective variety with a marking $\phi$ that satisfies
		$(\starM)$. Then $X$ has a self-dual multiplicative Chow--K\"unneth
		decomposition with the property that $\DCH_{\phi^{\otimes n}}^*(X^n) \subseteq
		\CH^*(X^n)_{(0)}$.
		%		 If in addition the marking satisfies $(\starC)$, then the Chern classes of
		%$X$ belong to $\CH^*(X)_{(0)}$.
		Moreover, equality holds if Murre's conjecture $(D)$ in  \cite{Murre} is true.
	\end{prop}
	\begin{proof}
		The proof of Proposition \ref{P:products} shows that if $X$ and $Y$ are two
		smooth projective varieties endowed each with markings satisfying $(\starM)$,
		then the product marking on $X\times Y$ also satisfies $(\starM)$. Moreover, the
		graphs of the projection morphisms are distinguished for the product markings.
		Therefore, composition of distinguished correspondences are distinguished.
		
		Let $A$ be an abelian variety, and let $p \in \DCH(A\times A)$ be a
		symmetrically distinguished projector. The Deninger--Murre Chow--K\"unneth
		projectors $\pi_A^i$ in \cite{DM} of $A$ are symmetrically distinguished. Since
		the Chow--K\"unneth projectors are central modulo homological equivalence, we
		see that $p\circ \pi_A^i = \pi_A^i \circ p \in \CH^*(A\times A)$ and in
		particular that these provide distinguished Chow--K\"unneth projectors for
		$(A,p)$.
		
		It follows that, assuming $X$ has a marking $\phi$ that satisfies $(\starM)$,
		$X$ admits a distinguished Chow--K\"unneth decomposition. We conclude that $X$
		has a self-dual multiplicative Chow--K\"unneth decomposition by noting that
		since a K\"unneth decomposition is always self-dual and multiplicative, any
		distinguished Chow--K\"unneth decomposition is self-dual and multiplicative.
		
		Finally, the inclusion $\DCH_{\phi^{\otimes n}}^*(X^n) \subseteq
		\CH^*(X^n)_{(0)}$ is due to the following three facts\,: the product
		Chow--K\"unneth decomposition $\{\pi^i_{X^n}\}$ is distinguished, the cycle
		$(\pi^i_{X^n})_*\alpha$ is homologically trivial (and hence numerically trivial)
		for all $\alpha \in \CH^j(X^n)$ and all $i\neq 2j$, and $(\pi^i_{X^n})_*\alpha$ 
		is distinguished if $\alpha$ is. Murre's conjecture (D) for $X^n$ stipulates
		that $\CH^i(X^n)_{(0)}$ should inject in cohomology \emph{via} the cycle class
		map, and in particular that the surjective quotient morphism $\CH^i(X^n) \to
		\bar\CH^*(X^n)$ is an isomorphism when restricted to $\CH^i(X^N)_{(0)}$. Since
		the quotient morphism is surjective when restricted to $\DCH_{\phi^{\otimes
				n}}^*(X^n)$, Murre's conjecture implies $\DCH_{\phi^{\otimes n}}^*(X^n) =
		\CH^*(X^n)_{(0)}$.
	\end{proof}

	\section{Varieties with motive of abelian type that do not satisfy
		$(\star)$}\label{S:counterex}
	The previous sections raise the question of determining a natural class of varieties which satisfy the condition $(\star)$ of Definition \ref{def:Star} or more weakly, the \emph{Section Property}. Beyond the case of hyper-K\"ahler varieties, which we expect to satisfy the Section Property, the answer is unfortunately not clear to us at this stage.
	To give some hint, in this section, we provide some examples of varieties with motive of abelian type (\emph{i.e.}, in $\M^{ab}$) which fail to satisfy $(\star)$ and/or the Section Property.
	
	\subsection{The Ceresa cycle and the condition $(\star)$} 
	Let $C$ be a smooth projective curve. In this section we give a necessary
	condition on the Ceresa cycle of $C$ for $C$ to admit a marking that satisfies
	$(\star)$. In fact, we give a necessary condition on the Ceresa cycle of $C$ for
	$C$ to admit a self-dual multiplicative Chow--K\"unneth decomposition\,; see
	Proposition \ref{prop:multmarking}.
	
	% Recall that the Ceresa cycle is the 1-cycle on the Jacobian $J(C)$ of $C$
	%defined as follows.
	Fix a zero-cycle $\alpha$ of degree~1 on $C$, and denote $\iota : C\to J(C)$
	the Abel--Jacobi map which maps a point $c\in C$ to the divisor class $[c] -
	\alpha$. We denote $[C]$ the class of the image of $C$ under $\iota$. Denote
	$[k] : J(C) \to J(C)$ the multiplication-by-$k$ homomorphism. The \emph{Ceresa
		cycle} is then the one-cycle $[C] - [-1]_*[C]$\,; it is numerically trivial, and
	its class modulo algebraic equivalence does not depend on the choice of the
	degree 1 zero-cycle $\alpha$.

	\begin{prop} \label{prop:ceresa}
		Let $C$ be a smooth projective curve. If $C$ has a self-dual multiplicative
		Chow--K\"unneth decomposition, then the Ceresa cycle is algebraically trivial.
	\end{prop}
	\begin{proof}
		Since a smooth projective curve has finite-dimensional motive in the sense of
		Kimura \cite{MR2107443}, any idempotent that is homologically equivalent to the
		K\"unneth projector on $H^0(C)$ is rationally equivalent to $\alpha\times C$ for
		some zero-cycle $\alpha$ of degree~1. Thus if $C$ has a  self-dual
		multiplicative Chow--K\"unneth decomposition, it must be of the form $\pi_C^0 :=
		\alpha \times C$, $\pi_C^2 := C\times \alpha$, $ \pi_C^1 := \Delta_C - \pi_C^0 -
		\pi_C^2$  for some zero-cycle $\alpha$ of degree~1. According to
		\cite[Proposition~8.14]{SV} this decomposition is multiplicative if and only if
		the modified diagonal cycle
		$$\mathfrak{z}:= \delta_{C} - \{(x,x,\alpha)\}- \{(x,\alpha,x)\} -
		\{(\alpha,x,x)\} + \{(x,\alpha,\alpha)\} + \{(\alpha,x,\alpha)\} +
		\{(\alpha,\alpha,x)\}$$ is zero in $\CH_1(C\times C\times C)$. Now we argue as
		in the proof of \cite[Proposition~3.2]{BVK3}. Let $\iota : C \to J(C)$ be the
		Abel--Jacobi map which maps a point $c\in C$ to the divisor class $[c] -
		\alpha$, and let $\iota^3 : C^3 \to J(C)$ be the map deduced from $\iota$ by
		summation.  We have 
		$$(\iota^3)_*(\mathfrak{z}) = [3]_*[C] - 3[2]_*[C] + 3[C] = 0 \quad \text{in}\
		\CH_1(J(C)).$$
		According to the Beauville decomposition \cite{MR826463}, we have $$\CH_1(J(C))
		= \CH_1(J(C))_{(0)} \oplus \cdots \oplus \CH_1(J(C))_{(g-1)},$$ where $g$ is the
		dimension of $J(C)$, and where $[k]_*$ acts on $\CH_1(J(C))_{(s)}$ by
		multiplication by $k^{2+s}$. Since $3^{2+s} - 3\cdot 2^{2+s} + 3>0$ for $s>0$,
		we find that $[C]$ belongs to $\CH_1(J(C))_{(0)}$. In particular, taking $k=-1$,
		we see that $[C] - [-1]_*[C] = 0$ in $\CH_1(J(C))$, and hence that the Ceresa
		cycle is algebraically trivial.
	\end{proof}

	\subsection{A very general curve of genus $>2$ does not satisfy $(\star)$}
	%Suppose such a curve has a marking satisfying $(\star)$. Then $c_1(C)$ is
	%distinguished. Denote $\iota: C\to C\times C$ the diagonal embedding. The
	%0-cycle $c_1(C)\times c_1(C) - \deg(c_1(C))\iota_*c_1(C)$ is numerically
	%trivial and distinguished, but it is not rationally trivial \cite{GG}.
	%\medskip
	%
	%In fact, a very general curve $C$ of genus $>2$ does not even have a marking
	%satisfying $(\star1)$ and $(\starM)$. Indeed, if it did, then since the graph
	%of $\iota$ is nothing but the small diagonal $\delta_C \in \CH^2(C\times
	%C\times C)$ we would have that $c_1(C) = \delta_C^*\Delta_C$ is distinguished.
	
	Although motives of curves are of abelian type, they do not necessarily have a
	marking that satisfies $(\star)$\,:

	\begin{prop} \label{prop:gencurve}Let $C$ be a curve, and let $\alpha$ be a
		degree $1$ zero-cycle on $C$. 
		If $C$ is very general  of genus $>2$, then the self-dual Chow--K\"unneth
		decomposition $\pi_C^0 := \alpha \times C$, $\pi_C^2 := C\times \alpha$, $
		\pi_C^1 := \Delta_C - \pi_C^0 - \pi_C^2$ is not multiplicative, and $C$ does not
		satisfy $(\star)$.
	\end{prop}
	\begin{proof}
		Ceresa \cite{Ceresa} proves that the Ceresa cycle of a very general curve of
		genus $>2$ is not algebraically trivial. The proposition follows then from
		Proposition \ref{prop:ceresa} (together with Proposition
		\ref{prop:multmarking}).
	\end{proof}
	
	\begin{rmk}
		This example involving the Ceresa cycle is mentioned in \cite[\S
		6.3]{MR2795752}.
	\end{rmk}

	\subsection{The Fermat quartic curve does not satisfy $(\star)$}

	\begin{prop}\label{prop:Fermatnot}
		Let $C$ be a Fermat curve of degree $d$ with $d\geq 4$, and let $\alpha$ be a
		zero-cycle of degree one on $C$. 
		If $d\leq 1000$, then the self-dual Chow--K\"unneth decomposition $\pi_C^0 :=
		\alpha \times C$, $\pi_C^2 := C\times \alpha$, $ \pi_C^1 := \Delta_C - \pi_C^0 -
		\pi_C^2$ is not multiplicative, and $C$ does not satisfy $(\star)$.
	\end{prop}
	\begin{proof}
		B.~Harris \cite{HarrisFermat} and S.~Bloch \cite{BlochL} prove that the
		Ceresa cycle of quartic Fermat curves is algebraically non-trivial, and Otsubo
		\cite{Otsubo} proves that the Ceresa cycle of Fermat curves of degree $4\leq d
		\leq 1000$ is not algebraically trivial. We can now apply Proposition
		\ref{prop:ceresa} (together with Proposition \ref{prop:multmarking}).
	\end{proof}

	\subsection{Varieties with motive of abelian type that do not admit a section}
	By considering a K3 surface of Picard rank $\geq 19$, the following proposition
	provides a simple example of a variety $X$ whose motive is of abelian type but
	for which the $\Q$-algebra epimorphism $\CH(X) \twoheadrightarrow
	\overline{\CH}(X)$ does not admit a section. In particular, by Proposition
	\ref{prop:Disting}, such a variety $X$ does not satisfy $(\star)$.
	
	\begin{prop}
		Let $S$ be a complex K3 surface and $P$ be a point of $S$ not representing
		the Beauville--Voisin zero-cycle. Denote $\tilde S$ the
		blow-up of $S$ along $P$. Then the $\Q$-algebra epimorphism $\CH(\tilde S)
		\twoheadrightarrow \overline{\CH}(\tilde S)$ does not admit a section.
	\end{prop}
	\begin{proof}
		The theorem of Beauville--Voisin \cite{BVK3} asserts that $\im(\CH^1(S)\otimes
		\CH^1(S) \to \CH^2(S))$ has rank one and is spanned by the class of any point
		lying on a rational curve on $S$. Such a class is called the Beauville--Voisin
		zero-cycle. Since $\dim_\Q \CH^2(S) = \infty$, there exists a point $P$ on $S$
		whose class is not rationally equivalent to the Beauville--Voisin zero-cycle. It
		is then straightforward to check that $\im(\CH^1(\tilde S)\otimes \CH^1(\tilde
		S) \to \CH^2(\tilde S))$ has rank 2 and is spanned by the class of $P$ and the
		Beauville--Voisin zero-cycle. Since $\CH^1(\tilde S) \twoheadrightarrow
		\overline{\CH}^1(\tilde S)$ is an isomorphism, if $\CH(\tilde S)
		\twoheadrightarrow \overline{\CH}(\tilde S)$ had a section, then
		$\im(\CH^1(\tilde S)\otimes \CH^1(\tilde S) \to \CH^2(\tilde S))$  would have
		rank 1 (equal to $\rk \overline{\CH}^2(\tilde S)$). This is a contradiction.
	\end{proof}

	\bibliographystyle{amsplain}
\providecommand{\bysame}{\leavevmode\hbox to3em{\hrulefill}\thinspace}
\providecommand{\MR}{\relax\ifhmode\unskip\space\fi MR }
\providecommand{\MRhref}[2]{%
	\href{http://www.ams.org/mathscinet-getitem?mr=#1}{#2}
}
\providecommand{\href}[2]{#2}

\end{document}